\documentclass[12pt,a4paper]{amsart}
\usepackage{pb-diagram}
\calclayout

\makeatletter
\let\pbdggeometry\dggeometry
\newcommand{\gridcorr}[4]{
  \def\dggeometry{
            \pbdggeometry
            \dg@XGRID=#1
            \dg@YGRID=#2
            \divide\unitlength by#3
  }
  \dgARROWLENGTH=#4
}
\makeatother
\newcommand{\thrfrdiag}{\gridcorr{4}{3}{3}{2em}}

\let\le\undefined
\DeclareMathSymbol{\le}{\mathrel}{AMSa}{"36}         
\let\ge\undefined
\DeclareMathSymbol{\ge}{\mathrel}{AMSa}{"3E}         
\let\empty\undefined
\DeclareMathSymbol{\empty}{\mathord}{AMSb}{"3F}      
\DeclareMathSymbol{\birarrow}{\mathrel}{AMSa}{"13}   

\newcommand{\ds}{\dots}

\newcommand{\dsb}{\dotsb}
\newcommand{\dsc}{\dotsc}

\newcommand{\tm}{\times}
\newcommand{\pro}{\sqcap}
\newcommand{\cop}{\sqcup}
\newcommand{\bu}{\bullet}
\def\ot{\DOTSB\otimes}
\def\op{\DOTSB\oplus}
\def\rarrow{\DOTSB\longrightarrow}
\def\larrow{\DOTSB\longleftarrow}
\def\lrarrow{\DOTSB\.\relbar\joinrel\relbar\joinrel\rightarrow\.}
\def\llarrow{\DOTSB\.\leftarrow\joinrel\relbar\joinrel\relbar\.}

\newcommand{\sub}{\subset}
\newcommand{\maps}{\longmapsto}
\newcommand{\rar}{\medspace{\rightarrow}\medspace}
\newcommand{\ovrarrow}[1]{\overset{\scriptscriptstyle #1}\rarrow}
\newcommand{\ovrar}[1]{\medspace
      {\overset{\scriptscriptstyle #1}{\rightarrow}}\medspace}
\newcommand{\bop}{\bigoplus\nolimits}

\renewcommand{\d}{\partial}
\newcommand{\D}{{\mathcal D}}
\newcommand{\E}{{\mathcal E}}
\newcommand{\M}{{\mathcal M}}
\newcommand{\A}{{\mathcal A}}
\newcommand{\B}{{\mathcal B}}
\newcommand{\C}{{\mathcal C}}
\newcommand{\T}{{\mathcal T}}
\newcommand{\F}{{\mathcal F}}
\renewcommand{\H}{{\mathcal H}}
\newcommand{\G}{{\mathcal G}}
\newcommand{\K}{{\mathcal K}}
\newcommand{\J}{{\mathcal J}}
\newcommand{\I}{{\mathcal I}}
\renewcommand{\S}{{\mathcal S}}
\newcommand{\Z}{{\mathbb Z}}
\newcommand{\Q}{{\mathbb Q}}

\DeclareMathOperator{\coker}{coker}
\DeclareMathOperator{\im}{im}
\DeclareMathOperator{\Ker}{Ker}
\DeclareMathOperator{\Coker}{Coker}
\let\Im\undefined
\DeclareMathOperator{\Im}{Im}
\DeclareMathOperator{\Coim}{Coim}
\DeclareMathOperator{\Cone}{Cone}
\DeclareMathOperator{\Tor}{Tor}
\DeclareMathOperator{\Ext}{Ext}
\DeclareMathOperator{\Hom}{Hom}
\DeclareMathOperator{\Yon}{Yon}
\DeclareMathOperator{\Nilp}{Nilp}
\DeclareMathOperator{\qu}{qu}
\DeclareMathOperator{\chr}{char}
\DeclareMathOperator{\Gal}{Gal}
\DeclareMathOperator{\Ob}{Ob}
\DeclareMathOperator{\Spec}{Spec}
\DeclareMathOperator{\Aut}{Aut}

\newcommand{\oK}{\,\overline{\!K}}
\newcommand{\KM}{\mathrm{K}^{\mathrm{M\mskip-.2\thinmuskip\relax}}}
\newcommand{\Kpi}{\mathrm{K}(\pi,1)}
\newcommand{\Fun}{\operatorname{\text{${\mathcal F}
               \mskip -.6\thinmuskip\relax un$}}}
\newcommand{\Com}{\operatorname{\text{${\mathcal C}
               \mskip -.3\thinmuskip\relax om$}}}
\newcommand{\Ab}{{\mathcal A}\mskip .5\thinmuskip\relax b}
\newcommand{\Ac}{{\mathcal A}{c}}

\newcommand{\ad}{{\operatorname{ad}}}
\newcommand{\gr}{{\operatorname{gr}}}
\renewcommand{\ss}{{\operatorname{ss}}}
\newcommand{\sat}{{\operatorname{sat}}}
\newcommand{\opp}{{\operatorname{op}}}
\newcommand{\Id}{{\operatorname{Id}}}
\newcommand{\id}{{\operatorname{id}}}

\renewcommand{\:}{\colon}
\renewcommand{\.}{\mskip .5\thinmuskip\relax}
\renewcommand{\;}{,\medspace}
\renewcommand{\implies}{\thickspace\Longrightarrow\thickspace}
\newcommand{\nbk}{\nobreak}

\newcommand{\abk}{\allowbreak}
\newcommand{\+}{\nobreakdash-}

\theoremstyle{plain}
\newtheorem*{conj}{Conjecture}
\newtheorem*{thm}{Theorem}
\newtheorem*{thm1}{Theorem 1}
\newtheorem*{thm2}{Theorem 2}
\newtheorem*{cor}{Corollary}
\newtheorem*{cor1}{Corollary 1}
\newtheorem*{cor2}{Corollary 2}
\newtheorem*{cor3}{Corollary 3}
\newtheorem*{prop}{Proposition}
\newtheorem*{prop1}{Proposition 1}
\newtheorem*{prop2}{Proposition 2}
\newtheorem*{prop3}{Proposition 3}
\newtheorem*{prop4}{Proposition 4}
\newtheorem*{lem}{Lemma}
\newtheorem*{lem1}{Lemma 1}
\newtheorem*{lem2}{Lemma 2}
\newtheorem*{lem3}{Lemma 3}
\newtheorem*{lem4}{Lemma 4}
\theoremstyle{definition}
\newtheorem*{ex}{Example}
\newtheorem*{ex1}{Example 1}
\newtheorem*{ex2}{Example 2}
\newtheorem*{rem}{Remark}
\newtheorem*{rem1}{Remark 1}
\newtheorem*{rem2}{Remark 2}

\newcommand{\Section}[1]{\bigskip\section{#1}\medskip}
\numberwithin{equation}{section}

\begin{document}

\title{Mixed Artin--Tate Motives with Finite Coefficients}
\author{Leonid Positselski}

\address{Sector of Algebra and Number Theory, Institute for
Information Transmission Problems, Bolshoy Karetny per.~19 str.~1,
Moscow 127994, Russia}
\email{posic@mccme.ru}

\begin{abstract}
 The goal of this paper is to give an explicit description of
the triangulated categories of Tate and Artin--Tate motives with
finite coefficients $\Z/m$ over a field $K$ containing a primitive
$m$\+root of unity as the derived categories of exact categories
of filtered modules over the absolute Galois group of $K$ with
certain restrictions on the successive quotients.
 This description is conditional upon (and its validity is
equivalent to) certain Koszulity hypotheses about the Milnor
K-theory/Galois cohomology of~$K$. 
 This paper also purports to explain what it means for an arbitrary
nonnegatively graded ring to be Koszul.
 Tate motives with integral coefficients are discussed in
the ``Conclusions'' section.
\end{abstract}


\maketitle

\tableofcontents

\section*{Introduction}
\medskip

\setcounter{subsection}{-1}
\subsection{{}}
 In the paper~\cite{Beil} published in 1987, A.~Beilinson formulated
his famous conjectures on the properties of hypothetical categories
of mixed motivic sheaves over a scheme.
 In addition to the classical case of motives with rational
coefficients, some conjectures about the category of motives
with a finite coefficient ring~$\Z/m$ were proposed there.
 Equivalent conjectures describing motivic complexes with finite
coefficients were earlier formulated by S.~Lichtenbaum in~\cite{Lich}.

 In the subsequent works of V.~Voevodsky~\cite{Voev} and others,
triangulated categories $\D\M(K,k)$ of motives over $K$ were
constructed (assuming only the resolution of singularities) for any
field~$K$ and any coefficient ring $k=\Z$, $\Q$, or~$\Z/m$.
 Beilinson predicted existence of abelian categories of mixed motives;
the triangulated categories now known would then be equivalent to
the derived categories of those abelian categories.
 The problem of constructing such categories of mixed motives,
or finding them as subcategories of the triangulated categories
of motives, remains open.
 In this paper we study the cases of mixed Tate and Artin--Tate
motives over a field with finite coefficients and discuss the related
homological algebra formalism in general.

\subsection{{}}  \label{introd-t-structures}
 There are three essentially independent conditions one has to
verify in order to establish that a triangulated category $\D$ is
equivalent to the derived category of an abelian category.
 First one has to find an abelian subcategory $\A\sub\D$, or,
more technically speaking, define a \emph{t\+structure}~\cite{BBD}
on the category~$\D$.
 Secondly, one has to check that this t\+structure is ``of
the derived type''~\cite[Subsection~4.0]{BGSch}, i.~e., the natural
homomorphisms of the $\Ext$ groups $\Ext_\A^*(X,Y)\rarrow
\Hom_\D(X,Y[*])$ are isomorphisms for all $X$,~$Y\in\A$.
 Thirdly, one has to construct a triangulated functor $\D^b(\A)\rarrow
\D$ compatible with the embeddings of $\A$ into $\D^b(\A)$ and~$\D$.

 Very roughly, given a full subcategory $\A\sub\D$, the condition
that $\A$ is the heart of a t\+structure on $\D$ is a restriction
on the groups $\Hom_\D(X,Y[n])$ for $n\le 0$, while the condition
that the t\+structure is of the derived type is a restriction on
the groups $\Hom_\D(X,Y[n])$ for $n\ge 2$.
 The former condition amounts to the existence of \emph{canonical
filtrations} on the objects of $\D$ with subquotients in $\A$,
while the latter condition is equivalent to the existence of
\emph{silly filtrations}.
 As to the problem of existence of a triangulated functor from
$\D^b(\A)$ to $\D$, it can be viewed independently of the choice
of the subcategory $\A\sub\D$ as arising solely from shortcomings
of the notion of a triangulated category per se.

\subsection{{}}
 Let us discuss the canonical and silly filtrations in some more
detail.
 Suppose we are given a complex $\dsb\rarrow X^{i-1}\rarrow X^i
\rarrow\dsb$ with objects and morphisms from an abelian category~$\A$.
 Then there are two filtrations which one can define on
the complex~$X^\bu$: the \emph{canonical filtration}
by the subcomplexes $\dsb\rarrow X^{i-1}\rarrow\Ker d_i\rarrow 0$
and the \emph{silly filtration} by the subcomplexes
$0\rarrow X^i\rarrow X^{i+1}\rarrow\dsb$.
 The former filtration is increasing when $i$~increases, while
the latter one is decreasing; the successive quotients of the
former filtration are the cohomology objects $H^i(X^\bu)$, while
the successive quotients of the latter one are the terms~$X^i$ of
the complex~$X^\bu$.
 The canonical filtration owes its name to the fact that it
does not depend on the choice of a particular complex inside
a given quasi-isomorphism class and therefore is uniquely
(and functorially) defined on objects of the derived category~$\D(\A)$.
 The silly filtration on an object of the derived category
is not unique; nevertheless, it is sensible to ask about
its existence.

 A typical example when the canonical filtrations do exist,
but the silly filtrations may not is the following one.
 Let $\A$ be an abelian subcategory of an abelian category~$\B$;
consider the derived category $\D^b_\A(\B)$ of bounded complexes
in the category~$\B$ with cohomology in the subcategory~$\A$.
 Then one can see that the silly filtrations exist in $\D^b_\A(\B)$
if and only if the natural functor $\D^b(\A)\rarrow\D^b_\A(\B)$ is
an equivalence of categories, which means that the $\Ext$ groups
between the objects of~$\A$ computed in the larger category~$\B$
should coincide with those computed in~$\A$.
 On the other hand, the canonical filtrations may not exist
on complexes whose terms belong to an additive category
which is not abelian.

\subsection{{}}  \label{introd-silly-filtrations}
 The silly filtrations condition can be stated in several equivalent
forms as a condition on a full subcategory $\M$ closed under
extensions in a triangulated category~$\D$.
 One of these formulations is quite simply that any element of
the group $\Hom_\D(X,Y[n])$ with $X$, $Y\in\M$ and $n\ge 2$ is
the composition of such elements of degree $n=1$.
 This makes it possible to consider this condition independently of
the condition that $\M$ is the heart of a t\+structure.

 Given a t\+structure with the heart $\A$ on a triangulated category
$\D$ and assuming the existence of a certain
refinement~\cite{Beil2,Neem2} of the triangulated category structure
on $\D$ (which does exist, at least, for all triangulated categories
of algebraic origin), one can construct a triangulated functor
$\D^b(\A)\rarrow\D$.
 This functor is an equivalence of triangulated categories if and
only if the t\+structure is of the derived type, i.~e., the silly
filtrations, or $\Ext$ decomposition, condition holds.
 This answers, in theory, our question when a triangulated category
is equivalent to the derived category of an abelian category.

 Notice the difference between our silly filtrations and the ``weight
structures'' of Bondarko~\cite{Bond}: the orthogonality condition
from the definition in~\cite{Bond} is not assumed, and does not hold,
in the situations considered in this paper.
 Indeed, the question of existence of silly filtrations on objects of
a triangulated category with respect to a generating full subcategory
(closed under extensions) that we are interested in here becomes
trivial when the orthogonality, i.~e., the vanishing of
$\Hom_\D(X,Y[n])$ for $X$, $Y\in\M$ and $n\ge1$, is assumed.

\subsection{{}}
 Let us discuss the case of Tate motives with rational coefficients
first.
 One can define the category of mixed Tate motives as the minimal
subcategory of Voevodsky's triangulated category $\D\M=\D\M(K,\Q)$
containing the Tate objects~$\Q(i)$ and closed under extensions.
 Beilinson's conjectures claim---and Voevodsky for his triangulated
category of motives proves---that the $\Ext$ spaces between the Tate
objects are isomorphic to the appropriate pieces of
the algebraic K\+theory groups of the field~$K$.
 This means that the first problem one encounters is to prove
that the pieces of the algebraic K\+theory corresponding to
the spaces $\Hom_{\D\M}(\Q(i),\Q(j)[n])$ with negative numbers~$n$,
and actually also with $n=0$ and $i\ne j$, vanish---which
is the statement of the Beilinson--Soule \emph{vanishing
conjectures}.
 This is the necessary and sufficient condition for existence
of a t\+structure on the triangulated subcategory of $\D\M(K,\Q)$
generated by the Tate objects having the above-defined category
of mixed Tate motives as its heart.

 As to the silly filtrations, let us note that even if the abelian
category of arbitrary mixed motives does exist, the condition that
the $\Ext$ spaces between the Tate objects computed in this large
abelian category coincide with those with respect to the abelian
category of mixed Tate motives is still highly nontrivial.
 This is one of Beilinson's
conjectures~\cite[Subsection~5.10.D(iv)]{Beil}; it is supported by
the facts that one has $\Hom_{\D\M}(\Q(i),\Q(j)[n])=0$ for $n>j-i$,
and, most remarkably, that the diagonal $\Ext$ algebra
$\bop_n\Hom_{\D\M}(\Q,\Q(n)[n])$, being isomorphic to the Milnor
K\+theory algebra $\KM(K)\ot_\Z\Q$ of the field~$K$ with rational
coefficients, is a quadratic algebra.

\subsection{{}}
 Another name for the silly filtration conjecture is
the ``$\Kpi$-conjecture'' of Bloch and Kriz~\cite{BK}.
 This terminology comes from the following series of examples.

 Let $X$ be a connected CW\+complex and $\D$ be the derived category
of sheaves of abelian groups on~$X$.
 Denote by $\A\sub\D$ be the abelian subcategory of locally constant
sheaves.
 The subcategory $\A$ is the heart of a bounded t\+structure on
the full triangulated subcategory in $\D$ generated by~$\A$.
 This t\+structure is of the derived type if and only if $X$ is
a $\Kpi$.

 Alternatively, let $\D$ be the derived category of sheaves of
$\Q$\+vector spaces on~$X$ and $\A\sub\D$ be the abelian subcategory
consisting of all unipotent local systems of finite rank.
 The abelian category $\A$ is the heart of a bounded t\+structure
on the full triangulated subcategory in $\D$ generated by $\A$.
 This t\+structure is of the derived type if and only if
the rational homotopy type of $X$ is a $\Kpi$.

\subsection{{}}
 As we have already mentioned, the above homomorphisms $\Ext_\A^*(X,Y)
\rarrow \Hom_\D(X,Y[*])$ are isomorphisms whenever they are surjective,
that is, if and only if any element of the group $\Hom_\D(X,Y[n])$
with $X$,~$Y\in\A$ and $n\ge 2$ is the composition of such elements
of degree $n=1$.
 In the case of mixed Tate motives, however, we only know the higher
$\Hom$ spaces between the simple objects~$\Q(i)$.
 One can even show that it suffices to decompose any homomorphism
$X\rarrow Y[n]$ with a simple object~$X$, but the condition that
any higher homomorphism between two simple objects is decomposable
is \emph{not} sufficient.
 It is not always possible to express the silly filtrations condition
in terms of the algebra of higher homomorphisms between the simple
objects; when it is, it typically turns out to be
a \emph{Koszulity}~\cite{BGS,PV} condition.

 For mixed Tate motives with rational coefficients, such a result
can be obtained in the case of a field~$K$ of finite characteristic.
 A conjecture of Parshin and Beilinson claims that in this case
the algebraic and Milnor K-theory rings with rational coefficients
coincide; this means that $\Hom_{\D\M}(\Q(i),\Q(j)[n])=0$ unless
$n=j-i$.
 Then the $\Ext$ spaces computed in the abelian category of mixed
Tate motives are isomorphic to the higher $\Hom$ spaces in
the triangulated category if and only if the diagonal $\Hom$ algebra,
that is the Milnor K\+theory algebra $\KM(K)\ot_\Z\Q$, is Koszul.

\subsection{{}}
 Now let us turn to Tate motives with finite coefficients.
 As in the rational coefficients case, there is a spectral sequence
starting from the motivic cohomology groups
$\Hom_{\D\M}(\Z/m,\Z/m(i)[n])$, where $\D\M=\D\M(K,\.\Z/m)$, and
converging to the algebraic K\+theory of the field~$K$ with
coefficients~$\Z/m$.
 In particular, the diagonal cohomology algebra
$\bop_n\Hom_{\D\M}(\Z/m,\Z/m(n)[n])$ is isomorphic to the Milnor
K\+theory algebra $\KM(K)\ot_\Z\Z/m$.

 On the other hand, the Beilinson--Lichtenbaum conjectures
connect motivic cohomology with Galois cohomology of the field~$K$.
 The triangulated category of motives with finite coefficients
$\D\M(K,\.\Z/m)$ over a field~$K$ of characteristic prime to~$m$
comes together with the \'etale realization functor
 $$
  \D\M(K,\Z/m)\lrarrow\D(G_K,\Z/m)
 $$
where $G_K=\Gal(\oK/K)$ is the absolute Galois group of~$K$ and
$\D(G_K,\Z/m)$ is the derived category of discrete $G_K$\+modules
over $\Z/m$.
 This functor sends the Tate object $\Z/m(i)$ to the cyclotomic
$G_K$\+module $\mu_m^{\ot i}$.
 The conjectures claim that the \'etale realization functor
induces isomorphisms of the higher $\Hom$ spaces as follows:
\begin{equation}\label{BL-formulas}
 \begin{aligned}
    \Hom_{\D\M}(\Z/m(i),\.\Z/m(j)[n])&\simeq
                       H^n(G_K,\,\mu_{m}^{\ot j-i}),
      \quad \text{for \ $n\le j-i$;} \\
    \Hom_{\D\M}(\Z/m(i),\.\Z/m(j)[n])&=0,
      \quad \text{otherwise.}
 \end{aligned}
\end{equation} 

 In particular, comparing the two formulas for the diagonal
cohomology one comes to the Milnor--Bloch--Kato conjecture
connecting the Galois cohomology with cyclotomic coefficients
with the Milnor K\+theory.
 It was proven by Voevodsky and Suslin~\cite{VS} (see also~\cite{GL2})
that the converse implication holds as well: if the Milnor--Bloch--Kato
conjecture is true for any field containing the field~$K$, then
the Beilinson--Lichtenbaum conjecture for the field~$K$ follows.

\subsection{{}}
 In view of the Beilinson--Lichtenbaum formulas~\eqref{BL-formulas},
it is a natural problem to give a precise description of
the triangulated subcategory generated by the Tate objects in
$\D\M(K,\Z/m)$ in terms of the Galois group~$G_K$.
 The question of finding an elementary construction, in terms of
the Galois group, of an abelian category of mixed Tate motives
with finite coefficients with the $\Ext$ spaces given by
the Beilinson--Lichtenbaum formulas was posed
in~\cite[Subsection~5.10.D(vi)]{Beil}.

 As in the rational coefficients case, the minimal subcategory
of $\D\M(K,\Z/m)$ containing the Tate objects and closed under
extensions seems to be the natural candidate for the category
of mixed Tate motives.
 By the definition, any object of this category has a canonical
filtration whose successive quotients are direct sums of the
Tate objects, and this is the maximal subcategory with this property.
 It turns out, however, that this category is \emph{not} abelian.

 The reason is that it follows from the formulas~\eqref{BL-formulas}
that, unlike in the rational coefficients case, for certain $i>0$
there exists a nonzero morphism $\Z/m\rarrow\Z/m(i)$ corresponding to
an isomorphism of $G_K$\+modules $\Z/m\simeq\mu_m^{\ot i}$.
 It is easy to see that the kernel and cokernel of such a morphism
in the above category of mixed Tate motives are zero, and still
this morphism is not an isomorphism.
 In the simplest case of an algebraically closed field~$K$ and
a prime number~$m$, this category is equivalent to the category
of finite-dimensional filtered vector spaces over~$\Z/m$.

\subsection{{}}
 Instead of being abelian, the category of mixed Tate motives
with finite coefficients has a natural structure of an
\emph{exact category}.
 The construction of the derived category of an abelian category
can be generalized to exact categories smoothly, and, conversely,
an additive subcategory of a triangulated category has a natural
exact category structure if it is closed under extensions and
there are no $\Hom$ groups with negative degrees between
its objects.
 All the results about t\+structures mentioned
in~\ref{introd-t-structures} and~\ref{introd-silly-filtrations}
can be extended to exact subcategories of triangulated categories.

\subsection{{}}  \label{introd-list-of-situations}
 Depending on the conditions on the field $K$ and the coefficient
ring~$k$, it may or may not be possible to express all of
the motivic cohomology $\Hom_{\D\M}(k,k(i)[n])$, where $\D\M=
\D\M(K,k)$, in terms of the diagonal cohomology
$\Hom_{\D\M}(k,k(n)[n])\simeq \KM_n(K)\ot_\Z k$ in a simple way.
 Various conjectures and results predict that in the following
situations this should be possible:
\begin{enumerate}
 \renewcommand{\theenumi}{\roman{enumi}}
 \item $k=\Q$ and $\chr K = p > 0$;
 \item $k=\Z/p^r$ and $\chr K = p > 0$;
 \item $k=\Z/m$ and $K$ contains a primitive root of unity of
degree~$m$.
\end{enumerate}
 In all of these cases, the silly filtration conjecture for
the mixed Tate motives over $K$ with coefficients in $k$ is
equivalent to the Koszulity condition on the graded algebra
of diagonal motivic cohomology $\KM_n(K)\ot_\Z k$.

 On the other hand, in the following cases there is apparently no
simple way to express the motivic cohomology in terms of its
diagonal part:
\begin{enumerate}
 \renewcommand{\theenumi}{\roman{enumi}}
 \setcounter{enumi}{3}
 \item $k=\Z/l$ for a prime number~$l$, the field $K$ does
not contain a primitive root of unity of degree~$l$, and
$\chr K \ne l$;
 \item $k=\Q$ and $\chr K = 0$.
\end{enumerate}
 In these cases, we do not know how to express the silly filtration
conjecture in terms of the graded algebra $\KM_n(K)\ot_\Z k$, nor
do we know whether one should expect the Milnor algebra
$\KM_n(K)\ot_\Z k$ to be Koszul for any motivic reasons.

\subsection{{}}
 The main results of this paper are as follows.

 Assuming the Beilinson--Lichtenbaum conjecture, we prove that
the exact category of mixed Tate motives with coefficients in~$\Z/m$
is equivalent to the category of filtered modules~$(N,F)$
over~$\Z/m$ with a discrete action of the Galois group~$G_K$ such that
each quotient module $F^iN/F^{i+1}N$ is isomorphic to a direct sum
of several copies of the cyclotomic module $\mu_m^{\ot i}$.
 Furthermore, we show that if the field~$K$ contains a primitive
$m$\+root of unity, then the $\Ext$ spaces computed in this exact
category are isomorphic to the higher $\Hom$ spaces in
the triangulated category of motives (i.~e., are given by
the Beilinson--Lichtenbaum formulas) if and only if the Galois
cohomology algebra $H^*(G_K,\Z/m)$ is Koszul.

 Recall that it was proven in the paper~\cite{PV} that
the Milnor--Bloch--Kato conjecture claiming the isomorphism between
the Galois cohomology and the Milnor K\+theory algebra modulo~$m$
follows from its low-degree part (an isomorphism in degree~$2$ and
a monomorphism in degree~$3$) provided that the Milnor algebra is
Koszul for all algebraic extensions of a given field.
 The result of this paper means that the Koszulity conjecture
together with the low-degree part of the Milnor--Bloch--Kato
conjecture is equivalent to the full Milnor--Bloch--Kato conjecture
together with the silly filtrations condition for the exact category
of mixed Tate motives.

\subsection{{}}
 As a generalization of mixed Tate motives, one can consider mixed
Artin--Tate motives over a field~$K$.
 These are the motives $k[L]$ of finite field extensions $L/K$,
their Tate twists $k[L](i)$, and any extensions of these.
 Let us choose a finite Galois extension $M/K$ and restrict ourselves
to the exact subcategory $\M$ of $\D\M(K,k)$ formed by all
the successive extensions of the objects $k[L](i)$ with
$K\sub L\sub M$.

 Our \emph{main conjecture} about Artin--Tate motives claims that
this exact subcategory satisfies the silly filtrations condition,
i.~e., the groups $\Ext$ in $\M$ are isomorphic to the corresponding
$\Hom$ groups in the triangulated category of motives $\D\M(K,k)$.
 We show that in the above situations~(i--iii) this conjecture is
equivalent to the Koszulity condition on the algebra of diagonal $\Hom$
\begin{equation}\label{AT-algebra}
 A=\textstyle\bop_n \Hom_{\D\M}(\bop_L\!\. k[L]\;\bop_L\!\. k[L](n)[n])
\end{equation}
between the Artin--Tate objects $k[L](i)$ with $K\sub L\sub M$.

 Moreover, for any fields $K\sub M$ of characteristic prime to $m$,
the exact category $\M$ of mixed Artin--Tate motives with coefficients
$k=\Z/m$ over $K$ is equivalent to the category $\F$ of filtered
discrete modules $G_K$\+modules $(N,F)$ over $\Z/m$ such that for any
$i\in\Z$ the $G_K$\+module $F^iN/F^{i+1}N$ is isomorphic to a finite
direct sum of $G_K$\+modules induced from the cyclotomic modules
$\mu_m^{\ot i}$ over open subgroups of $G_K$ containing~$G_M$.
 Assuming the main conjecture, this provides a description of
the triangulated subcategory of $\D\M(K,k)$ generated by $k[L](i)$
as the bounded derived category of the exact category $\F$ of
filtered $G_K$\+modules.

\subsection{{}}
 To end, let us say a few words about the Koszulity condition that
appears in connection with the Artin--Tate motives with finite
coefficients.
 The theory of Koszul algebras as developed in~\cite{PV,BGS}
applies to graded algebras $A=\bop_{n\ge0} A_n$ such that $A_0$
is a semi-simple algebra.
 It is not difficult to generalize it to nonnegatively graded
rings $A$ with an arbitrary base ring $A_0$, assuming that $A$
is a flat left or right graded $A_0$\+module~\cite{Psemi}.
 However, the graded algebra~\eqref{AT-algebra} we are interested in
does not satisfy the latter assumption.

 It turns out that the Koszul property, just as the property of
a nonnegatively graded ring to be quadratic, does not depend on
the base ring in the zero-degree component.
 For the graded algebra $A$ from~\eqref{AT-algebra}, one can simply
replace the component $A_0$ with $A'_0=\Z/m$ in order to define $A$
to be Koszul if the graded ring
$$
 A'=A'_0\oplus A_1\oplus A_2\oplus\dsb
$$
is Koszul in the sense of~\cite{PV} (if $m$ is a prime number) or
in the sense of~\cite{Psemi} (in the general case).
 Such is the contribution that this paper makes to the general
theory of Koszul rings.

\subsection{{}}
 After a more than a decade-long effort, V.~Voevodsky, in collaboration
with M.~Rost and others, have recently finished their work on
a complete proof of the Milnor--Bloch--Kato
conjecture~\cite{Voev2,Voev3}.
 Their approach is entirely different from the one suggested
in~\cite{PV}; instead, it builds upon the ideas and techniques
of A.~Merkurjev and A.~Suslin's proofs~\cite{MS1,MS2} in
degrees~$2$ and~$3$.
 A simplified, elementary exposition of the easy first step of
their argument can be found in the present author's
paper~\cite{Pdivis}.
 
 Our Koszulity conjectures remain wide open.

\subsection{{}}
 The fairly simple homological formalism describing
the situations~\ref{introd-list-of-situations}\,(i-ii) for Tate
motives with coefficients in a field~$k$ is presented in
Section~1.
 An almost equally simple
situation~\ref{introd-list-of-situations}\,(iii) for Tate motives
with coefficients in a finite field~$k$ is partly treated in
Section~2 using the filtered bar construction.

 We pass to the full generality starting from Section~3, where
we describe the exact subcategory of mixed objects in a triangulated
category with the $\Hom$ groups given by Beilinson--Lichtenbaum
formulas.
 The construction of the associated graded category to a filtered
exact category with a twist functor and related natural transformation
is presented in Section~4.
 The restriction-of-base construction underlying the claim that
the Koszul property does not depend on the base ring is introduced
in Section~5.
 These three sections form the technical heart of the paper.

 We prove that the diagonal cohomology ring is quadratic, and any
quadratic ring can be realized as the diagonal cohomology, in
Section~6.
 The Koszul property of (big) graded rings in the general and
the flat cases is studied in Section~7.
 We digress to apply our techniques in order to generalize
the results of~\cite{PV} and~\cite[Section~5]{Pbogom} in
Section~8.
 In particular, we obtain a Koszulity-based sufficient condition
for existence of silly filtrations with respect to an exact
subcategory of a triangulated category.
 The proofs of main results are finished and the conclusions
discussed in Section~9.
 In particular, the silly filtration conjecture for Artin-Tate
motives is formulated, and some remarks about Tate motives with
integral coefficients are included.

 The purpose of Appendix~A is to supply preliminary material on exact
categories complementary to the standard expositions.
 It covers big graded rings, saturatedness conditions, various axioms
and examples of exact categories, two canonical embeddings to abelian
categories, the derived categories of exact categories, the Yoneda
$\Ext$, and exact subcategories of triangulated categories.
 The formalism of silly filtrations is presented in Appendix~B,
and the $\Kpi$-conjecture of Bloch and Kriz is discussed in
Appendix~C\hbox{}.
 It is explained why these are two equivalent formulations of
the same conjecture; both the rational and the finite coefficients
are considered.
 The realization functor to a (filtered) triangulated category from
the derived category of its exact subcategory is constructed
in Appendix~D.

\subsection*{Acknowledgement}
 The author is grateful to V.~Voevodsky and A.~Beilinson for posing
the problem and for numerous very helpful conversations.
 I would like also to thank P.~Deligne, S.~Bloch, V.~Retakh,
A.~Goncharov, D.~Orlov, A.~Vishik, A.~Polishchuk, V.~Vologodsky,
and M.~Bondarko for very helpful discussions.

 This work was started when the author was visiting the Mathematics
Department of Harvard University on the invitation of D.~Kazhdan in
the Fall of 1994, and most of it has been done when I was a graduate
student at the same university in the subsequent years.
 I am glad to use this opportunity to thank Harvard for its
hospitality.
 These results were presented at the Fall 1999 conference in
Oberwolfah, and I want to thank the MFO and the organizers of
the conference for the invitation.

 The author was supported by a grant from P.~Deligne 2004 Balzan
prize and an RFBR grant when developing the later final ideas of
the paper and writing it up.

\Section{Toy Example~I: Strictly Exceptional Sequence}

\subsection{Existence of t\+structure}  \label{t-existence}
 Let $\D$ be a triangulated category generated by a sequence of
objects $E_i\in\D$, \ $i\in\Z$ satisfying the following conditions
\begin{equation} \label{exceptional-weak}
\begin{aligned}  
 &\Hom_\D(E_i,E_j[n]) = 0 
  \quad\text{for all $i>j$ and $n\in\Z$;}    \\
 &\Hom_\D(E_i,E_i) \quad\text{is a division ring for all $i\in\Z$.}
\end{aligned}
\end{equation}
 Let $\A$ be the minimal full subcategory of $\D$, containing
the objects $E_i$ and closed under extensions (see~\cite[1.2.6]{BBD}
or~\ref{exact-triangulated} for the definition).
 The following result is a generalization of a theorem of
M.~Levine~\cite{Lev} (see also~\cite{BGSch}) inspired by
R.~Bezrukavnikov's paper~\cite{Bez}.

\begin{thm}
 The subcategory $\A$ is the heart of a (necessarily bounded)
t\+structure on $\D$ if and only if\/ $\Hom_\D(E_i,E_j[n])=0$
for all\/ $i$, $j\in\Z$ and\/ $n<0$ and\/ $\Hom_\D(E_i,E_j)=0$
for all\/ $i\ne j$.
\end{thm}

\begin{proof}
 ``If'': let $\E$ denote the full subcategory of $\D$ consisting
of direct sums of objects $E_i$.
 Clearly, $\E$ is a semisimple abelian category and one has
$\Hom_\D(X,Y[n])=0$ for $n<0$ and $X$, $Y\in\E$.
 It follows that $\E$ is an admissible abelian subcategory of $\D$
in the sense of~\cite[Subsection~1.2]{BBD}, hence
by~\cite[Subsections~1.3.13--14]{BBD} the subcategory $\A$ is
the heart of a bounded t\+structure on~$\D$.
 The reader can find some additional details in~\cite[Lemma~3]{Bez}.

 ``Only if'': the condition of vanishing of $\Hom_\D(X,Y[n])$
for $n<0$ and $X$, $Y\in\A$ follows immediately from
the definitions of a t\+structure and its heart.
 Now let $f\:E_i\rarrow E_j$ be a morphism of degree~$0$.
 By the conditions~\eqref{exceptional-weak}, $f=0$ if $i>j$, so it
remains to consider the case $i<j$.

 Let $C$ denote the cone of the morphism~$f$.
 By~\cite[Th\'eor\`eme~1.3.6]{BBD}, there exists a distinguished
triangle $X[1]\rarrow C\rarrow Y\rarrow X[2]$ with $X$, $Y\in\A$.
 Let $\D_i$ denote the full triangulated subcategory of $\D$
generated by the object~$E_i$.
 It follows easily from~\eqref{exceptional-weak} that there is
a triangulated functor of ``successive quotients'' from $\D$
to the Cartesian product of the triangulated subcategories
$\D_i\sub\D$ sending $E_i\in\D$ to $(\dsc,0,E_i,0,\dsc)\in\prod_i\D_i$
(see~\cite[Sections~1 and~4]{BoKa} or~\cite[Lemma~1.3.2]{BGSch}).
 Applying this functor to the above triangle, one can see that $X$
is naturally isomorphic to $E_i$ and $Y$ is isomorphic to $E_j$
in~$\D$.
 The compositions $E_j\rarrow C\rarrow E_j$ and $E_i[1]\rarrow
C\rarrow E_i[1]$ are the identity morphisms, hence $f=0$.
\end{proof}

\begin{rem}
 One \emph{cannot} move the second condition 
in~\eqref{exceptional-weak} from the list of premises of Theorem
to the list of conditions in the right hand side of the equivalence.
 Indeed, there exists a triangulated category $\D$ generated by
a single object $E=E_0$ with $\Hom_\D(E,E[n])=0$ for $n\ne0$ such that
the full subcategory $\A\sub\D$ whose only objects are $E$ and $0$
is the heart of a bounded t\+structure on $\D$, and $\Hom_\D(E,E)$
is not a division ring.
 One simply takes $\A$ to be the quotient category of the category
of not more than countably dimensional vector spaces over a field~$k$
by the Serre subcategory of finite-dimensional vector spaces, and
$\D=\D^b(\A)$.
\end{rem}

\subsection{Diagonal cohomology and Koszulity}
 Let $k$ be a field and $\D$ be a $k$\+linear triangulated category
generated by a sequence of objects $E_i\in\D$, \ $i\in\Z$ satisfying
the conditions
\begin{equation}  \label{exceptional-strong}
\begin{alignedat}{3}  
 &\Hom_\D(E_i,E_j[n]) = 0  
  &&\quad\text{for all $i>j$ } &&\text{and $n\in\Z$;}    \\
 &\Hom_\D(E_i,E_i[n]) = 0  
  &&\quad\text{for all $i\in\Z$ } &&\text{and $n\ne 0$;} \\
 &\Hom_\D(E_i,E_i) = k     
  &&\quad\text{for all $i\in\Z$.}
\end{alignedat}
\end{equation}

 Furthermore, assume that $E_i\in\D$ satisty the equivalent conditions
of Theorem from~\ref{t-existence}.
 In addition, suppose that a triangulated autoequivalence of
the category $\D$, denoted by $X\maps X(1)$, is given together
with isomorphisms $E_i(1)\simeq E_{i+1}$.
 The functor $X\maps X(1)$ will be called the \emph{twist} functor
and its integral powers will be denoted by $X\maps X(j)$, \ $j\in\Z$.
 
 Since $\A$ is the heart of a t\+structure on $\D$, there are natural
maps $\Ext^n_\A(X,Y)\rarrow\Hom_\D(X,Y[n])$ for all $X$, $Y\in\A$
and $n\ge 0$, where $\Ext_\A$ denotes the Yoneda $\Ext$ in
the abelian category~$\A$ (see \ref{exact-triangulated}, or~\cite{Lev},
or~\cite[Subsection~4.0]{BGSch}).
 These maps are compatible with the multiplicative structure on
$\Ext_\A$ and $\Hom_\D$.
 Besides, they are always isomorphisms for $n=1$ and monomorphisms for
$n=2$ \cite[Remark~3.1.17]{BBD}.
 We will be interested in the question when all these maps are
isomorphisms.

 The results below in this section are essentially due to
A.~Beilinson, V.~Ginzburg, and V.~Schechtman~\cite{BGSch}.

 Recall that a graded algebra $A= A_0\oplus A_1\oplus A_2\oplus\dsb$
over a field $k$ with $A_0=k$ is called \emph{quadratic}~\cite{PV}
if it is multiplicatively generated by $A_1$ with relations
in degree~$2$.

\begin{prop} \ 
 \begin{enumerate}
 \renewcommand{\theenumi}{\arabic{enumi}}
 \item For the above abelian category $\A$, one has\/
$\Ext^n_\A(E_i,E_j)=0$ for $n>j-i$.
 \item The graded algebra $A$ with the components
$A_n=\Ext^n_\A(E_0,E_n)$ (and the multiplication defined in terms
of the Yoneda multiplication on the Ext using the twist functor 
$X\maps X(1)$ on the category $\A$) is quadratic.
 \item For any quadratic graded algebra $A$ over~$k$ there exists
a $k$\+linear abelian category $\A$ with a sequence of objects
$E_i\in\A$ and a twist functor $X\maps X(1)$ on $\A$ such that
$E_i(1)\simeq E_{i+1}$, all the objects of $\A$ can be obtained from
$E_i$ as iterated extensions, the derived category $\D=\D^b(\A)$ with
the objects $E_i\in\D$ satisfies
the conditions~\eqref{exceptional-strong} and the diagonal Ext
algebra $\bigoplus_n \Ext^n_\A(E_0,E_n)$ is isomorphic to
the graded algebra~$A$.
 \item Moreover, for any quadratic algebra $A$ there exists a unique,
up to a unique exact equivalence, preserving $E_i$ and the twist
functor, abelian category $\A$ in~\textup{(3)} such that\/
$\Ext^n_\A(E_i,E_j)=0$ for all\/ $n\ne j-i$, \ $n=1$ or~$2$.
\end{enumerate}
\end{prop}

\begin{proof}
 Under the conditions~\eqref{exceptional-strong}, the functor of
``successive quotients'' mentioned in the proof in~\ref{t-existence}
becomes simply a triangulated functor from $\D$ to the bounded
derived category of finite-dimensional graded $k$\+vector spaces
sending $E_i$ to the vector space $k$ placed in the cohomological
degree~$0$ and the grading~$i$.
 Restricting this functor to the category $\A$, one obtains
an exact faithful functor from $\A$ to the category of
finite-dimensional graded $k$\+vector spaces transforming
the twist to the shift of grading.
 According to~\cite[Proposition 2.14]{DM}, one can identify $\A$ with
the category of finite-dimensional graded left comodules over
a graded coalgebra~$C$ over~$k$ so that the object $E_i$ corresponds
to a one-dimensional comodule $k$ placed in degree~$i$.
 Besides, $\A$ is a mixed category in the sense
of~\cite[Subsection~2.1.2]{BMS}, the indices of the increasing
filtration in~\cite{BMS} being minus the indices~$i$ of
the objects~$E_i$.
 It follows that $C$ is nonpositively graded with $C_0=k$.
 An explicit construction of the coalgebra $C$ can be found
in~\cite[Subsection~2.1.7]{BMS}.

 Since any $C$\+comodule is a union of its finite-dimensional
subcomodules, the spaces $\Ext$ computed in the categories of
arbitrary graded $C$\+comodules and finite-dimensional graded
$C$\+comodules coincide.
 Hence the spaces $\Ext_\A^n(E_0,E_i)$ can be computed as
the bigrading pieces of the reduced cobar-complex~\cite{PV}
$$
 k\lrarrow C_+\lrarrow C_+\ot_k C_+\lrarrow C_+\ot_k C_+\ot_k C_+
 \lrarrow\dsb,
$$
where $C_+=\Ker(C\to k)$.
 This proves parts~(1) and~(2) of the Proposition
(see~\cite[beginning of Section~2 and Proposition~2]{PV}).
 To verify uniqueness in~(4), notice that in the assumptions of~(4)
the coalgebra $C$ is quadratic by~\cite[Proposition~1]{PV}, so
it is uniquely determined by its quadratic dual algebra $A$.

 To prove existence in~(3) and~(4), it suffices to set $C$ to be
the coalgebra quadratic dual to $A$ and $\A$ to be the category of
finite-dimensional graded $C$\+comodules.
\end{proof}

 Consequently, if the natural maps $\Ext^n_\A(X,Y)\rarrow
\Hom_\D(X,Y[n])$ are isomorphisms for all $X$, $Y\in\A$ and $n\ge0$,
then $\Hom_\D(E_i,E_j[n])=0$ for $n>j-i$ and the graded algebra $A$
with the components $A_n=\Hom_\D(E_0,E_n[n])$ is quadratic.
 Conversely, any quadratic algebra $A$ can be realized in this way.

\begin{thm}
 Assume that\/ $\Hom_\D(E_i,E_j[n])=0$ for $n\ne j-i$.
 Then the maps\/ $\Ext^n_\A(X,Y)\rarrow\Hom_\D(X,Y[n])$ are
isomorphisms for all\/ $X$, $Y\in\A$ and all\/ $n\ge0$ if and only if
the graded algebra $A$ is Koszul (see~\cite{PV} for the definition).
\end{thm}

\begin{proof}
 As it was explained in the proof of Proposition,
the spaces $\Ext^n_\A(E_0,E_i)$ are isomorphic to the bigraded
pieces of the cohomology of a nonpositively graded coalgebra~$C$.
 Now if $\Ext^n_\A(E_0,E_i)=0$ for $n\ne i$, then the coalgebra
$C$ is Koszul, so by~\cite[Proposition~3]{PV}, the algebra $A$
is Koszul, too.
 This proves the ``only if'' part.

 ``If'': since the maps $\Ext^n_\A(X,Y)\rarrow\Hom_\D(X,Y[n])$
are injective for $n\le 2$, one has $\Ext^n_\A(E_0,E_i)=0$ for
$n\ne i$ and $n=1$ or~$2$, so the coalgebra $C$ is quadratic.
 Since these maps are also bijective for $n=1$, the morphism of
graded algebras $\bigoplus_n \Ext^n_\A(E_0,E_n)\rarrow A$ is
an isomorphism in degree~$1$ and a monomorphisms in degree~$2$.
 If the algebra $A$ is quadratic, it follows from part~(2) of
Proposition that this morphism is an isomorphism.
 Consequently, the algebra $A$ is quadratic dual to the coalgebra
$C$, and if $A$ is Koszul, then $C$ is Koszul, too.
 Therefore, the morphisms $\Ext^n_A(E_i,E_j)\rarrow
\Hom_\D(E_i,E_j[n])$ are isomorphisms for all~$i$, $j$,
and~$n$.
 Since all the objects of $\A$ are successive extensions of
the objects $E_i$, we are done.
\end{proof}

\begin{rem}
 One can drop the condition that $\D$ be linear over a field
in the above Proposition and Theorem, replacing the condition
$\Hom_\D(E_i,E_i)=k$ with the condition that $\Hom_\D(E_i,E_i)$
be a division ring, as in~\ref{t-existence}.
 An even greater generality of arbitrary base ring is achieved
in Sections~\ref{diagonal-secn}--\ref{koszul-rings-secn}
by replacing abelian hearts with exact subcategories.
\end{rem}

\Section{Toy Example~II: Conilpotent Coalgebra}  \label{coalgebra-secn}

 Let $C$ be a coalgebra over a field~$k$ and $k\rarrow C$ be
a coaugmentation of $C$ (i.~e., a morphism of coalgebras).
 Recall~\cite[Subsection~3.1]{PV} that the \emph{coaugmentation
filtration} on $C$ is an increasing filtration defined by the rule
$$
 F_n C = \Ker(C\to C^{\ot n+1} \to (C/k)^{\ot n+1}),
$$
where $C\rarrow C^{\ot n}$ is the iterated comultiplication map
and the map $C^{\ot n+1}\rarrow (C/k)^{\ot n+1})$ is induced by
the cokernel $C\rarrow C/k$ of the coaugmentation morphism.
 The coalgebra $C$ is called \emph{conilpotent} if the filtration
$F$ is exhaustive, i.~e., $C=\bigcup_n F_nC$.
 For any coaugmented coalgebra $C$, the subcoalgebra $\Nilp C = 
\bigcup_n F_nC\subset C$ is the maximal conilpotent subcoalgebra
of~$C$.

 The coaugmentation morphism endows any vector space over~$k$ with
a structure of $C$\+comodule, called the \emph{trivial}
$C$\+comodule structure.
 A finite-dimensional $C$\+comodule is a comodule over $\Nilp C$
if and only if it is a successive extension of copies of the trivial
comodule $k$ over~$C$.

 Set $F^{-i}C=F_iC$.
 Consider the category $\F$ of finite-dimensional left $C$\+comodules
$N$ endowed with a decreasing filtration $F$ compatible with
the filtration $F$ on~$C$.
 This is equivalent to the filtration $F$ on $N$ being a filtration
by $C$\+subcomodules such that all the quotient $C$\+comodules
$F^iN/F^{i+1}N$ have trivial $C$\+comodule structures.
 We also assume that $F^iN=N$ for $i\ll0$ and $F^iN=0$ for $i\gg0$.
 The category $\F$ has a natural exact category structure in which
a short sequence $0\rarrow N'\rarrow N\rarrow N''\rarrow0$ is exact
if and only if the sequence of associated graded vector spaces
$0\rarrow\gr_FN'\rarrow\gr_FN\rarrow\gr_FN''\rarrow0$ is exact
(cf.~\ref{exact-cat-examples}$\.$(5)).
 Notice that the category $\F$ only depends on the subcoalgebra
$\Nilp C\sub C$.

 Let $\E$ denote the abelian category of finite-dimensional left
$C$\+comodules; then there is an exact forgetful functor $\F\rarrow\E$.
 There is a twist functor $X\maps X(1)$ on the category $\F$ defined
by the rule $F^iN(1)=F^{i-1}N$.
 For any object $X\in\F$ there is a natural morphism $X\rarrow X(1)$
that is transformed to the identity endomorphism by the forgetful
functor $\F\rarrow\E$.
 Denote by $E_i\in\F$ the one-dimensional trivial $C$\+comodule $k$
placed in the filtration component~$i$.
 Then $E_{i+1}=E_i(1)$ and there are natural morphisms $E_i\rarrow
E_{i+1}$ for all $i\in\Z$ corresponding to the identity morphism
$k\rarrow k$.
 Let $H(C)=\bop_n H^n(C)$ be the cohomology algebra of
the coaugmented coalgebra $C$ (see~\cite[Subsection~1.1]{PV})
and $\D=\D^b(\F)$ be the bounded derived category of~$\F$
(see \ref{exact-derived} or~\cite{Neem}).

\begin{thm} \
\begin{enumerate}
\renewcommand{\theenumi}{\arabic{enumi}}
 \item The objects $E_i\in\D$ satisfy
the conditions~\eqref{exceptional-strong}.
 \item One has\/ $\Ext_\F^n(E_i,E_j)=0$ for $n>j-i$.
 The graded $k$\+algebra $A$ with the components
$A_n=\Ext_\F^n(E_0,E_n)$ is quadratic and isomorphic to
the ``quadratic part''\/ $\qu H(C)$ of the graded algebra $H(C)$
(see~\cite[Subsection~2.1]{PV} for the definition).
 \item The morphisms\/ $\Ext_\F^n(E_i,E_j)\rarrow H^n(C)$
induced by the functor $\F\rarrow\E$ are isomorphisms for all\/
$n\le j-i$ if and only if the graded $k$\+algebra $H(C)$ is Koszul.
\end{enumerate}
\end{thm}

\begin{proof}
 Let $\F'$ be the category of (possibly infinite-dimensional)
filtered $C$\+comodules $(N,F)$ such that $F^iN=0$ for $i\gg0$
and $N=\bigcup_iF^iN$.
 The filtration $F$ on $N$ is assumed to be compatible with
the filtration $F$ on $C$, so the successive quotients
$F^iN/F^{i+1}N$ have trivial $C$\+comodule structures.
 Then the embedding functor $\F\rarrow\F'$ induces isomorphisms
on the $\Ext$ spaces.
 Indeed, for any objects $X\in\F$ and $Y\in\F'$ and an admissible
epimorphism $Y\rarrow X$ there exists an object $Z\in\F$ and
a morphism $Z\rarrow Y$ such that the composition $X\rarrow Y
\rarrow Z$ is an admissible epimorphism.
 The same applies to the embedding of the abelian category $\E$ to
the category $\E'$ of all $C$\+comodules.

 For any filtered vector space $V$ over~$k$ the $C$\+comodule
$C\ot_kV$ with its filtration $F$ induced by the filtration $F$ on
$C$ and the filtration on $V$ is an injective object of $\F'$.
 Indeed, for any object $X=(N,F)\in\F$ the space $\Hom_\F(X\;C\ot_kV)$
is isomorphic to the space of filtration-preserving $k$\+linear
maps $X\rarrow V$, and in the exact category of filtered vector
spaces all exact sequences are split.
 Now consider the reduced cobar-resolution $\widetilde B^\bu$ of
the trivial $C$\+comodule~$k$
$$
 C\lrarrow C\ot_k C_+\lrarrow C\ot_k C_+\lrarrow C\ot_k C_+\ot_k C_+
 \lrarrow\dsb
$$
and endow it with a filtration $F$ induced by the filtration $F$
on $C$.
 The complex $(\widetilde B^\bu,F)$ is an injective resolution of
the object $E_0\in\F$, since the associated graded complex is exact.
 Computing the $\Ext$ spaces in $\F'$ in terms of this resolution,
we obtain the isomorphisms
$$
 \Ext_\F^n(E_i,E_j)=H^n(F^{i-j}B^\bu),
$$
where $B^\bu=\Hom_C(k,\widetilde B^\bu)$ is the cobar-complex
$$
 k\lrarrow C_+\lrarrow C_+\ot_k C_+\lrarrow C_+\ot_k C_+\ot_k C_+
 \lrarrow\dsb
$$
with its induced filtration~$F$.
 Similarly, $\Ext_\E^n(k,k) = H^n(B^\bu)$.

 This immediately proves part~(1) and the first assertion of~(2).
 Besides, one can easily see that the algebra $A$ is quadratic dual
to the quadratic part of the graded coalgebra $\gr_FC$
(cf.~\cite[Proposition~2]{PV}).
 It is isomorphic to $\qu H(C)$ by~\cite[Lemma~5.1]{Pbogom}
and~\cite[proof of Main Theorem]{PV}.
 So it remains to prove part~(3).

 ``Only if'': if $H^n(F^{-i}B^\bu)=H^n(B^\bu)$ for all $n\le i$,
then the quotient complex $F^{-n}B^\bu/F^{-n+1}B^\bu$ has
no cohomology except in degree~$n$.
 But the accociated graded complex $\gr_FB^\bu$ is isomorphic to
the cobar-complex of the graded coalgebra $\gr_FC$.
 If the latter has no cohomology outside of the diagonal $n=i$,
then the graded coalgebra $\gr_FC$ is Koszul and consequently
the quadratic dual algebra $A$ is Koszul, too.

``If'': if the algebra $H(C)$ is Koszul, then
$H(\Nilp C)\simeq H(C)$ \cite[Corollary~5.3]{Pbogom}
and the graded coalgebra $\gr_FC$ is
Koszul~\cite[proof of Main Theorem]{PV}.
 Thus $H^n(F^{-i+1}B^\bu)=H^n(F^{-i}B^\bu)$ for all $n<i$.
 It follows by passing to the inductive limit in~$i$ that
$H^n(F^{-i}B^\bu)=H^n(\Nilp C)$ for all $n\le i$.
\end{proof}

\begin{rem}
 Instead of using the results
of~\cite[Main Theorem and its proof]{PV}
and~\cite[Section~5]{Pbogom} in the above argument, one can reprove
and generalize these results using the methods developed in
this paper, and particularly in Section~\ref{associated-graded-secn}.
 See Section~\ref{nonfiltered-secn}.
\end{rem}

\Section{Filtered Exact Subcategory}  \label{filtered-exact-secn}

 Let $\D$ be a triangulated category and $\E_i\sub \D$, $i\in\Z$
be full subcategories, closed under extensions and such that
\begin{equation}  \label{triang-vanish-area}
\text{$\Hom_\D(X,Y[n])=0$ for all $X\in\E_i$, \ $Y\in\E_j$, }
\begin{array}{ll}
 \text{$i$, $j\in\Z$} &\text{and $n=-1$,} \\
 \text{or $i>j$} &\text{and $n=0$ or~$1$.}
\end{array}
\end{equation}
 Let $\E$ be an exact category and $\Phi\:\D\rarrow\D(\E)$ be
a triangulated functor mapping $\E_i$ into~$\E$.
 Assume that 
\begin{equation}  \label{triang-bl-area}
\begin{array}{l}
\text{the induced morphisms $\Hom_\D(X,Y[n])\rarrow
\Ext_\E^n(X,Y)$} \\
\text{are isomorphisms for all $X\in\E_i$, \ $Y\in\E_j$, \ $i<j$,
and $n=0$ or~$1$,} \\
\text{and monomorphisms for $i+2\le j$ and $n=2$.}
\end{array}
\end{equation}

 Let $\M$ be the minimal full subcategory of $\D$, containing
all $\E_i$ and closed under extensions.
 Then both $\E_i$ and $\M$ have natural exact category
structures (see~\cite{Dy} or~\ref{exact-triangulated}) and
the functors $\Phi\:\E_i\rarrow\E$ are exact.

 Let $\F$ be the category whose objects are triples $(N,Q,\rho)$,
where $N=(N,F)$ is a finitely filtered object of $\E$, \
$Q=(Q_i)$ is a finitely supported object of the Cartesian
product of $\E_i$, and $\rho:\gr_FN\rarrow \Phi(Q)$ is an isomorphism.
 Here a ``finitely filtered object'' $(N,F)$ is a sequence
$$
 \dsb\llarrow F^{i-1}N\llarrow F^iN\llarrow F^{i+1}N\llarrow\dsb
$$
of objects of $\E$ and admissible monomorphisms between them
such that $F^iN=0$ for $i\gg0$ and $F^{i+1}N\rarrow F^iN$ is
an isomorphism for $i\ll0$.
 The (stabilizing) inductive limit of $F^iN$ as $i\to-\infty$ is
denoted by~$N$.
 The object $\gr_FN$ is the collection of objects $F^iN/F^{i+1}N\in\E$.
 A ``finitely supported object'' $(Q_i)$ is a collection of objects
$Q_i\in\E_i$ such that $Q_i=0$ for all but a finite number of
indices~$i$.
 The object $\Phi(Q)$ is the collection of objects $\Phi(Q_i)\in\E$.

 The category $\F$ has a natural exact category structure in which
a short sequence with a zero composition is exact if the related
sequence of graded objects $Q=(Q_i)$ is exact in the Cartesian product
of $\E_i$, i.~e., exact in each~$i$
(cf.~\ref{exact-cat-examples}$\.$(4-5)).
 The exact category $\E_i$ is equivalent to the full exact subcategory
of $\F$ consisting of all the triples $(N,Q,\rho)$ such that $Q_j=0$
for all $j\ne i$.

\begin{thm} \
\begin{enumerate}
\renewcommand{\theenumi}{\arabic{enumi}}
 \item In the above situation, the exact categories $\M$ and $\F$
are naturally equivalent.
 \item Conversely, given any exact categories $\E$ and\/ $\E_i$ and
exact functors $\Phi_i\:\E_i\rarrow\E$, construct the exact category
$\F$ by the above procedure.
 Set $\D=\D^b(\F)$.
 Then the subcategories $\E_i\sub\F\sub\D$ and the functor
$\Phi\:\D\rarrow\D(\E)$ induced by the forgetful functor $\F\rarrow\E$
satisfy the conditions~\textup{(\ref{triang-vanish-area}--%
\ref{triang-bl-area})}.
 Moreover, the second assertion of~\eqref{triang-bl-area} holds
for all\/ $i<j$ and\/ $n=2$.
\end{enumerate}
\end{thm}

\subsection{Proof of part~(1)} \label{proof-part-one}
 Since the natural morphisms
$$
 \Hom_\M^n(X,Y)\lrarrow\Hom_\D(X,Y[n]), \quad
 \text{$X$, $Y\in\M$, \ $n\ge0$}
$$
(see~\ref{exact-triangulated}) are isomorphisms for $n=0$ or~$1$,
by~\eqref{triang-vanish-area} one has $\Ext_\M^n(X,Y)=0$ for
$X\in\E_i$, \ $Y\in\E_j$, \ $i>j$, and $n=0$ or~$1$.
 Using associativity of extensions in exact categories
(associativity of the ``$*$\+operation''; cf.~\cite[Lemma~1.3.10]{BBD}
where this is done for triangulated categories), one can deduce
from this vanishing of $\Ext^1$ that any object of $\M$ has a finite
decreasing filtration with successive quotients in $\E_i$.
 It follows from the vanishing of $\Ext^0$ that such filtrations
are unique and preserved by all the morphisms in~$\M$.
 One can also see that a short sequence with zero composition in
$\F$ is exact if and only if its short sequence of successive
quotients is exact in the Cartesian product of~$\E_i$.
 
 Applying the functor $\Phi$ to the above filtration on an object
of~$\M$, one obtains a filtered object of $\E$.
 This defines the desired functor $\M\rarrow\F$.
 To prove that it is an equivalence of exact categories, one can
use the following lemma. 

\begin{lem}
 Let $\Lambda\:\A\rarrow\B$ be an exact functor between exact
categories and\/ $\C\sub\A$ be a class of objects such that
every object of\/ $\A$ can be obtained from objects of\/ $\C$ by
successive extensions and every object of\/ $\B$ can be obtained
from objects of $\Lambda(\C)$ in the same way.
 Then if the maps\/ $\Ext_\A^n(X,Y)\rarrow
\Ext_\B^n(\Lambda(X),\Lambda(Y))$ are isomorphisms for all $X$, $Y\in\C$
and $n=0$ or~$1$ and monomorphisms for all $X$, $Y\in\C$ and $n=2$,
then the functor $\Lambda$ is an equivalence of exact categories.
\end{lem}

\begin{proof}
 Using the five-lemma and induction on the number of subquotients
in an iterated extension, one can show that the functor $\Lambda$
induces isomorphisms on $\Ext^n$ for $n=0$ and~$1$.
 This implies the assertion of Lemma.
\end{proof}

 Let us check that the conditions of Lemma are satisfied for
the functor $\M\rarrow\F$ and the class of objects $\C=\bigcup_i\E_i$.
 Let $X\in\E_i$ and $Y\in\E_j$.
 For $i>j$, one has $\Ext^n_\M(X,Y)=0=\Ext^n_\F(X,Y)$ for any~$n$,
because the natural decreasing filtrations on objects of $\M$ and
$\F$ split any such extensions.
 For $i=j$, one has $\Ext^n_\M(X,Y)=\Ext^n_{\E_i}(X,Y)=\Ext^n_\F(X,Y)$
for any~$n$, due to the same natural filtrations.

 It remains to consider the case $i<j$.
 For $n=0$ or~$1$, by~\ref{exact-triangulated}
and~\eqref{triang-bl-area} the natural maps $\Ext^n_\M(X,Y)\rarrow
\Hom_\D(X,Y[n])\rarrow\Ext^n_\E(X,Y)$ are isomorphisms.
 It is straightforward to check that the natural maps $\Ext^n_\F(X,Y)
\rarrow\Ext^n_\E(X,Y)$ are isomorphisms, too.
 For $n=0$, this is so because any morphism $\Phi(X)\rarrow\Phi(Y)$
in $\E$ is compatible with the filtrations on $\Phi(X)$ and $\Phi(Y)$,
while any morphism between $X$ and $Y$ considered as objects of
the Cartesian product of $\E_i$ is zero.
 For $n=1$, it suffices to say that any extension of $\Phi(X)$ and
$\Phi(Y)$ in $\E$ determines an object of $\F$, which becomes
an extension of $X$ and $Y$ in~$\F$.

 It follows from what we have proven so far that the full exact
subcategories consisting of extensions of objects from $\E_i$ and
$\E_{i+1}$ in $\M$ and $\F$ are equivalent for any~$i$.
 Indeed, bijectivity on $\Ext^0$ and injectivity on $\Ext^1$
between the generating objects suffice to conclude that
the functor is fully faithful (see the proof of Lemma), and then
surjectivity on $\Ext^1(X,Y)$ for $X\in\E_i$, \ $Y\in\E_{i+1}$
implies surjectivity of the functor on the objects.
 The exact category structures on the two categories are the same,
since in each of them a triple with zero composition is exact if
and only if its triple of successive quotients is exact in
$\E_i\times\E_{i+1}$.
 Consequently, the map $\Ext^2_\M(X,Y)\rarrow\Ext^2_\F(X,Y)$
is an isomorphism for $i+1=j$.

 Finally, for $i+2\le j$ by~\ref{exact-triangulated}
and~\eqref{triang-bl-area} the natural maps $\Ext^2_\M(X,Y)\rarrow
\Hom_\D(X,Y[2])\rarrow\Ext^2_\E(X,Y)$ are monomorphisms and they
form a commutative square with the maps $\Ext^2_\M(X,Y)\rarrow
\Ext^2_\F(X,Y)\rarrow\Ext^2_\E(X,Y)$, hence the map $\Ext^2_\M(X,Y)
\rarrow\Ext^2_\F(X,Y)$ is a monomorphism.

\subsection{Proof of part~(2)}  \label{proof-part-two}
 We have already explained why $\Ext_\F^n(X,Y)=0$ for all $i>j$ and
$n\ge0$ and why $\Ext_\F^n(X,Y)\simeq\Ext^n_\E(X,Y)$ for all $i<j$
and $n=0$ or~$1$.
 Let us show that the map $\Ext_\F^2(X,Y)\rarrow\Ext_\E^2(X,Y)$
is injective for all $i<j$.

 Denote by $\F_{[i,j]}$ the minimal full subcategory of $\F$
containing $\E_r$, \ $i\le r\le j$, and closed under extensions.
 The notation $\F_{(-\infty,i]}$ and $\F_{[i,+\infty)}$ has
the similar meaning.
 Any object of $Z\in\F$ can be presented as an extension of
an object $Z_{\le r}\in \F_{(-\infty,r]}$ and an object
$Z_{\ge r+1}\in\F_{[r+1,+\infty)}$ in a unique and functorial way.

 Suppose that a class in $\Ext_\F^2(X,Y)$ is presented as
the product of two classes in $\Ext_\F^1(X,Z)$ and $\Ext_\F^1(Z,Y)$.
 Consider the exact triple $Z_{\ge i}\rarrow Z\rarrow Z_{\le i-1}$.
 Since $\Ext^1_\F(X,Z_{\le i-1})=0$, our class in $\Ext_\F^1(X,Z)$
comes from a class in $\Ext_\F^1(X,Z_{\ge i})$, so we can replace
$Z$ with $Z_{\ge i}$.
 Analogously, one can replace $Z$ with $Z_{\le j}$, so as to get
$Z\in\F_{[i,j]}$.

 Now suppose that our class in $\Ext_\F^2(X,Y)$ is annihilated
by the forgetful functor $\Phi\:\F\rarrow\E$.
 Let our classes in $\Ext_\F^1(X,Z)$ and $\Ext_\F^1(Z,Y)$ be
presented by exact triples $Z\rarrow V\rarrow X$ and $Y\rarrow U
\rarrow Z$ in $\F$, so our class in $\Ext_\F^2(X,Y)$ is
presented by the Yoneda extension $Y\rarrow U\rarrow V\rarrow X$.
 So the Yoneda extension
$$
 \Phi(Y)\lrarrow \Phi(U)\lrarrow \Phi(V)\lrarrow \Phi(X)
$$
in $\E$ is trivial.
 This means that the morphism $\Phi(U)\rarrow \Phi(V)$ can be presented
as a composition $\Phi(U)\rarrow T\rarrow \Phi(V)$ in such a way that
the triples $\Phi(Y)\rarrow T\rarrow \Phi(V)$ and $\Phi(U)\rarrow T
\rarrow \Phi(X)$ are exact in $\E$
(see Corollary~\ref{exact-derived}.3).

 Define a decreasing filtration on~$T$ by setting $F^iT=T$ and
$F^rT=F^r\Phi(U)$ for all $r>i$, where the filtration $F$ on $\Phi(U)$
is a part of the data constituting the object $U\in\F$.
 The morphisms $F^{r+1}T\rarrow F^rT$ are admissible monomorphisms with
the successive quotients $T/F^{i+1}T\simeq \Phi(V)/F^{i+1}\Phi(V)$,
\ $F^rT/F^{r+1}T\simeq F^r\Phi(Z)/F^{r+1}\Phi(Z)$ for $i<r<j$, and
$F^jT\simeq F^j\Phi(U)$.
 This allows to lift $T$ to an object $W$ in $\F$ so that
the morphism $U\rarrow V$ factorizes as $U\rarrow W\rarrow V$
and the triples $Y\rarrow W\rarrow V$ and $U\rarrow W\rarrow X$
are exact.
 Thus our class in $\Ext^2_\F(X,Y)$ vanishes.  \qed

\Section{Associated Graded Category}  \label{associated-graded-secn}

\subsection{Posing the problem}  \label{graded-category-posing}
 Let $\F$ be a small exact category endowed with a sequence of
full subcategories $\E_i\subset\F$, \ $i\in\Z$.
 Assume that each subcategory $\E_i$ is closed under extensions
and every object of $\F$ can be obtained as a successive
extension of objects from~$\E_i$.
 Moreover, assume that
\begin{equation}  \label{exact-filtered-eqn}
 \Hom_\F(X,Y)=0=\Ext^1_\F(X,Y)
 \quad\text{for all $X\in\E_i$, \ $Y\in\E_j$, and $i>j$.}
\end{equation}
 Furthermore, suppose that a twist functor $X\maps X(1)$ is
defined on $\F$ such that it is an autoequivalence of $\F$
as an exact category and $\E_i(1)=\E_{i+1}$.
 Finally, suppose that a natural transformation
$\sigma\:X\rarrow X(1)$ is defined for all $X\in\F$
and $\sigma_{X(1)}=\sigma_X(1)$ for all~$X$.
 Assume that
\begin{equation}  \label{nat-transf-iso}
\begin{array}{l}
 \text{the induced maps $\Hom_\F(X,Y)\rarrow\Hom_\F(X,Y(r))$} \\
 \text{are isomorphisms for all $X$, $Y\in\E_i$, \
 $i\in\Z$, \ $r\ge0$.}
\end{array}
\end{equation}
 It follows from~\eqref{exact-filtered-eqn} (see the beginning of
Subsection~\ref{proof-part-one}) that there is an exact
functor of ``successive quotients'' $q = (q_i)_{i\in\Z}\:
\F\rarrow\prod_i\E_i$.
 Let us require that
\begin{equation}  \label{nat-transf-factorize}
 \begin{array}{l}
 \text{any morphism $X\rarrow Y$ annihilated by~$q$} \\
 \text{factorizes through the morphism
$\sigma_{Y(-1)}\:Y(-1)\rarrow Y$,} \\
 \text{or equivalently, through the morphism
$\sigma_X\:X\rarrow X(1)$.}
 \end{array}
\end{equation}
 It follows by induction on the number of subquotients in an iterated
extension from~\eqref{nat-transf-iso} and the first equation
in~\eqref{exact-filtered-eqn} that such a factorization is unique if
it exists, i.~e., the morphisms $\sigma_X$ are surjective and
injective (see~\ref{additive-categories}).

\begin{ex}
 Let $\E$ and $\E_0$ be exact categories and $\Phi_0\:\E_0\rarrow\E$
be a fully faithful exact functor (so the image of an exact triple
under $\Phi_0$ must be an exact triple, but the converse is not
required).
 Set $\E_i=\E_0$ and $\Phi_i=\Phi_0$ for all $i\in\Z$ and consider
the exact category $\F$ of finitely filtered objects in $\E$ with
subquotients lifted to $\E_i$ as constructed in
Section~\ref{filtered-exact-secn}.
 Identify each $\E_i$ with the full exact subcategory of~$\F$
consisting of all the triples $(N,Q,\rho)$ such that $Q_j=0$
for $j\ne i$.
 Let the twist functor $X\maps X(1)$ on $\F$ be defined by
$F^iN(1)=F^{i-1}N$, \ $Q(1)_i=Q_{i-1}$, and the morphism
$\sigma_X$ be acting by the identity on $N$ and by zero on~$Q$.
 Then all the conditions~(\ref{exact-filtered-eqn}--%
\ref{nat-transf-factorize}) are satisfied.
\end{ex}

 In particular, given a coaugmented coalgebra $C$ over a field~$k$
one can take $\E$ to be the abelian category of left $C$\+comodules,
$\E_0$ to be the abelian category of finite-dimensional $k$\+vector
spaces, and $\Phi_0$ to be the functor endowing a vector space with
the trivial $C$\+comodule structure.
 Then the category $\F$ from the above example coincides with
the category $\F$ from Section~\ref{coalgebra-secn}.
 More generally, given a coalgebra $C$ over $k$ endowed with
an increasing filtration $F_0C\subset F_1C\subset\dsb$ compatible
with the comultiplication, one can consider the exact category $\F$
of finite-dimensional filtered $C$\+comodules.
 This category $\F$ also satisfies~(\ref{exact-filtered-eqn}--%
\ref{nat-transf-factorize}).

 Together with the latter category, one can consider the category $\G$
of finite-dimensional graded comodules over the graded coalgebra
$\gr_FC$.
 Then there will be the exact functor of associated graded comodule
$\gr_F\:\F\rarrow\G$.
 Our goal in this section is to extend this construction to arbitrary
exact categories $\F$ satisfying
the conditions~(\ref{exact-filtered-eqn}--\ref{nat-transf-factorize}).
 So we would like to assign to such a category $\F$ an exact category
$\G$ with the following properties.

 The exact category $\G$ should be endowed with a sequence of full
exact subcategories $\E_i$, \ $i\in\Z$.
 Each subcategory $\E_i$ should be closed under extensions and
every object of $\G$ should be an iterated extension of objects
from~$\E_i$.
 Furthermore, one should have
\begin{equation}
\begin{alignedat}{2}  \label{exact-graded-eqn}
 \Hom_\G(X,Y)&=0
 &&\quad\text{for all $X\in\E_i$, \ $Y\in\E_j$, \ $i\ne j$;} \\
 \Ext^1_\G(X,Y)&=0
 &&\quad\text{for all $X\in\E_i$, \ $Y\in\E_j$, \ $i>j$.}
\end{alignedat} 
\end{equation} 
 Finally, an autoequivalence of the exact category $\G$, denoted by
$X\maps X(1)$, should be defined so that $\E_i(1)=\E_{i+1}$.

 The categories $\F$ and $\G$ should be related in the following way.
 There should be an exact functor $\gr\:\F\rarrow\G$ mapping each
$\E_i\sub\F$ to $\E_i\sub\G$ and defining an equivalence between
these exact categories.
 The functor $\gr$ should be compatible with the twist functors
$X\maps X(1)$ on $\F$ and~$\G$.
 And most importantly, for any objects $X$, $Y\in\F$ there should be
a functorial long exact sequence
\begin{multline}  \label{filtered-graded-sequence}
 \dsb\lrarrow\Ext^n_\F(X,Y(-1))\lrarrow\Ext^n_\F(X,Y) \\
 \lrarrow\Ext^n_\G(\gr\.X\;\gr\.Y)\lrarrow\Ext^{n+1}_\F(X,Y(-1))
 \lrarrow\dsb
\end{multline}

\begin{thm}
 There is a construction assigning to any small exact category $\F$
with the additional data
satisfying~\textup{(\ref{exact-filtered-eqn}--%
\ref{nat-transf-factorize})} a small exact category\/ $\G$ with
the additional data and an exact functor\/ $\gr\:\F\rarrow\G$
satisfying~\textup{(\ref{exact-graded-eqn}--%
\ref{filtered-graded-sequence})}.
\end{thm}

 The proof of Theorem occupies the rest of this section.

\subsection{A construction of the category~$\G$}
\label{graded-category-construction}
 Consider the category $\H$ whose objects are diagrams of the form
$$
 V\lrarrow U\lrarrow V(1)\lrarrow U(1),
$$
where $U$, $V\in\F$, the rightmost map is obtained by applying
the twist functor $X\maps X(1)$ to the leftmost map,
the compositions $V\rarrow U\rarrow V(1)$ and
$U\rarrow V(1)\rarrow U(1)$ are equal to the morphisms
$\sigma_V$ and $\sigma_U$, and the induced sequence
$$
 q(V)\lrarrow q(U)\lrarrow q(V(1))\lrarrow q(U(1))
$$
is exact (see~\ref{exact-derived}) in $\prod_i\E_i$.

 Let $\Delta\:\H\rarrow\prod_i\E_i$ be the functor assigning
to a diagram $(U,V)$ the object $\Im(q(U)\to q(V(1)))$.
 Let $\I$ denote the ideal of morphisms in $\H$ annihilated
by $\Delta$, and let $\H/\I$ be the quotient category.
 Let $\S$ be the class of morphisms in $\H/I$ that
the functor~$\Delta$ sends to isomorphisms.
 Our first aim is to show that the class $\S\sub\H/\I$ is localizing
(i.~e., satisfies the Ore conditions).
 It is clear that if any two morphisms $X\birarrow Y$ in $\H/\I$
have equal compositions with a morphism $X'\rarrow X$ or $Y\rarrow Y'$
belonging to $\S$, then these two morphisms $X\birarrow Y$ are equal.
 
 Let $(X,Y)\rarrow (K,L)\larrow (U,V)$ be two morphisms in $\H$
such that the morphism $\Delta(U,V)\rarrow\Delta(K,L)$ is
an admissible epimorphism in $\prod_i\E_i$.
 Then the morphism $U\oplus L\rarrow K$ is an admissible epimorphism
in~$\F$.
 Indeed, it suffices to check that the morphism $q(U)\oplus q(L)
\rarrow q(K)$ is an admissible epimorphism, since an extension of
admissible epimorphisms is always an admissible epimorphism.
 Consider the fibered product $X\pro_K(U\oplus L)$ in $\F$
(see axiom~Ex$2'$ in~\ref{axioms}).
 Let the map $X\pro_K(U\oplus L)\rarrow (Y\oplus V)(1)$ be
defined as the composition $X\pro_K(U\oplus L)\rarrow X\oplus U
\rarrow (Y\oplus V)(1)$ and the map $Y\oplus V\rarrow
X\pro_K(U\oplus L)$ be induced by the maps $Y\rarrow X$, \
$Y\rarrow L$, \ $V\rarrow U$, and minus the map $V\rarrow L$.
 Then the diagram
$$
 Y\oplus V\lrarrow X\pro_K(U\oplus L)\lrarrow
 (Y\oplus V)(1)\lrarrow (X\pro_K(U\oplus L))(1)
$$
is an object of the category~$\H$.
 There are natural morphisms from this object to the objects $(X,Y)$
and $(U,V)$; the square diagram formed by these two morphisms and
the morphisms $(X,Y)\rarrow (K,L)\larrow (U,V)$ is commutative
modulo~$\I$.
 The object $\Delta(X\pro_K(U\oplus L)\;Y\oplus V)$ is the fibered
product of $\Delta(X,Y)$ and $\Delta(U,V)$ over $\Delta(K,L)$.
 In particular, if the morphism $(U,V)\rarrow(K,L)$ belongs to $\S$,
then so does the morphism $(X\pro_K(U\oplus L)\;Y\oplus V)\rarrow
(X,Y)$. 

 This proves a half of the Ore conditions; the dual half is proved
in the dual way.

\begin{rem}
 The category of diagrams $(U,V)$ is perhaps best viewed as
a DG\+category, and even an exact DG\+category in the sense
of~\cite[Remark~3.5]{Pkoszul}; its full subcategory $\H$ of all
diagrams satisfying the exactness condition can be then
considered as its full exact DG\+subcategory.
 Very roughly, the natural transformation~$\sigma$ plays the role of
a (central) curvature element in a (purely even) ``CDG\+ring'' $\F$. 
 It would be interesting to know whether one can use some derived
category of the second kind of this exact DG\+category in order
to approach the following problem.
 How to construct the exact quotient category $\G=\F/\sigma$ given
only a twist functor $X\maps X(1)$ and a natural transformation
$\sigma\:X\rarrow X(1)$ on an exact category $\F$, satisfying some
reasonable conditions (e.~g., that the morphisms $\sigma_X$ be
injective and surjective)?
\end{rem}

\subsection{Exact category structure on~$\G$}  \label{exact-structure}
 The category $\G$ is defined as the localization
$\G=(\H/\I)[\S^{-1}]$.
 Set a short sequence in $\G$ to be exact if its image under
the functor $\Delta$ is exact.
 Let us check that this defines an exact category structure on~$\G$. 

 Consider a morphism~$f$ in $\G$ whose image under $\Delta$ is
an admissible epimorphism.
 It is clear that such a morphism is surjective in~$\G$.
 Represent~$f$ by a morphism $(U,V)\rarrow (K,L)$ in $\H$ and
apply the construction from~\ref{graded-category-construction}
to the pair of morphisms $(0,0)\rarrow (K,L)\larrow (U,V)$.
 We obtain a morphism $(\Ker(U\oplus L\to K)\; V)\rarrow (U,V)$
in $\H$ whose image~$g$ in $\G$ completes the morphism~$f$ to
an exact triple.
 Let us check that the morphism~$g$ is the kernel of~$f$.
 Any morphism with the target $(U,V)$ in $\G$ can be represented
by a morphism $(X,Y)\rarrow(U,V)$ in~$\H$.
 Assume that the composition $(X,Y)\rarrow (U,V)\rarrow (K,L)$
is annihilated by~$\Delta$.
 Then it follows from~\eqref{nat-transf-factorize} that
the morphism $X\rarrow K$ factorizes through~$L$, since
the composition $X\rarrow K\rarrow L(1)$ is annihilated by~$q$.
 This allows to lift the morphism $(X,Y)\rarrow (U,V)$ to
a morphism $(X,Y)\rarrow (\Ker(U\oplus L\to K)\; V)$ in~$\H$.
 
 Finally, suppose that we are given an exact triple in~$\G$; it
can be represented by a short sequence $(S,T)\rarrow(U,V)
\rarrow(K,L)$ in~$\H$.
 Any morphism with the target $(K,L)$ in $\G$ can be represented
by a morphism $(X,Y)\rarrow (K,L)$ in~$\H$.
 Applying the construction of~\ref{graded-category-construction}
again, we obtain an object $(X\pro_K(U\oplus L)\;Y\oplus V)$
in $\H$ together with natural morphisms from it to $(X,Y)$
and $(U,V)$.
 Just as above, one can construct a morphism
$(S,T)\rarrow(X\pro_K(U\oplus L)\;Y\oplus V)$ in $\H$
and the triple $(S,T)\rarrow (X\pro_K(U\oplus L)\;Y\oplus V)
\rarrow (X,Y)$ is exact in~$\G$.
 These observations together with their dual versions suffice to
check that $\G$ satisfies the exact category axioms Ex0\+-Ex3
from~\ref{axioms}.

 The functor $\gr\:\F\rarrow\G$ assigns to an object $X$
the diagram $(X,X(-1))$; it is obviously exact.
 The twist functor $Z\maps Z(1)$ on $\G$ is induced by
the twist functor $(X,Y)\maps (X(1),Y(1))$ on~$\H$.
 The exact subcategory $\E_i\sub\G$ consists of all objects~$Z$
such that the graded object $\Delta(Z)\in\prod_j\E_j$ is
concentrated in the grading $j=i$.
 One can construct a filtration $(Z_{\ge i})$ on an object $Z=(X,Y)
\in\G$ with successive quotients in $\E_i$ by the rules $Z_{\ge i}
= (X_{\ge i}\;\Ker(Y_{\ge i-1}\to q_{i-1}(X))$ or, equivalently,
$Z_{\ge i} = (\Ker(X_{\ge i-1}\to q_{i-2}(Y)\;Y_{\ge i-1})$;
see Subsection~\ref{proof-part-two} for the notation $W_{\ge i}$.
 Using this filtration, one can easily check that the functor $\gr$
induces equivalences between the exact subcategories $\E_i$ in $\F$
and $\G$ and the exact category $\G$
satisfies~\eqref{exact-graded-eqn}.

\subsection{Two lemmas}  \label{two-lemmas}
 The following lemmas will be needed in the remaining part of
the proof.

\begin{lem1}
 Let $\Lambda\:\A\rarrow\B$ be an exact functor between exact
categories such that for any admissible epimophism $T\rarrow\Lambda(X)$
in $\B$ there exist an admissible epimorphism $Z\rarrow X$ in $\A$
and a morphism $\Lambda(Z)\rarrow T$ in $\B$ making the triangle
$\Lambda(Z)\rarrow T\rarrow\Lambda(X)$ commute.
 Let $\xi$ be a class in $\Ext^n_\A(X,Y)$ and $\eta$ be a class
in $\Ext^m_B(\Lambda(Y),W)$ such that $\eta\Lambda(\xi)=0$ and $m\ge1$.
 Then there exist a morphism $f\:Y'\to Y$ in $\A$ and a class
$\xi'\in\Ext^n_\A(X,Y')$ such that $\xi=f\xi'$ and $\eta\Lambda(f)=0$.
\end{lem1}

\begin{proof}
 We will be mostly interested in the case $m=1$; however, in
Section~\ref{base-restriction-secn} we will also use the case $n=1$.
 Let us start with some general remarks.
 By Proposition~\ref{exact-derived}, a Yoneda extension
$(B\to C_1\to\dsb\to C_n\to A)$ in an exact category $\E$ is trivial
if and only if there exists a Yoneda extension
$(B\to D_1\to\dsb\to D_n\to A)$ mapping both to the extension
$(B\to C_1\to\dsb\to C_n\to A)$ and to the trivial extension
$(B\to B\to 0\to\dsb\to0\to A\to A)$.
 The latter condition simply means that the admissible monomorphism
$B\to D_1$ splits.
 Furthermore, one can make the morphisms $D_i\rarrow C_i$ admissible
epimorphisms by replacing $D_i$ with $D_i\oplus C_i\oplus C_{i-1}$
for $1<i<n$, \ $D_1$ with $D_1\oplus C_1$, and $D_n$ with
$D_n\oplus C_{n-1}$.

 Now we apply these observations to the case of the Yoneda extension
$$
 (W\rarrow V_1\rarrow\dsb\rarrow V_m\rarrow\Lambda(Z_1)\rarrow
\dsb\rarrow\Lambda(Z_n)\rarrow\Lambda(X))
$$
obtained by composing an extension representing $\eta$ with
the image under $\Lambda$ of an extension representing~$\xi$.
 An extension $(W\to T_1\to\dsb\to T_{n+m}\to\Lambda(X))$ maps
both to this composition and to the trivial extension, and moreover,
the maps $T_{j+m}\rarrow \Lambda(Z_j)$ are admissible epimorphisms.
 Using the assumption of Lemma and decreasing induction on~$j$,
one can construct admissible epimorphisms $Z'_j\rarrow Z_j$
forming a map of Yoneda extensions $(Y'\to Z'_1\to\dsb\to Z'_n\to X)
\rarrow (Y\to Z_1\to\dsb\to Z_n\to X)$ whose image under $\Lambda$
factorizes through the map of Yoneda extensions $\Im(T_m\to T_{m+1})
\to T_{m+1}\to\dsb\to T_{m+n}\to \Lambda(X))\rarrow
(\Lambda(Y)\to \Lambda(Z_1)\to\dsb\to\Lambda(Z_n)\to\Lambda(X))$.
 This provides the desired morphism $f\:Y'\rarrow Y$ and
class $\xi'\in\Ext^n(X,Y')$.

 Notice that we have obtained slightly more than we wanted:
the assertion of Lemma does not require the map~$f$ to be
an admissible epimorphism.
 However, applying the assumption of Lemma to construct
an admissible epimorphism $Z'_n\rarrow Z_n$ provides a way
to obtain an admissible epimorphism $Z'_n\rarrow X$ together
with a morphism $Z'_n\rarrow Z_n$ in the category $\A$
that we really need.
\end{proof}

 Recall the definition of a \emph{big graded ring}
from~\ref{big-graded-rings}.
 The next lemma can be viewed as a module version of
Corollary~\ref{exact-triangulated}.1.

\begin{lem2}
 Let $\Lambda\:\A\rarrow\B$ be an exact functor between small exact
categories satisfying the assumptions of Lemma~1.
 Then for any object $W\in\B$ the right graded module
$(\Ext^n_\B(\Lambda(X),W))_{X\in\A;\.n\ge0}$ over the big graded ring
$(\Ext^n_\A(X,Y))_{Y,X\in\A;\.n\ge0}$ over the set of all objects
of~$\A$ is induced from the right module $(\Hom_\B(\Lambda(Y),W))_Y$
over the big subring $(\Hom_\A(X,Y))_{Y,X}\sub
(\Ext^n_\A(X,Y))_{Y,X;\.n}$.
\end{lem2}

\begin{proof}
 There is an obvious natural map
$$
 (\Hom_\B(\Lambda(Y),W))_Y\ot_{(\Hom_\A(X,Y))_{Y,X}}
 (\Ext^n_\A(X,Y))_{Y,X;\.n}\lrarrow
 (\Ext^n_\B(\Lambda(X),W))_{X;\.n}.
$$
 It is surjective, since for any class $\eta\in\Ext^n_\B(\Lambda(X),W)$
there exist a class $\zeta\in\Ext^n_\A(X,Y)$ and a morphism
$g\:\Lambda(Y)\rarrow W$ in $\B$ such that $\eta=g\Lambda(\zeta)$.
 One shows this using the assumption of Lemmas~1--2 in a way similar
to (and simpler than) that of the proof of Lemma~1.

 Since the category $\A$ admits finite direct sums, in order to
check injectivity it suffices to show that for any class
$\zeta\in\Ext^n_\A(X,Y)$ and morphism $g\:\Lambda(Y)\rarrow W$
in $\B$ such that $g\Lambda(\zeta)=0$ and $n\ge1$ there exists
a morphism $h\:Y\rarrow Z$ in $\A$ and a morphism $t\:\Lambda(Z)
\rarrow W$ in $\B$ such that $g=t\Lambda(h)$ and $h\zeta=0$.
 Let us first consider the case $n=1$.
 Present the class~$\zeta$ by an extension $Y\rarrow Z\rarrow X$;
then the equation $g\Lambda(\zeta)=0$ means that the morphism~$g$
factorizes through the morphism $\Lambda(Y)\rarrow\Lambda(Z)$.
 So it suffices to take the admissible monomorphism $Y\rarrow Z$
as~$h$.

 Now we return to the general case $n\ge1$.
 Present the class~$\zeta$ as the composition of a class
$\xi\in\Ext^{n-1}_\A(X,U)$ and a class $\theta\in\Ext^1_\A(U,Y)$.
 Set $\eta=g\Lambda(\theta)$; then we have $\eta\Lambda(\xi)=0$.
 By Lemma~1, there exists a morphism $U'\rarrow U$ in $\A$
and a class $\xi'\in\Ext^{n-1}_\A(X,U')$ such that $\xi=f\xi'$
and $\eta\Lambda(f)=0$.
 Set $\theta'=\theta f$; then we have $g\Lambda(\theta')=0$.
 Consequently, there exists a morphism $h\:Y\rarrow Z$ in $\A$
and a morphism $t\:\Lambda(Z)\rarrow W$ in $\B$ such that
$g=t\Lambda(h)$ and $h\theta'=0$.
 Then $h\zeta=h\theta\xi=h\theta f\xi'=h\theta'\xi'=0$.
\end{proof}

\subsection{Construction of the boundary map}
 It remains to obtain the long exact
sequence~\eqref{filtered-graded-sequence}; we start with
constructing the boundary map $\d\:\Ext^n_\G(\gr\.X\;\gr\.Y)\rarrow
\Ext^{n+1}_\F(X,Y(-1))$ for all $X$, $Y\in\F$ and $n\ge 0$.

 Clearly, any object $(U,V)$ in $\G$ is the target of an admissible
epimorphism $\gr\.U=(U,U(-1))\rarrow (U,V)$; analogously, $(U,V)$
is the source of an admissible monomorphism into the object
$\gr\.V(1)=(V(1),V)$.
 Moreover, any admissible epimorphism $T\rarrow \gr\.X$ in $\G$
can be represented by a morphism $(U,V)\rarrow (X,X(-1))$ in $\H$,
hence there are admissible epimorphisms $U\rarrow X$ in $\F$ and
$\gr\.U\rarrow (U,V)$ in $\G$ such that the triangle
$\gr\.U\rarrow (U,V)\rarrow \gr\.X$ commutes.
 So the exact functor $\gr\:\F\rarrow\G$ satisfies the assumption of
Lemmas~1--2 of~\ref{two-lemmas}.

 To construct the image of a morphism $\gr\.X\rarrow\gr\.Y$ under
the homomorphism $\d_0\:\Hom_\G(\gr\.X\;\gr\.Y)\rarrow
\Ext^1_\F(X,Y(-1))$, choose an admissible epimorphism $X'\rarrow X$
and a morphism $X'\rarrow Y$ in $\F$ such that the triangle
$\gr\.X'\rarrow\gr\.X\rarrow\gr\.Y$ commutes in~$\G$.
 Let $K$ be the kernel of the admissible epimorphism $X'\rarrow X$;
then the composition $K\rarrow X'\rarrow Y$ is annihilated by
the functor $\gr$, and consequently factorizes through the morphism
$\sigma_{Y(-1)}\:Y(-1)\rarrow Y$.
 The morphism $K\rarrow Y(-1)$ that we have obtained induces from
the exact triple $K\rarrow X'\rarrow X$ the desired extension of
$X$ and $Y(-1)$ in~$\F$.

 Alternatively, choose an admissible monomorphism $Y\rarrow Y'$
and a morphism $X\rarrow Y'$ in $\F$ such that the triangle
$\gr\.X\rarrow\gr\.Y\rarrow\gr\.Y'$ commutes in~$\G$.
 Let $C$ be the cokernel of the admissible monomorphism $Y\rarrow Y'$;
then the composition $X\rarrow Y'\rarrow C$ is annihilated by
the functor $\gr$, so we obtain a morphism $X(1)\rarrow C$.
 This morphism induces from the extension $Y\rarrow Y'\rarrow C$
an extension of $X(1)$ and $Y$ in~$\F$.
 Let us show that the two extensions that we have obtained only
differ by (a twist and) the minus sign.

 The difference of the compositions $X'\rarrow X\rarrow Y'$
and $X'\rarrow Y\rarrow Y'$ is annihilated by the functor $\gr$,
so there is a morphism $X'\rarrow Y'(-1)$ in~$\F$.
 Together with the exact triples $K\rarrow X'\rarrow X$ and
$Y\rarrow Y'\rarrow C$ and the morphisms $K\rarrow Y(-1)$ and
$X\rarrow C(-1)$, this morphism forms a diagram in which one
square commutes and the other one anticommutes.
 It follows immediately that the map $\Hom_\G(\gr\.X\;\gr\.Y)
\rarrow\Ext^1_\F(X,Y(-1))$ given by either of the above two rules
is a well-defined homomorphism of bimodules over the big ring
$\Hom_\F(X,Y)_{Y,X}$ over the set $\Ob\F$ of all objects of~$\F$.

 By Lemma~2 of~\ref{two-lemmas} and its dual version,
the bimodule $(\Ext^n_\G(\gr\.X\;\gr\.Y))_{Y,X\in\F;\.n\ge0}$
over the big graded ring $(\Ext^n_\F(X,Y))_{Y,X\in\F;\.n\ge0}$
over $\Ob\F$ considered as either a left or right module is
induced from its zero grading component
$(\Hom_\G(\gr\.X\;\gr\.Y))_{Y,X\in\F}$ as a module over the zero
grading component $(\Hom_\F(X,Y))_{Y,X\in\F}$ of the big graded ring.
 We want our maps
$$
 \d=\d_n\:\Ext^n_\G(\gr\.X\;\gr\.Y)\lrarrow\Ext^{n+1}_\F(X,Y(-1))
$$
to satisfy the equations $\d(\Lambda(\xi)\eta)=(-1)^{|\xi|}\xi\d(\eta)$
and $\d(\eta\Lambda(\zeta))=\d(\eta)\zeta$ for any $\Ext$ classes
$\xi$ and~$\zeta$ of the degrees $|\xi|$ and~$|\zeta|$ in
the exact category $\F$ and any $\eta\in\Ext^n_\G(\gr\.X\;\gr\.Y)$.
 Either of these two equations defines the sequence of maps~$\d_n$
uniquely, and one only has to check that the two conditions are
compatible.
 It suffices to check this for a class in $\Ext_\G^1(\gr\.X\;\gr\.Y)$
decomposed into the product of a class in $\Ext_\F^1(U,Y)$ and
an element in $\Hom_\G(\gr\.X\;\gr\.U)$ and also into the product
of a class in $\Ext_\F^1(X,V)$ and an element in $\Hom_\G(\gr\.V\;
\gr\.Y)$.

 We have two exact triples $Y\rarrow S\rarrow U$ and
$V\rarrow T\rarrow X$ in $\F$ and a morphism of exact triples
$(\gr\.V\to\gr\.T\to\gr\.X)\rarrow(\gr\.Y\to\gr\.S\to\gr\.U)$
in~$\G$.
 Choose an admissible epimorphism $X'\rarrow X$ in $\F$ such that
the composition $\gr\.X'\rarrow \gr\.X\rarrow\gr\.U$ comes from
a morphism $X'\rarrow U$ in~$\F$.
 Denote by $T'''$ the fibered product of $T$ and $X'$ over~$X$.
 Choose an admissible epimorphism $T''\rarrow T'''$ such that
the composition $\gr\.T''\rarrow\gr\.T\rarrow\gr\.S$ comes from
a morphism $T''\rarrow S$ in~$\F$.
 Consider the difference of the compositions $T''\rarrow T'''
\rarrow X'\rarrow U$ and $T''\rarrow S\rarrow U$.
 It is annihilated by the functor $\gr$, and consequently factorizes
through the morphism $\sigma_{U(-1)}$; hence we get a morphism
$T''\rarrow U(-1)$.
 Denote by $T'$ the fibered product of $T''$ and $S(-1)$ over $U(-1)$.
 Define the morphism $T'\rarrow X'$ as the composition $T'\rarrow
T''\rarrow T'''\rarrow X'$ and the morphism $T'\rarrow S$ as 
the sum of the compositions $T'\rarrow T''\rarrow S$ and
$T'\rarrow S(-1)\rarrow S$.
 Then the square formed by the morphisms $T'\rarrow X'\rarrow U$
and $T'\rarrow S\rarrow U$ is commutative, as is the triangle
$\gr\.T'\rarrow\gr\.T\rarrow\gr\.S$.

 Let $V'$ be the kernel of the admissible epimorphism $T'\rarrow X'$.
 Then there is an admissible epimorphism of exact triples
$(V'\to T'\to X')\rarrow(V\to T\to X)$ and a morphism of exact triples
$(V'\to T'\to X')\rarrow(Y\to S\to U)$ whose images in $\G$ form
a commutative triangle with the morphism of exact triples
$(\gr\.V\to\gr\.T\to\gr\.X)\rarrow(\gr\.Y\to\gr\.S\to\gr\.U)$.
 Let $K\rarrow L\rarrow M$ be the kernel of the admissible epimorphism
$(V'\to T'\to X')\rarrow(V\to T\to X)$.
 Then the morphism of exact triples $(K\to L\to M)\rarrow
(Y\to S\to U)$ is annihilated by the functor $\gr$, so there is 
a morphism of exact triples $(K\to L\to M)\rarrow(Y(-1)\to S(-1)
\to U(-1))$.
 Consider the extension of exact triples $V\rarrow T\rarrow X$ and
$K\rarrow L\rarrow M$ and induce an extension of the exact triples
$V\rarrow T\rarrow X$ and $Y(-1)\rarrow S(-1)\rarrow U(-1))$ using
the above morphism.
 We have obtained a commutative $3\times3$ square formed by exact
triples.
 For any such square, the two $\Ext^2$ classes between the objects
at the opposite vertices obtained by composing the $\Ext^1$ classes
along the perimeter differ by the minus sign.
 This proves the desired equation in $\Ext^2_\F(X,Y)$.

\subsection{Exactness of the long sequence}
 Checking that the long sequence is a complex is easy.
 We will start with proving exactness of the segment
\begin{multline*}
 0\lrarrow\Hom_\F(X,Y(-1))\lrarrow\Hom_\F(X,Y) \lrarrow
\Hom_\G(\gr\.X\;\gr\.Y) \\
 \lrarrow\Ext^1_\F(X,Y(-1))\lrarrow\Ext^1_\F(X,Y)\lrarrow
\Ext^1_\G(\gr\.X\;\gr\.Y).
\end{multline*}

 We have already explained in~\ref{graded-category-posing}
that exactness at the term $\Hom_\F(X,Y(-1))$ follows from 
the conditions (\ref{exact-filtered-eqn}--\ref{nat-transf-iso}).
 Exactness at the term $\Hom_\F(X,Y)$ is provided by
the condition~\eqref{nat-transf-factorize}.
 Let us prove exactness at the term $\Hom_\G(\gr\.X\;\gr\.Y)$.
 Suppose we are given a morphism $\gr\.X\rarrow\gr\.Y$,
an admissible epimorphism $X'\rarrow X$, and a morphism $X'\rarrow Y$
such that the triangle $\gr\.X'\rarrow\gr\.X\rarrow\gr\.Y$ commutes.
 Let $K$ be the kernel of the morphism $X\rarrow Y$ and the composition
$K\rarrow X'\rarrow Y$ be factorized as $K\rarrow Y(-1)\rarrow Y$.
 Assume that the extension induced from the extension
$K\rarrow X'\rarrow Y$ using the morphism $K\rarrow Y(-1)$ splits.
 Then the morphism $K\rarrow Y(-1)$ factorizes through the morphism
$K\rarrow X'$, so there is a morphism $X'\rarrow Y(-1)$.
 Subtracting from the morphism $X'\rarrow Y$ the composition
$X'\rarrow Y(-1)\rarrow Y$ we obtain a new morphism $X'\rarrow Y$
that annihilates~$K$.
 Hence this morphism factorizes through the admissible epimorphism
$X'\rarrow X$, providing the desired morphism $X\rarrow Y$.

 Let us check exactness at the term $\Ext_\F^1(X,Y(-1))$.
 Suppose we are given an extension $Y(-1)\rarrow Z\rarrow X$
such that the morphism $\sigma_{Y(-1)}$ factorizes through
the admissible monomorphism $Y(-1)\rarrow Z$.
 Then we have a morphism $Z\rarrow Y$, and the induced morphism
$\gr\.Z\rarrow\gr\.Y$ annihilates the admissible monomorphism
$\gr\.Y(-1)\rarrow\gr\.Z$.
 Consequently, the morphism $\gr\.Z\rarrow\gr\.Y$ factorizes through
the admissible epimorphism $\gr\.Z\rarrow\gr\.X$, providing
a morphism $f\:\gr\.X\rarrow\gr\.Y$.
 By the definition, the class of our extension $Y(-1)\rarrow Z
\rarrow X$ is equal to~$\d f$.

 Let us prove exactness at the term $\Ext_\F^1(X,Y)$.
 Suppose that an exact triple $Y\rarrow Z\rarrow X$ becomes split
after the functor $\gr$ is applied to it.
 Then there exists a splitting morphism $\gr\.X\rarrow\gr\.Z$.
 Consequently, there exist an admissible epimorphism $f\:X'\rarrow X$
and a morphism $X'\rarrow Z$ such that the composition $X'\rarrow Z
\rarrow X$ is the sum of~$f$ and a morphism annihilated by $\gr$,
while the composition $\Ker(f)\rarrow X'\rarrow Z$ is also
annihilated by~$\gr$.
 Let $g\:X'\rarrow X$ be the sum of the morphism~$f$ and
the composition $X'\rarrow Z\rarrow X$.
 Then $g$~is also an admissible epimorphism, since $\gr\.g$ is;
and the composition $\Ker(g)\rarrow X'\rarrow Z$ is also
annihilated by $\gr$, since $\gr\.g=\gr\.f$.
 Now our exact triple $Y\rarrow Z\rarrow X$ is induced from
the exact triple $\Ker(g)\rarrow X'\rarrow X$ by a morphism
$\Ker(g)\rarrow Y$ annihilated by~$\gr$.
 Since the latter map factorizes through $\sigma_{Y(-1)}$,
we are done.

 Exactness at the further terms can be deduced from the exactness
in this initial segment using Lemma~1 of~\ref{two-lemmas}
applied to the functors $\Id_\F$ and $\gr\:\F\rarrow\G$. \qed

\Section{Restriction of Base}  \label{base-restriction-secn}

 Let $\G$ be an exact category endowed with a sequence of full
subcategories $\E_i$,\ $i\in\Z$.
 Assume that each subcategory $\E_i$ is closed under extensions
and every object of $\G$ is an iterated extension of objects
from~$\E_i$.
 Furthermore, assume that
\begin{equation}  \label{exact-trivial}
 \begin{array}{l}
 \text{the induced exact category structures on $\E_i$ are trivial,}\\
 \text{i.~e., for any $i\in\Z$ every exact triple in $\E_i\sub\G$
 splits.}
 \end{array}
\end{equation}
 Finally, assume that the groups $\Hom$ and $\Ext^1$ in the category
$\G$ satisfy the conditions~\eqref{exact-graded-eqn}.

 Let $\A_i$ be additive categories and $\psi_i\:\A_i\rarrow\E_i$
be additive functors such that
\begin{equation}  \label{essentially-surjective}
\text{every object of $\E_i$ is a direct summand of an object
coming from $\A_i$.}
\end{equation}
 Let $\H$ denote the category whose objects are the triples
$(X,Q,\rho)$, where $Q\in\prod_i\A_i$ is a finitely supported
graded object, $X$ is an object of $\G$, and $\rho\:q(X)\rarrow
\psi(Q)$ is an isomorphism in $\prod_i\E_i$
(cf.\ Section~\ref{filtered-exact-secn}), where $q$ denotes
the functor of ``successive quotients'' $\G\rarrow\prod_i\E_i$
(see Subsection~\ref{graded-category-posing}) and $\psi=(\psi_i)$.

 The category $\H$ has an exact category structure in which
a short sequence is exact if the related sequence of graded objects
$Q$ is split exact.
 The additive categories $\A_i$ can be considered as the full
subcategories of $\H$ consisting of all the triples $(X,Q,\rho)$
such that $Q_j=0$ for $j\ne i$.
 Then the exact category $\H$ with the full subcategories $\A_i$
satisfies all the above conditions~\eqref{exact-trivial}
and~\eqref{exact-graded-eqn} imposed on the exact category $\G$
with the full subcategories $\E_i$.

 There is the natural (forgetful) exact functor $\Psi\:\H\rarrow \G$.

\begin{ex}
 Let $S\rarrow R$ be a morphism of rings, $\E$ be the category
of finitely generated projective left $R$\+modules, $\A$ be
the category of finitely generated projective left $S$\+modules,
and $\psi\:\A\rarrow\E$ be the functor of extension of scalars,
$\psi(N)=R\ot_SN$.
 Then the functor $\psi$ satisfies~\eqref{essentially-surjective}.
\end{ex}

\begin{thm}
 The homomorphism $\Ext^n_\H(X,Y)\rarrow\Ext^n_\G(\Psi(X),\Psi(Y))$
induced by the functor $\Psi$ are
\begin{enumerate}
\renewcommand{\theenumi}{\arabic{enumi}}
\item epimorphisms for all $n\ge 1$ and $X$, $Y\in\H$;
\item isomorphisms for $n=1$ and $X$, $Y\in\H$, if the graded objects
$q(X)$ and $q(Y)\in\prod_i\A_i$ are supported in disjoint sets of
indices~$i$;
\item isomorphisms for all $n\ge 2$ and $X$, $Y\in\H$.
\end{enumerate}
\end{thm}

\begin{proof}
 Let $\A'_i$ be the additive category obtained by adjoining to $\A_i$
the images of those idemponent endomorphisms $p$ for which
the image of $\psi_i(p)$ exists in $\E_i$.
 Then the functor $\psi_i\:\A_i\rarrow\E_i$ factorizes through
the embedding $\A_i\rarrow\A'_i$, so there is an additive functor
$\psi'_i\:\A'_i\rarrow\E_i$.
 According to~\eqref{essentially-surjective}, the functors $\psi'_i$
are surjective on the isomorphism classes of objects.
 The same construction as above provides the exact category $\H'$
with the full subcategories $\A'_i$, the embedding of exact
categories $\H\rarrow\H'$, and the exact functor
$\Psi'\:\H'\rarrow\G$.
 The full exact subcategory $\H$ is closed under extensions in $\H'$
and all objects of the exact category $\H'$ are direct summands of
certain objects of $\H$, hence the embedding $\H\rarrow\H'$
induces isomorphisms on $\Ext^n(X,Y)$ for all $n\ge0$ and
$X$, $Y\in\H$ (see Corollary~\ref{exact-triangulated}.3).
 This reduces the problem to the case when the additive functors
$\psi_i$ are surjective on the isomorphism classes of objects,
which we will assume in the sequel.
 This part of the argument does not depend on the assumption that
the exact category structures on $\A_i$ and $\E_i$ are trivial.

 Conversely, given an exact category $\G$ with full subcategories
$\E_i$ as above, the exact category $\G^\sat$ with the full
subcategories $\E_i^\sat$ (see~\ref{additive-categories}) satisfies
the same conditions that we have imposed on $\G$ and $\E_i$.
 Applying the above construction to the exact category $\G^\sat\supset
\E_i^\sat$ and the additive functors $\E_i\rarrow\E_i^\sat$, one can
recover the original exact category $\G$ with the full
subcategories~$\E_i$.

 Part~(1): notice that the functor $\Psi$ is surjective on
the isomorphism classes of objects.
 So it suffices to prove that the map of the $\Ext$ groups is
surjective for $n=1$.
 Suppose that we are given an exact triple $\Psi(Y)\rarrow T\rarrow
\Psi(X)$ in~$\G$.
 Consider the related split exact triple $q\Psi(Y)\rarrow q(T)\rarrow
q\Psi(X)$ in $\prod_i\E_i$.
 Choosing a splitting, one can identify $q(T)$ with $q\Psi(Y)\oplus
q\Psi(X) \simeq \psi(q(Y)\oplus q(X))$.
 This defines a lifting of the object $T\in\G$ to an object $Z\in\H$
and of the exact triple $\Psi(Y)\rarrow T\rarrow \Psi(X)$ to
an exact triple $Y\rarrow Z\rarrow X$ in~$\H$.
 Part~(2) is obvious, since in its assumptions any splitting
of the exact triple $\Psi(Y)\rarrow\Psi(Z)\rarrow\Psi(X)$ is
simultaneously  a splitting of the exact triple $Y\rarrow Z\rarrow X$.

 It follows from surjectivity on $\Ext^1$ by means of the five-lemma
and induction on the number of iterated extensions that it suffices
to check injectivity on $\Ext^2$ in the case when $X\in\A_i$ and
$Y\in\A_j$ for some $i$, $j\in\Z$.
 Present our class in $\Ext^2_\H(X,Y)$ as the composition of some
classes in $\Ext^1_\H(X,Z)$ and $\Ext^1_\H(Z,Y)$.
 It was explained in Subsection~\ref{proof-part-two} that one can
choose $Z\in\H_{[i,j]}$; since we now assume the exact category
structures on $\A_i$ to be trivial, one can actually choose
$Z\in\H_{[i+1,j-1]}$, for the same reason.
 Let these classes $\Ext^1$ be presented by exact triples
$Y\rarrow U\rarrow Z$ and $Z\rarrow V\rarrow X$.
 The class $\Ext^2_\G(\Psi(X),\Psi(Y))$ represented by the Yoneda
extension $\Psi(Y)\rarrow\Psi(U)\rarrow\Psi(V)\rarrow\Psi(X)$
is trivial if and only if one can decompose the morphism
$\Psi(U)\rarrow\Psi(V)$ as $\Psi(U)\rarrow T\rarrow\Psi(V)$ in
such a way that the triples $\Psi(Y)\rarrow T\rarrow\Psi(V)$
and $\Psi(U)\rarrow T\rarrow\Psi(X)$ are exact
(see Corollary~\ref{exact-derived}.3).
 In this case one has $q_i(T)\simeq q_i\Psi(X)$, \ 
$q_r(T)\simeq q_r\Psi(Z)$ for $i<r<j$, and $q_j(T)\simeq q_j\Psi(Y)$
(see Section~\ref{associated-graded-secn} for the notation~$q_r(T)$;
cf.~\ref{proof-part-two}).
 This allows to lift the object $T\in\G$ to an object $W\in\H$
in such a way that the diagram remains commutative and the triples
remain exact.
 Then the original class in $\Ext^2_\H(X,Y)$ is also trivial.

 Finally, let us prove injectivity on $\Ext^n$ for $n\ge3$ using
injectivity on $\Ext^{n-1}$.
 Notice that the functor $\Psi\:\H\rarrow\G$ satisfies
the assumption of Lemma~1 from Subsection~\ref{two-lemmas}.
 Indeed, for any admissible epimorphism $T\rarrow\Psi(X)$ in $\G$
one can lift the object $T$ to an object $W\in\H$ in such a way
that the morphism $T\rarrow\Psi(X)$ lifts to an admissible
epimorphism $W\rarrow X$ in~$\H$.
 Now consider a class in $\Ext^n_\H(X,Y)$ and decompose it into
the product of classes in $\Ext^1_\H(X,Z)$ and $\Ext^{n-1}_\H(Z,Y)$.
 Assume that the product of the images of these classes in
$\Ext^1_\G(\Psi(X),\Psi(Z))$ and $\Ext^{n-1}_\G(\Psi(Z),\Psi(Y))$
vanishes.
 By the mentioned lemma, there exists a morphism $Z'\rarrow Z$
in $\H$ such that the class in $\Ext^1_\H(X,Z)$ comes from a class
in $\Ext^1_\H(X,Z')$ and the composition of the morphism
$\Psi(Z')\rarrow\Psi(Z)$ in $\G$ with the class in
$\Ext^{n-1}_\G(\Psi(Z),\Psi(Y))$ vanishes.
 Since the map $\Ext^{n-1}_\H(Z',Y)\rarrow\Ext^{n-1}_\G(\Psi(Z'),
\Psi(Y))$ is injective, it follows that the composition of
the morphism $Z'\rarrow Z$ in $\H$ with the class in
$\Ext^{n-1}_\H(Z,Y)$ also vanishes.
 Therefore, the original class in $\Ext^n_\H(X,Y)$ is zero.
\end{proof}

\Section{Diagonal Cohomology}  \label{diagonal-secn}

\subsection{Diagonal Ext is quadratic}  \label{diagonal-is-quadratic}
 Let $\D$ be a small triangulated category endowed with a sequence
of full subcategories $\E_i$, \ $i\in\Z$.
 Assume that each $\E_i$ is closed under extensions in $\D$ and
one has
\begin{equation}  \label{triang-vanish-silly}
\text{$\Hom_\D(X,Y[1])=0$ for all $X\in\E_i$, \ $Y\in\E_j$,
and $i\ge j$.}
\end{equation}
 Finally, suppose that a triangulated autoequivalence $X\maps X(1)$
is defined on the triangulated category $\D$ such that
$\E_i(1)=\E_{i+1}$.
 Let $\J$ be a full subcategory of $\E_0$ such that any object of
$\E_0$ is a finite direct sum of objects from~$\J$.
 Introduce the big graded ring of diagonal cohomology
$A=(\Hom_\D(X,Y(n)[n]))_{Y,X\in\J;\.n\ge0}$ over
the set $\Ob\J$ of all objects of~$\J$ (see~\ref{big-graded-rings}).
 Let $\M$ denote the minimal full subcategory of $\D$ containing
all $\E_i$ and closed under extensions.

\begin{thm}
 Assume that every morphism $X\rarrow Y[n]$ of degree~$n\ge2$
in $\D$ between two objects $X$, $Y\in\M$ can be presented as
the composition of a chain of morphisms $Z_{i-1}\rarrow Z_i[1]$
with $Z_i\in\M$, \ $Z_0=X$, and $Z_n=Y$ \textup{(}cf.\
Appendix~\textup{\ref{silly-filtrations-appx}}\textup{)}.
 Then
\begin{enumerate}
\renewcommand{\theenumi}{\arabic{enumi}}
\item one has\/ $\Hom_\D(X,Y[n])=0$ for all\/ $X\in\E_i$, \
 $Y\in\E_j$, \ $n\ge1$, and\/ $n>j-i$;
\item the big graded ring $A$ is \emph{quadratic}, i.~e., generated
by the $A_0$\+bimodule $A_1$ with relations in degree~$2$.
\end{enumerate}
\end{thm}

 More precisely, a big graded ring $A=\bop_{n=0}^\infty A_n$ is said
to be \emph{quadratic} if the natural maps
$$
 A_1^{\ot n}\big/\textstyle\sum_{1\le j\le n-1}A_1^{\ot j-1}\ot_{A_0}
 I\ot_{A_0}A_1^{\ot n-j-1}\lrarrow A_n
$$
are isomorphisms for all $n\ge 2$,
where the tensor powers of the $A_0$\+bimodule $A_1$ are taken
over $A_0$ and the $A_0$\+bimodule of quadratic relations $I$
is defined as the kernel of the multiplication map
$A_1\ot_{A_0}A_1\rarrow A_2$.

 Notice that the property of a big graded ring $A$ to be quadratic
actually does not depend on the component $A_0$, i.~e., for any
big graded ring $A$ and a morphism of big rings $A_0'\rarrow A_0$
over the same set $\Sigma$ the big graded rings $A$ and
$A'=A_0'\oplus A_1\oplus A_2\oplus\dsb$ are quadratic
simultaneously.
 Indeed, clearly $A$ is generated by $A_1$ over $A_0$ if and only
if $A'$ is generated by $A_1$ over $A'_0$; and the relation
$(ar)b=a(rb)$ for $a$, $b\in A_1$ and $r\in A_0$ has degree~$2$.

\begin{proof}
 As in Subsection~\ref{proof-part-two}, denote by $\M_{[i,j]}$
the minimal full subcategory of $\D$, containing $\E_r$, \
$i\le r\le j$, and closed under extensions.
 It follows from the $*$\+associativity lemma~\cite[Lemma~3.1.10]{BBD}
and the condition~\eqref{triang-vanish-silly} that one has
$\M_{[i,j]}=\M_{[r+1,j]}*\M_{[i,r]}$ for any $i\le r < j$, i.~e.,
for any object $Z\in\M_{[i,j]}$ there exists a distinguished triangle
$Z_{\ge r+1}\rarrow Z\rarrow Z_{\le r}\rarrow Z_{\ge r+1}[1]$ with
$Z_{\ge r+1}\in\M_{[r+1,j]}$ and $Z_{\le r}\in\M_{[i,r]}$.
 This triangle does not have to be unique or functorial in our weak
assumptions~\eqref{triang-vanish-silly}.

 However, it follows from these assumptions that, given $X\in
\M_{[i,+\infty)}$ and $Y\in\M$, any element in $\Hom_\D(X,Y[1])$
comes from an element in $\Hom_\D(X,Y_{\ge i+1}[1])$ (for any choice
of $Y_{\ge i+1}$).
 Decomposing an arbitrary element of $\Hom_\D(X,Y[n])$ with $n\ge 1$
into the product of elements from $\Hom_\D(\M,\M[1])$ and using
induction on~$n$, one can see that such an element has to come from
an element in $\Hom_\D(X,Y_{\ge i+n}[n])$.
 This proves part~(1).

 Applying the above argument together with its dual version, one can
see that any element of $\Hom_\D(X,Y[n])$ with $X\in\E_0$, \
$Y\in\E_n$, and $n\ge2$ can be decomposed into a product of elements
from $\Hom_\D(Z_{i-1},Z_i[1])$ with $Z_i\in\E_i$.
 Taking $X\in\J$, \ $Y\in\J(n)$ and presenting $Z_i$ as finite direct
sums of objects from $\J(i)$, one can show that the big graded ring $A$
is generated by~$A_1$.

 It remains to prove quadraticity.
 With any relation of degree~$n$ between the elements of $A_1$ in
the big graded ring $A$ one can associate a relation of the form
$\xi_1\dsb\xi_n=0$ in the big graded ring
$(\Hom_\D(X,Y(n)[n]))_{Y,X\in\E_0;\.n\ge0}$, where
$\xi_i\in\Hom(Z_{i-1},Z_i(1)[1])$ and $Z_i\in\E_0$.
 Let us show that all such relations follow from relations of
degree~$2$.
 Assume that the product of an element $\xi\in\Hom_\D(X,Z[1])$ and
a class $\eta\in\Hom_\D(Z[1],Y[n])$ is zero, where $X\in\E_0$, \
$Z\in\E_1$, \ $Y\in\E_n$, and $n\ge 3$.

 Consider the distinguished triangle $Z\rarrow T\rarrow X\rarrow
Z[1]$; then the object $T$ belongs to $\M_{[0,1]}$.
 Since $\eta\xi=0$, the morphism $\eta$ extends to a morphism
$\eta'\in\Hom_\D(T[1],Y[n])$.
 By the assumption of Theorem, one can decompose the morphism~$\eta'$
as a product of a morphism $T[1]\rarrow S[2]$ and a morphism
$S[2]\rarrow Y[n]$ in $\D$, where $S\in\M$.
 As explained above, one can assume that $S\in\M_{[1,2]}$.
 So there is a distinguished triangle $U\rarrow S\rarrow V\rarrow
U[1]$ with $U\in\E_2$ and $V\in\E_1$.
 Denote the morphism $V[1]\rarrow U[2]$ by~$\theta$.
 The composition $Z[1]\rarrow T[1]\rarrow S[2]$ factorizes through
the morphism $U[2]\rarrow S[2]$, providing a morphism
$\zeta\:Z[1]\rarrow U[2]$.
 The composition $X\rarrow Z[1]\rarrow U[2]\rarrow S[2]$ vanishes,
since the composition $X\rarrow Z[1]\rarrow T[1]$ does; hence
the composition $X\rarrow Z[1]\rarrow U[2]$ factorizes through
the morphism~$\theta$, providing a morphism $\chi\:X\rarrow V[1]$.
 Denote the composition $U[2]\rarrow S[2]\rarrow Y[n]$ by~$\lambda$.
 Now the relation $\eta\xi=0$ of degree~$n$ follows from the relations
$\eta=\lambda\zeta$, \ $\zeta\xi=\theta\chi$, and $\lambda\theta=0$
of the degrees $n-1$, \ $2$, and $n-1$, respectively.

 Decomposing the objects $V$ and $U$ into finite direct sums of
objects from $\J(1)$ and $\J(2)$, one can conclude that the original
relation of degree~$n$ in $A$ follows from relations of
degree~$\le n-1$.
\end{proof}

\subsection{Any quadratic ring can be realized}
\label{quadratic-realized}
 Let $A$ be a big nonnegatively graded ring over a set $\Sigma$;
suppose that $A$ is quadratic.
 Starting from $A$, we would like to construct an exact category
$\G$ with the following properties.

 The exact category $\G$ should be endowed with a sequence of
full subcategories $\E_i$ and an exact autoequivalence $X\maps X(1)$
such that each $\E_i$ is closed under extensions, every object
of $\G$ is an iterated extension of objects from $\E_i$, and
$\E_i(1)=\E_{i+1}$.
 The conditions \eqref{exact-graded-eqn} and~\eqref{exact-trivial}
should be satisfied.

 There should be a full subcategory $\J\sub\E_0$ such that
every object of $\E_0$ is a finite direct sum of objects from~$\J$.
 The set $\Ob\J$ of all objects of $\J$ should be identified with
$\Sigma$, and the big graded ring of diagonal cohomology
$$
 (\Ext^n_\G(X,Y(n))_{Y,X\in\J;\.n\ge0}
$$
should be identified with~$A$. 

 Finally, one should have
\begin{equation} \label{ext-one-two-diagonal}
 \Ext^n_\G(X,Y)=0 \quad\text{for all $X\in\E_i$, \ $Y\in\E_j$,
\ $n < j-i$, and $n=1$ or~$2$.}
\end{equation}
 Notice that the conditions \eqref{exact-graded-eqn}
and~\eqref{exact-trivial} imply $\Ext^n_\G(X,Y)=0$ for $n>j-i$
by the result of Subsection~\ref{diagonal-is-quadratic}.
 So the sum total of the conditions \eqref{exact-graded-eqn}, \
\eqref{exact-trivial}, and~\eqref{ext-one-two-diagonal}
can be simply restated as
\begin{equation} \label{ext-zero-one-two-diagonal}
 \Ext^n_\G(X,Y)=0 \quad\text{for all $X\in\E_i$, \ $Y\in\E_j$,
\ $n\ne j-i$, and $n\le 2$.}
\end{equation}

\begin{thm}
 For any quadratic big graded ring $A$ over a set $\Sigma$
there exists a unique, up to a unique exact equivalence,
exact category\/ $\G$ with the additional data described
above satisfying the condition~\eqref{ext-zero-one-two-diagonal}.
\end{thm}

 We will give two proofs of this theorem, both of which
will be useful in the sequel.

\begin{proof}[First proof]
 Consider the big ring $R=A_0$.
 We will construct the category $\G$ as a certain category of graded
comodules over a graded coring over the big ring~$R$.
 Generally, a \emph{coring} $C$ over $R$ is an $R$\+bimodule endowed
with $R$\+bimodule maps $C\rarrow C\ot_RC$ and $C\rarrow R$, called
the comultiplication and counit, satisfying the conventional
coassociativity and counit axioms.
 A \emph{right comodule} $N$ over a coring $C$ is a right $R$\+module
endowed with an $R$\+module map $N\rarrow N\ot_RC$, called
the right coaction, satisfying the conventional axioms.

 Let $C=C_0\oplus C_{-1}\oplus C_{-2}\oplus\dsb$ be a nonpositively
graded coring over a big ring $R$ such that $C_0=R$.
 Such a coring $C$ is said to be \emph{quadratic} if
the comultiplication map $C_{-2}\rarrow C_{-1}\ot_R C_{-1}$
is injective and $C$ is the universal final object in the category
of graded corings over $R$ with the components
$C_0$, \ $C_{-1}$, and $C_{-2}$ fixed.
 Given an $R$\+bimodule $C_{-1}$ and a subbimodule $C_{-2}\sub
C_{-1}\ot C_{-1}$, the corresponding quadratic coring $C$ always
exists.
 Its components can be constructed by induction in
such a way that the initial fragment of the reduced cobar-complex
$$
 0\lrarrow C_+\lrarrow C_+\ot_R C_+\lrarrow C_+\ot_R C_+\ot_R C_+,
$$
where $C_+=C/R$, would be exact in the gradings~$-3$ and below.

 The coring $C$ we are interested in is the quadratic coring over $R$
with the initial components $C_{-1}=A_1$ and $C_{-2} = I = 
\Ker(A_1\ot_R A_1\to A_2)$.
 The quadratic coring $C$ constructed in this way is called
\emph{quadratic dual} to the quadratic big ring~$A$.

 For any finite set mapping to $\Sigma$ one can consider
the corresponding finitely generated free right $R$\+module
(see~\ref{big-graded-rings}).
 The additive category $\E_0$ of such $R$\+modules is endowed
with the full subcategory $\J$ consisting of free modules with one
generator; the set $\Ob\J$ is identified with $\Sigma$ and
the big ring $(\Hom_{\E_0}(X,Y))_{Y,X\in\J}$ is naturally
identified with~$R$.
 Set $\G$ to be the category of graded right $C$\+comodules that
are free and finitely generated as graded right $R$\+modules.
 The full subcategories $\E_i$ consist of the graded comodules
concentrated in degree~$i$; the twist functor $X\maps X(1)$
shifts the grading.
 A triple in $\G$ is exact if it is (split) exact in every degree.

 The conditions \eqref{exact-graded-eqn} and~\eqref{exact-trivial}
are obviously satisfied.
 It is also clear that $\Ext^1_\G(X,Y)=0$ for all $X\in\E_i$, \
$Y\in\E_j$, and $j-i>1$, since $C$ is cogenerated by~$C_{-1}$.
 To show that $\Ext^2_\G(X,Y)=0$ for all $X\in\E_i$, \ $Y\in\E_j$,
and $j-i>2$, present an arbitrary class in $\Ext^2_\G(X,Y)$ as
the product of classes in $\Ext^1_\G(X,Z)$ and $\Ext^1_\G(Z,Y)$,
where $Z\in\G_{[i+1,j-1]}$.
 Since $C$ is quadratic, one can check that the $C$\+comodule
structures on the graded $R$\+modules $X\oplus Z$ and $Z\oplus Y$
can be extended to a $C$\+comodule structure on $X\oplus Z\oplus Y$.

 It is straightforward to identify the $R$\+bimodule
$(\Ext^1_\G(X,Y(1)))_{Y,X\in\J}$ with $A_1$ and the $R$\+bimodule
$(\Ext^2_\G(X,Y(2)))_{Y,X\in\J}$ with $A_2$.
 Thus the big graded ring $(\Ext^n_\G(X,Y(n))_{Y,X\in\J;\.n\ge0}$
is isomorphic to $A$ by the result of~\ref{diagonal-is-quadratic}.

 Conversely, let $\H$ be an exact category with the additional data
satisfying all the conditions of Theorem except
perhaps~\eqref{ext-one-two-diagonal}.
 Then it is not difficult to construct an exact functor
$\Lambda\:\H\rarrow\G$, where $\G$ is the above category of
$C$\+comodules.
 More precisely, one starts with identifying the additive
subcategories $\E_i\sub\H$ with the category of finitely
generated free right $R$\+modules.
 Then one constructs the exact functors $\H_{[i,i+1]}\rarrow\G$
using the class $\Ext^1_\H(X,Y)$ for $X\in\E_i$ and $Y\in\E_{i+1}$
to define the coaction map $\Lambda(X)\rarrow \Lambda(Y)\ot_R C_{-1}$.
 To construct the graded $C$\+comodule structure corresponding to
an arbitrary object in $\H$, one only has to check that
the compositions $\Lambda(X)\rarrow \Lambda(Z)\ot_R C_{-1}\rarrow
\Lambda(Y)\ot_R C_{-1}\ot_R C_{-1}$ factorize through
$\Lambda(Y)\ot_R C_{-2}$ whenever $X=q_i(W)$, \ $Z=q_{i+1}(W)$,
and $Y=q_{i+2}(W)$ for some $W\in\H$.
 This is so because the product of the classes in $\Ext^1_\H(X,Z)$
and $\Ext^1_\H(Z,Y)$ corresponding to $W$ vanishes
in $\Ext^2_\H(X,Y)$.
 
 If the category $\H$ also satisfies~\eqref{ext-one-two-diagonal},
then the functor $\Lambda$ is an equivalence of exact categories
by Lemma from Subsection~\ref{proof-part-one}.
\end{proof}

\begin{proof}[Second proof]
 This is only a proof of existence.
 Consider the DG\+category $\C$ whose objects are indexed by the pairs
$(\sigma,i)$, where $\sigma\in\Sigma$ and $n\in\Z$.
 The complex of morphisms $\Hom_\C((\tau,i),(\sigma,j))$ has its only
possibly nonzero term equal to $A_{\sigma\tau;\.j-i}$ in
the cohomological degree~$j-i$.
 The composition of morphisms in $\C$ comes from the multiplication
in~$A$.

 Let $\D$ denote the derived category of contravariant DG\+functors
from $\C$ to the category of complexes of abelian groups.
 There is an autoequivalence $(\sigma,i)\maps(\sigma,i-1)$ on
the DG\+category $\C$; let $X\maps X(1)$ denotes the induced
autoequivalence of~$\D$.
 Let $\J\sub\D$ be the full subcategory of functors representable by
the objects $(\sigma,0)$.
 Set $\E_0\sub\D$ to be the minimal additive subcategory of $\D$
containing $\J$ and $\E_i=\E_0(i)$.
 Let $\E$ be the minimal full subcategory of $\D$, containing $\E_i$
and closed under extensions; then $\E$ has a natural structure
of exact category.

 Clearly, one has $\Hom_\D(X,Y[n])=0$ for $X\in\E_i$, \ $Y\in\E_j$,
and $n\ne j-i$.
 The big graded ring $(\Hom_\D(X,Y(n)[n]))_{Y,X\in\J;\.n\ge0}$ is
identified with~$A$.
 Since the natural maps $\Ext^n_\D(X,Y)\rarrow\Hom_\D(X,Y[n])$ are
isomorphisms for all $X$, $Y\in\E$ and $n\le1$, and monomorphisms 
for $n=2$, the condition~\eqref{ext-zero-one-two-diagonal} follows.
 Since the big graded ring $\Ext_\E^n(X,Y(n))_{Y,X\in\J;\.n\ge0}$
is quadratic by the result of~\ref{diagonal-is-quadratic},
its morphism to~$A$ has to be an isomorphism.
\end{proof}

\Section{Koszul Big Rings}  \label{koszul-rings-secn}

 Let $A=A_0\oplus A_1\oplus A_2\oplus\dsb$ be a big graded ring
over a set~$\Sigma$.
 A big graded ring $A$ is called \emph{Koszul} if there exists
an exact category $\G$ with full subcategories $\E_i$, a twist
functor $X\maps X(1)$ on it, and a full subcategory $\J\sub\E_0$
satisfying the assumptions of Subsection~\ref{quadratic-realized},
for which the following stronger version of
the condition~\eqref{ext-zero-one-two-diagonal} holds
\begin{equation} \label{ext-diagonal}
 \Ext^n_\G(X,Y)=0 \quad\text{for all $X\in\E_i$, \ $Y\in\E_j$,
 and $n\ne j-i$.}
\end{equation}
 In other words, $A$ is Koszul if the exact category $\G$ 
uniquely determined by the conditions of
Subsection~\ref{quadratic-realized}
(including~\eqref{ext-one-two-diagonal}),
satisfies~\eqref{ext-diagonal}.
 By Theorem from Subsection~\ref{diagonal-is-quadratic}, this
condition always holds for $n>j-i$; the nontrivial part is
the vanishing for $3\le n<j-i$.

 It follows from the result of Section~\ref{base-restriction-secn}
that the Koszul property does not depend on the base ring in
the component of degree~$0$ of $A$.
 More precisely, set $R=A_0$ and let $S\rarrow R$ be any morphism
of big rings over $\Sigma$.
 In addition to $A$, consider the big graded ring
$B=S\oplus A_1\oplus A_2\oplus\dsb$.
 Then the big graded ring $A$ is Koszul if and only if the big
graded ring $B$ is (cf.\ remarks before the proof in
Subsection~\ref{diagonal-is-quadratic}).

 It is known~\cite[Sections~0.4 and~11.4]{Psemi} what the Koszulity
condition means in the case of a nonnegatively graded ring $A$
that is a flat left or right module over its zero-degree component
$R=A_0$.
 We will see below in~\ref{flat-koszul} that our definition is
equivalent to the definition from~\cite{Psemi} in the flat situation.
 For the reader's convenience, we present a brief general discussion
of flat Koszulity over a base big ring in~\ref{generalities}.

 One would like also to have a more explicit definition of Koszulity
in the general case.
 Morally, the idea is to replace $R$ with the ``absolute ring'', known
also as the ``field with one element'' $\mathbb{F}_1$; one presumes
that every module over $\mathbb{F}_1$ is flat.
 In practice, this turns out to involve a certain condition of
``exactness of the matrix Koszul complex'' for~$A$.
 This is worked out, in two different ways, in~\ref{general-koszul}
and~\ref{triangulated-koszul}.

\begin{rem}
 One might think that the simplest way to interpret
the condition~\eqref{ext-diagonal} would be to compute explicitly
the groups $\Ext$ over the graded coring $C$ constructed
in the first proof in Subsection~\ref{quadratic-realized}.
 The problem is, it is not known how to compute $\Ext$ in
the exact category of $R$\+projective comodules over a coring $C$
over a ring $R$ in general; see~\cite[Question~5.1.4]{Psemi}.
 In the flat case, we use this approach in~\ref{flat-koszul}.
\end{rem}

\begin{ex}
 Any big graded ring $A$ such that $A_n=0$ for $n\ge2$ is Koszul.
 The objects of the related exact category $\G$ (as constructed
in the first proof in~\ref{quadratic-realized}) are finitely
supported sequences of finitely generated free right $A_0$\+modules
$N_i$, \ $i\in\Z$, endowed with right $A_0$\+module maps $N_i
\rarrow N_{i+1}\ot_{A_0} A_1$.
 No compatibility condition is imposed on these maps.
 The subcategory $\E_i$ consists of all sequences $N$ such that
$N_j=0$ for $j\ne i$, and $\J\subset\E_0$ is the subcategory of
free modules with one generator.
 One easily checks that $\Ext^2_\G(X,Y)=0$ for any $X\in\E_i$ and
$Y\in\E_j\subset\G$, so $\G$ is an exact category of homological
dimension~$1$ and \eqref{ext-diagonal}~is satisfied.
\end{ex}

\subsection{General case} \label{general-koszul}
 A \emph{$\Sigma$\+colored matrix} $M$ \emph{with entries in} $A_m$
is a (finite, rectangular) matrix whose rows and columns are marked
by elements of $\Sigma$ and the entry $M_{ij}$ belongs to
the group $A_{\sigma\tau;\.m}$ if the $i$\+th row is marked
by~$\sigma$ and the $j$\+th column is marked by $\tau\in\Sigma$.
 A pair of $\Sigma$\+colored matrices $(M,N)$ with entries in
$A_m$ and $A_n$, respectively, is called \emph{composable} if
the number of columns in $M$ equals the number of rows in $N$ and
the columns in $M$ and rows in $N$ corresponding to each other
are marked by the same elements of~$\Sigma$.
 Clearly, the product $MN$ of any two composable matrices $M$ and
$N$ is well-defined as a $\Sigma$\+colored matrix with entries in
$A_{m+n}$.

\begin{thm}
 A big nonnegatively graded ring $A$ is Koszul if and only if
the following condition holds.
 Let $M_{(1)}$,~\ds, $M_{(m)}$, $\.m\ge0$ be $\Sigma$\+colored
matrices with entries in $A_1$ such that every pair
$(M_{(i+1)},M_{(i)})$ is  composable and the product
$M_{(i+1)}M_{(i)}$ is zero.
 Let $N$ be a $\Sigma$\+colored matrix with entries in $A_n$,
$\,n\ge1$, such that the pair $(N,M_{(m)})$ is composable and
the product $NM_{(m)}$ is also zero.
 Then there should exist $\Sigma$\+colored matrices $K_{(1)}$,~\ds,
$K_{(m)}$ with entries in $A_0$, $\,M'_{(1)}$,~\ds, $M'_{(m)}$
with entries in $A_1$, $\,P$ with entries in $A_1$, and $Q$
with entries in $A_{n-1}$ such that
\begin{gather*}
 M_{(1)} = K_{(1)}M'_{(1)},\ M_{(2)}K_{(1)} = K_{(2)}M'_{(2)},\ 
 \dsc,\ M_{(m)}K_{(m-1)} = K_{(m)}M'_{(m)}, \\
 NK_{(m)} = QP,\ M'_{(i+1)}M'_{(i)} = 0
 \text{ for all\/ $i=1$, \ds, $m-1$, and\/ } PM'_{(m)} = 0,
\end{gather*}
where all the pairs of matrices being multiplied are composable.
\end{thm}

 Notice that the above matrix condition is not obviously
self-opposite, i.~e., it is not immediately clear why
the similar condition with the order of factors in the products
of matrices reversed is equivalent to the condition from Theorem.
 However, it will be clear from the proof below that the Koszul
property of the big graded ring $A$ is also equivalent to
the opposite matrix condition, hence the two opposite versions of
the matrix condition are equivalent to each other.

\begin{proof}
 Notice first of all that the condition for $n=1$ is always trivial:
it suffices to take $K_{(i)}$ and $Q$ to be the identity matrices,
$M'_{(i)}=M_{(i)}$, and $P=N$.

 The condition for $m=0$ simply means (or in any event should be
read to mean) that any $\Sigma$\+colored matrix $N$ with entries
in $A_n$ can be decomposed into a product $N=QP$ of two composable
$\Sigma$\+colored matrices $P$ and $Q$ with entries in $A_1$ and
$A_{n-1}$, respectively.
 This is equivalent to the big graded ring $A$ being
multiplicatively generated by $A_1$ over $A_0$.

 The condition for $m=1$ means that for any composable
$\Sigma$\+colored matrices $N$ and $M$ with entries in $A_n$
and $A_1$ such that $NM=0$ there exist $\Sigma$\+colored matrices
$K$ with entries in $A_0$, \ $M'$ with entries in $A_1$, \
$P$ with entries in $A_1$, and $Q$ with entries in $A_{n-1}$
such that $M=KM'$, \ $NK=QP$, and $PM'=0$.
 Together with the condition for $m=0$, this clearly implies
that $A$ is quadratic; the converse implication will follow
from the argument below.

 Let $\G\supset\E_i$, \ $\E_0\supset\J$ denote the exact category
with the additional data corresponding to the quadratic big
graded ring~$A$.

 ``If'': we will show by induction on~$n$ that $\Ext_\G^{n+1}(X,Y)=0$
for all $X\in\E_0$, \ $Y\in\E_{n+m}$, and $m\ge 2$.
 The case $n=1$ is known by~\eqref{ext-one-two-diagonal};
let $n\ge2$.
 Suppose we are given a class in $\Ext_\G^{n+1}(X,Y)$; decompose
it into a product of classes in $\Ext_\G^1(X,Z)$ and
$\Ext_\G^n(Z,Y)$.
 As it was explained in Subsections~\ref{proof-part-two}
and~\ref{diagonal-is-quadratic}, one can assume that
$Z\in\G_{[1,m]}$.
 The sequence of classes in $\Ext_\G^1(X,q_1(Z))$, \
$\Ext_\G^1(q_1(Z),q_2(Z))$,~\ds, $\Ext_\G^1(q_{m-1}(Z),q_m(Z))$
coming from the class in $\Ext_\G^1(X,Z)$ and the filtered object $Z$
defines a sequence of $\Sigma$\+colored matrices $M_{(1)}$,~\ds,
$M_{(m)}$.
 The class in $\Ext_\G^n(q_m(Z),Y)$ coming from the class in
$\Ext_\G^n(Z,Y)$ defines a matrix~$N$.
 The existence of the object $Z$ and the classes in $\Ext^1_\G(X,Z)$
and $\Ext^n_\G(Z,Y)$ implies the equations on the products of
consecutive matrices $M_{(i)}$ and~$N$.

 By assumption, it follows that there exist $\Sigma$\+colored matrices
$K_{(i)}$, \ $M'_{(i)}$, \ $P$, and $Q$ satisfying the equations of
Theorem.
 Since $\Ext^2_\G$ vanishes outside of the diagonal, starting
from the matrices $M'_{(i)}$ one can construct~\cite{Ret} an object
$Z'\in\G_{[1,m]}$ and a class in $\Ext^1_\G(X,Z')$ related to
these matrices in the same way as the object $Z$ and the class in 
$\Ext^1_\G(X,Z)$ are related to the matrices $M_{(i)}$.
 Since $\Ext^1_\G$ also vanishes outside of the diagonal, starting
from the matrices $K_{(i)}$ one can construct a morphism $Z'\rarrow Z$
such that the original class in $\Ext^1_\G(X,Z)$ comes from our
new class in $\Ext^1_\G(X,Z')$.
 
 Consider the class in $\Ext^n_\G(Z',Y)$ obtained by composing
the morphism $Z'\rarrow Z$ with the class in $\Ext^n_\G(Z,Y)$.
 Then the related class in $\Ext^n_\G(q_m(Z'),Y)$ can be decoposed
into the product of classes in $\Ext^1_\G(q_m(Z'),U)$ and
$\Ext^{n-1}_\G(U,Y)$ corresponding to the matrices $P$ and $Q$,
where $U\in\E_{m+1}$.
 Due to the equation $PM'_{(m)}=0$, the class in
$\Ext^1_\G(q_m(Z'),U)$ comes from a class in $\Ext^1_\G(Z',U)$.
 Present the latter class by an exact triple $U\rarrow T\rarrow Z'$.
 The class in $\Ext^1_\G(X,Z')$ comes from a class in
$\Ext^1_\G(X,T)$.
 Let us show that the composition of the morphism $T\rarrow Z'$
with the class in $\Ext^n_\G(Z',Y)$ vanishes.
 The image of that composition in $\Ext^n_\G(T_{\ge m},Y)$ is
the composition of the morphism $T_{\ge m}\rarrow q_m(Z')$ with
the class in $\Ext^n_\G(q_m(Z'),Y)$, which is clearly zero.
 So our class in $\Ext^n_\G(T,Y)$ comes from a class in
$\Ext^n_\G(T_{\le m-1},Y)$.
 The latter class is zero by the assumption of induction on~$n$.

``Only if'': suppose that we are given $\Sigma$\+colored matrices
$M_{(i)}$, \ $m\ge1$, and $N$ such that the consecutive pairs are
composable with a zero product.
 As above, with the matrices $M_{(i)}$ one can associate an object
$Z\in\G_{[1,m]}$ and an extension class in $\Ext^1_\G(X,Z)$ with
$Z\in\E_0$; from the matrix $N$ one can obtain a class in
$\Ext^n_\G(q_m(Z),Y)$ with $Y\in\E_{n+m}$.
 The equation $NM_{(m)}=0$ and the assumption of vanishing of
$\Ext^{n+1}_\G$ outside of the diagonal allow to conclude that
the latter class comes from a class in $\Ext^n_\G(Z,Y)$.
 By the same assumption, the product of these two classes vanishes
in $\Ext^{n+1}_\G(X,Y)$.
 By Lemma~1 from~\ref{two-lemmas}, there exists a morphism
$Z'\rarrow Z$ in $\G$ such that the class in $\Ext^1_\G(X,Z)$ comes
from a class in $\Ext^1_\G(X,Z')$, while the induced class in
$\Ext^n_\G(Z',Y)$ vanishes.
 One can assume that $Z'\in\G_{[1,\infty)}$.

 The morphism $Z'\rarrow Z$ factorizes as $Z'\rarrow Z'_{\le m}
\rarrow Z$, and the induced class in $\Ext^n_\G(Z'_{\le m},Y)$
is the product of the class in $\Ext^1_\G(Z'_{\le m},Z'_{\ge m+1})$
corresponding to $Z'$ and a certain class in
$\Ext^{n-1}_\G(Z'_{\ge m+1},Y)$.
 By the same argument from Subsection~\ref{diagonal-is-quadratic}
that was referred to above, one can assume that
$Z'_{\ge m+1}\in\E_{m+1}$ and $Z'\in\G_{[1,m+1]}$.
 It remains to construct matrices $M'_{(i)}$ from the object
$Z'_{\le m}$ and the class in $\Ext^1_\G(X,Z'_{\le m})$,
matrices $K_{(i)}$ from the morphism $Z'_{\le m}\rarrow Z$,
a matrix $P$ from the class in $\Ext^1_\G(q_m(Z'),q_{m+1}(Z'))$,
and a matrix $Q$ from the class in $\Ext^{n-1}_\G(q_{m+1}(Z'),Y)$.
\end{proof}

\subsection{Koszulity in triangulated setting}
\label{triangulated-koszul}
 Let $\D$ be a triangulated category and $\E_i\sub\D$ be full
additive subcategories such that
\begin{equation}
 \Hom_\D(X,Y[n])=0 \quad\text{for all $X\in\E_i$, \ $Y\in\E_j$, }
\begin{array}{l}
\text{$n=1$, and $i\ge j$, or}\\
\text{$n\ge2$, and $n\ne j-i$.}
\end{array}
\end{equation}
 Furthermore, suppose that a triangulated autoequivalence
$X\maps X(1)$ is defined on $\D$ such that $\E_i(1)=\E_{i+1}$.
 Let $\M$ denote the minimal full subcategory of $\D$, containing
all $\E_i$ and closed under extensions.

 Let $\J\sub\E_0$ be a full subcategory such that every object of
$\E_0$ is a finite direct sum of objects of~$\J$.
 Consider the big graded ring $A=(\Hom_\D(X,Y(n)[n]))_
{Y,X\in\J;\.n\ge0}$.

\begin{thm} \
\begin{enumerate}
\renewcommand{\theenumi}{\arabic{enumi}}
\item The big graded ring $A$ is Koszul if and only if
every morphism $X\rarrow Y[n]$ of degree~$n\ge2$ in $\D$ between
two objects $X$, $Y\in\M$ can be presented as the composition of
a chain of morphisms $Z_{i-1}\rarrow Z_i[1]$ with $Z_i\in\M$, \
$Z_0=X$, and $Z_n=Y$ \textup{(}cf.\
Appendix~\textup{\ref{silly-filtrations-appx}}\textup{)}.
\item An arbitrary big graded ring $A$ is Koszul if and
only if the following condition holds.
 Let $M_{(1)}$,~\ds, $M_{(m)}$, $\.m\ge0$ be $\Sigma$\+colored
matrices with entries in $A_1$ such that every pair
$(M_{(i+1)},M_{(i)})$ is  composable and the product
$M_{(i+1)}M_{(i)}$ is zero.
 Let $N$ be a $\Sigma$\+colored matrix with entries in $A_n$,
$\,n\ge1$, such that the pair $(N,M_{(m)})$ is composable and
the product $NM_{(m)}$ is also zero.
 Then there should exist $\Sigma$\+colored matrices $L_{(0)}$,~\ds,
$L_{(m)}$, $\,M'_{(1)}$,~\ds, $M'_{(m)}$ with entries in $A_1$
and $Q$ with entries in $A_{n-1}$ such that
\begin{gather*}
 N = QL_{(m)},\ L_{(m)}M_{(m)} = M'_{(m)}L_{(m-1)},\
 \dsc,\ L_{(1)}M_{(1)} = M'_{(1)}L_{(0)} \\
 M'_{(i+1)}M'_{(i)} = 0
 \text{ for all\/ $i=1$, \ds, $m-1$, and\/ } QM'_{(m)} = 0.
\end{gather*}
where all the pairs of matrices being multiplied are composable.
\end{enumerate}
\end{thm}

 Notice that~(2) provides a characterization of the Koszul
property of a big graded ring $A$ that is explicitly independent
of the component~$A_0$ (since the condition for $n=1$ is trivial,
see below).

\begin{proof}
 We will show that the decomposition condition in~(1) holds for
a triangulated category $\D$ and its full subcategory $\M$
if and only if the matrix condition in~(2) is satisfied for
the corresponding big graded ring~$A$.
 Then it will remain to consider the triangulated category $\D$
and its exact subcategory $\E$ from the second proof in
Subsection~\ref{quadratic-realized}.
 By Corollary~\ref{exact-triangulated}.2, the decomposition condition
for such $\D$ and $\M=\E$ is equivalent to the condition that
the morphisms $\Ext^n_\E(X,Y)\rarrow\Hom_\D(X,Y[n])$ be isomorphisms
for all $X$, $Y\in\E$ and $n\ge0$, and the latter is clearly equivalent
to the condition~\eqref{ext-diagonal} for the exact category~$\E$.

 As in~\ref{general-koszul}, the condition for $n=1$ is always
trivial: it suffices to take $Q$ to be the identity matrix,
$L_{(m)}=N$, \ $L_{(i)}=0$ for $i<m$, and $M'_{(i)}=0$ for all~$i$.
 The condition for $m=0$ should be read to mean that any
$\Sigma$\+colored matrix $N$ with entries in $A_n$ can be decomposed
into a product $N=QL$, which is equivalent to $A$ being generated
by $A_1$.
 The condition for $m=1$ implies that $A$ has quadratic relations;
the converse implication was essentially proven in
Section~\ref{diagonal-secn} (see below for the details).

 Suppose that the matrix condition in~(2) holds.
 By Proposition~\ref{silly-filtrations-appx}.1, it suffices to show
that any element in $\Hom_\D(X,Y[n])$ with $X\in\M$, \ $Y\in\E_0$,
and $n\ge2$ can be presented as the composition of an element in
$\Hom_\D(X,Z[1])$ and an element in $\Hom_\D(Z,Y[n-1])$ with $Z\in\M$.
 As it was explained in Subsection~\ref{diagonal-is-quadratic}, one
can assume that $X\in\M_{(-\infty,-n]}$.
 Suppose that $X$ is a successive extension of objects
$X_0$,~\ds, $X_m$, where $X_i\in\E_{-n-m+i}$.
 This means that there exist distinguished triangles
$X_i\rarrow T_i\rarrow T_{i-1}\rarrow X_i[1]$ for all
$1\le i\le m$ such that $T_0=X_0$ and $T_m=X$.
 Then the elements in $\Hom_\D(X_{i-1},X_i[1])$ obtained as
the compositions $X_{i-1}\rarrow T_{i-1}\rarrow X_i[1]$ provide
matrices $M_{(i)}$ and the element in $\Hom_\D(X_m,Y[n])$ obtained as
the composition $X_m\rarrow X \rarrow Y[n]$ provides a matrix~$N$.
 The existence of the morphisms $T_i\rarrow X_{i+1}[1]$ and
$X\rarrow Y[n]$ implies the equations on the products of
consecutive matrices $M_{(i)}$ and~$N$.

 By assumption, it follows that there exist matrices $L_{(i)}$, \
$M'_{(i)}$, and $Q$ satisfying the equations of part~(2).
 Let $Z_{i-1}\rarrow Z_i[1]$, \ $i=1$,~\ds, $m$, and
$Z_m\rarrow Y[n-1]$ be morphisms corresponding to the matrices
$M'_{(i)}$ and $Q$, where $Z_i\in\E_{-n-m+i+1}$.
 Since $\Hom_\D(Z_i,Z_j[2])=0$ for $j-i>2$, it follows from
the equations $M'_{(i+1)}M'_{(i)}=0$ that there exists
a successive extension of objects $Z_i$ corresponding to
our morphisms $Z_{i-1}\rarrow Z_i$.
 In other words, there exist distinguished triangles 
$Z_i\rarrow S_i\rarrow S_{i-1}\rarrow Z_i[1]$ for all
$1\le i\le m$ such that $S_0=Z_0$ and the compositions
$Z_{i-1}\rarrow S_{i-1}\rarrow Z_i[1]$ are equal to our morphisms
$Z_{i-1}\rarrow Z_i[1]$.
 Set $Z=S_m$.
 Since $\Hom_\D(Z_i,Y[n])=0$ for $i\le m-2$, it follows from
the equation $QM'_{(m)}=0$ that there exists a morphism
$Z\rarrow Y[n-1]$ making the triangle $Z_m\rarrow Z\rarrow Y[n-1]$
commutative.
 Finally, let $X_i\rarrow Z_i[1]$ be the morphisms corresponding
to the matrices~$L_{(i)}$.
 Since $\Hom_\D(Z_i,X_j[2])=0$ for $j-i>1$ and
$\Hom_\D(X_i,Y[n+1])=0$ for $i\le m-2$, it follows from
the equations on the matrices $L_{(i)}$ that one can construct
a morphism $X\rarrow Z[1]$ making the triangle $X\rarrow Z[1]
\rarrow Y[n]$ commutative.

 Suppose that the decomposition condition in~(1) holds.
 Let $M_{(i)}$, \ $m\ge1$, and $N$ be $\Sigma$\+colored matrices
satisfying the equations $M_{(i+1)}M_i=0$ and $NM_{(m)}=0$.
 As above, starting from these matrices one can construct
a successive extension $X$ of objects $X_i\in\E_{-n-m+i}$ and
a morphism $X\rarrow Y[n]$.
 Decompose this morphism as $X\rarrow Z[1]\rarrow Y[n]$, where
$Z\in\M$.
 By the argument from Subsection~\ref{diagonal-is-quadratic},
one can assume $Z\in\M_{[-n-m+1,-n+1]}$.
 So the object $Z$ is a successive extension of objects
$Z_i\in\E_{-n-m+i+1}$, \ $i=0$,~\ds,~$m$.
 Hence we obtain morphisms $Z_{i-1}\rarrow Z_i[1]$.
 Since $\Hom_\D(X_i,Z_j[1])=0$ for $i>j$, one can obtain morphisms
$X_i\rarrow Z_i[1]$ from the morphism $X\rarrow Z$.
 The composition $Z_m\rarrow Z\rarrow Y[n-1]$ provides a morphism
$Z_m\rarrow Y[n-1]$.
 These morphisms define the desired matrices $M'_{(i)}$, \
$L_{(i)}$, and~$Q$.
\end{proof}

\subsection{Flat case} \label{flat-koszul}
 Let $R$ be a big graded ring over a set~$\Sigma$.
 A right $R$\+module $N$ is called \emph{flat} if the functor
$M\maps N\ot_RM$ is exact on the category of left $R$\+modules.
 Flat left $R$\+modules are defined in the similar way.

 Let $A=A_0\oplus A_1\oplus A_2\oplus\dsb$ be a big graded ring;
set $R=A_0$.
 Assume that $A$ is quadratic.
 Consider the quadratic coring $C$ over $R$ defined in the first
proof in Subsection~\ref{quadratic-realized}.
 Introduce the reduced cobar-complex
\begin{equation} \label{coring-cobar-complex}
 R\lrarrow C_+\lrarrow C_+\ot_R C_+\lrarrow
 C_+\ot_R C_+\ot_RC_+ \lrarrow\dsb,
\end{equation}
where $C_+=C/R$; this complex is bigraded with the (cohomological)
grading~$n$ by the number of tensor factors and the (internal)
grading~$i$ induced by the grading of~$C$.

\begin{thm}
 Let $A$ be a quadratic big ring such that either all the left
$A_0$\+modules $A_i$ are flat or all the right $A_0$\+modules
$A_i$ are flat.
 Then $A$ is Koszul if and only if
the cobar-complex~\eqref{coring-cobar-complex} has no cohomology
outside of the diagonal $i+n=0$.
\end{thm}

\begin{proof}
 First of all let us show that it suffices only to consider the case
when the right $A_0$\+modules $A_i$ are flat.

\begin{lem1}
 Opposite big graded rings $A$ and $A^\opp$ are Koszul simultaneously.
\end{lem1}

\begin{proof}
 It is clear that $A$ and $A^\opp$ are simultaneously quadratic;
assume that they are.
 Let $\G\supset\E_i$, \ $\E_0\supset\J$ be the exact category
with the additional data corresponding to~$A$.
 Consider the opposite exact category $\G'=\G^\opp$ with the full
subcategories $\E'_i=\E_{-i}^\opp$, the twist functor $X^\opp(-1)
= X(-1)^\opp$, and the full subcategory $\J'=\J^\opp\sub\E'_0$.
 Then the exact category $\G'\supset\E'_i$, \ $\E'_0\supset\J'$
corresponds to~$A^\opp$.
 Clearly, $\G$ and $\G'$ satisfy~\eqref{ext-diagonal} simultaneously.
\end{proof}

 It remains to notice that the coring $C^\opp$ over $R^\opp$ opposite
to the coring $C$ corresponds to the quadratic ring $A^\opp$ in
the same way as the coring $C$ corresponds to $A$, and
the cobar-complexes~\eqref{coring-cobar-complex} of $C$ and $C^\opp$
are isomorphic.

 The rest of the proof is based on the following Lemmas~2\+-4.

\begin{lem2}
 Let $C$ be a nonpositively graded coring over a big ring $R$ with
$C_0=R$.
 Then the embedding $\G\rarrow\H$ of the exact category $\G$ of graded
right $C$\+comodules, free and finitely generated as graded
$R$\+modules, into the exact category $\H$ of $R$\+flat graded right
$C$\+comodules induces an isomorphism of the groups $\Ext$.
\end{lem2}

\begin{proof}
 Clearly, it suffices to consider the categories $\G_{[0,m]}$ and
$\H_{[0,m]}$ of graded $C$\+comodules concentrated in the gradings
$[0,m]$.
 We will show that the exact functor $\G_{[0,m]}\rarrow\H_{[0,m]}$
satisfies the assumption of Lemmas~1--2 from~\ref{two-lemmas}.
 For this purpose we will use the big ring version of the 
Govorov--Lazard theorem that any flat module is a filtered inductive
limit of finitely generated free modules.
 More precisely, we will need the fact that any morphism from
a finitely presented $R$\+module (i.~e., the cokernel of a morphism
of finitely generated free $R$\+modules) to a flat $R$\+module
factorizes through a finitely generated free $R$\+module.
 One proves this for modules over big rings in exactly the same way
as in the case of conventional rings and modules;
see~\cite[No.~1.5--6]{Bour}.

 Let $T\rarrow X$ be an admissible epimorphism in $\H_{[0,m]}$
onto an object $X\in\G_{[0,m]}$; we would like to construct
an admissible epimorphism $Z\rarrow X$ in $\G_{[0,m]}$ factorizable
through the morphism $T\rarrow X$.
 Proceed by induction on~$m$, assuming that the morphisms
$Z_{\le m-1}\rarrow T_{\le m-1}\rarrow X_{\le m-1}$ have been
constructed already.
 Let $Z'''_m\rarrow T_m$ be a morphism into the grading component
$T_m$ from a right finitely presented $R$\+module $Z'''_m$ such
that the composition $Z'''_m\rarrow T_m\rarrow X_m$ is surjective
and the compositions $Z_{m-i}\rarrow T_{m-i}\rarrow T_m\ot_R C_{-i}$
of the grading components of the morphism $Z_{\le m-1}\rarrow
T_{\le m-1}$ with the comultiplication maps $T_{m-i}\rarrow
T_m\ot_R C_{m-i}$ factorize through the morphism
$Z'''_m\ot_RC_{-i}\rarrow T_m\ot_RC_{-i}$ for all $1\le i\le m$.
 One can even easily choose $Z'''_m$ to be a finitely generated
free $R$\+module.

 Suppose the above factorizations of morphisms to be fixed.
 Let $Z'''_m\rarrow Z''_m$ be a morphism of finitely
presented $R$\+modules through which the morphism
$Z'''_m\rarrow T_m$ factorizes such that the right $R$\+module
$Z_{\le m-1}\oplus Z''_m$ is a $C$\+comodule, i.~e.,
the coassociativity equations for the above maps
$Z_{m-i}\rarrow Z'''_m\ot_RC_{-i}$ together with the coaction maps
of $Z_{\le m-1}$ hold after taking the composition with
the epimorphisms $Z'''_m\ot_RC_{-i}\rarrow Z''\ot_RC_{-i}$.
 It is clearly possible to find such finitely presented $R$\+module
$Z''_m$, since there is only a finite number of equations on
the tensor products of its elements that have to be satisfied.
 Now the morphism $Z''_m\rarrow T_m$ between a finitely presented
and a flat $R$\+module factorizes through a finitely generated
free $R$\+module~$Z'_m$.
 This provides a right $C$\+comodule $Z'$, free and finitely
generated as a right $R$\+module, together with an epimorphism
$Z'\rarrow X$ factorizable through~$T$.
 It remains to add to the comodule $Z'$ a direct summand, free
and finitely generated over $R$ and concentrated in degree~$m$,
to obtain a morphism $Z\rarrow X$ whose kernel is also a free and
finitely generated $R$\+module.
\end{proof}

\begin{lem3}
 Let $C$ be a graded coring over a big ring $R$ such that
the components $C_i$ are flat right $R$\+modules.
 Then for any $R$\+projective graded right $C$\+comodule $X$ and
$R$\+flat graded right $C$\+comodule $Y$ the groups\/ $\Ext^n_\G(X,Y)$
in the exact category $\H$ of $R$\+flat graded right $C$\+comodules
are computed by the cobar-complex
$$
 \Hom_R(X,Y)\lrarrow\Hom_R(X\;Y\ot_RC_+)\lrarrow
 \Hom_R(X\;Y\ot_R C_+\ot_R C_+)\lrarrow\dsb,
$$
where $C_+=\Ker(C\to R)$, while the differentials are constructed
in terms of the comultiplication in $C$ and the right coactions of\/
$C$ in $X$ and\/ $Y$.
\end{lem3}

\begin{proof}
 Notice that our assertion does not depend on any positivity assumptions
on the grading.
 The complex 
$$
 Y\ot_R C\lrarrow Y\ot_R C_+\ot_R C\lrarrow Y\ot_R C_+\ot_R C_+\ot_RC
 \lrarrow\dsb
$$
is a right resolution of the graded right $C$\+comodule $Y$ in
the exact category~$\H$.
 Actually, this resolution is even split over~$R$.
 Its terms are $C$\+comodules $V\ot_RC$ coinduced from flat right
$R$\+modules $V$.
 One only has to check that such coinduced comodules are adjusted to
the functor $\Hom(X,{-})$ on the exact category~$\H$.

 Indeed, the exact functor $V\maps V\ot_RC$ from the category of flat
right $R$\+modules to the category of $R$\+flat right $C$\+comodules
satisfies the dual version of the assumption of Lemmas~1--2
from~\ref{two-lemmas}.
 To check this, consider an admissible monomorphism $V\ot_RC\rarrow N$
in the category~$\H$.
 Consider the morphism of $R$\+modules $V\ot_RC\rarrow V$ induced by
the counit of~$C$.
 Let $K$ be the fibered coproduct of the $R$\+modules $V$ and $N$
over $V\ot_RC$; then $K$ is a flat right $R$\+module and the morphism
$V\rarrow K$ is an admissible monomorphism of flat right
$R$\+modules.
 The morphism of right $R$\+modules $N\rarrow K$ induces a morphism
of right $C$\+comodules $N\rarrow K\ot_RC$, which forms a commutative
triangle with the morphism $V\ot_RC\rarrow N$ and the morphism
$V\ot_RC\rarrow K\ot_RC$ coinduced from the admissible monomorphism
$K\rarrow V$.

 Now since the exact triples of right $C$\+comodules coinduced from
exact triples of right $R$\+modules remain exact after applying
the functor $\Hom_\H(X,{-})$, one has $\Ext^n_\H(X\;V\ot_RC)=0$ for
$n>0$.
\end{proof}

\begin{lem4}
 Let $C'\rarrow C$ be a morphism of nonpositively graded corings
over a big ring $R$ such that the map $C'_{-i}\rarrow C_{-i}$
is an isomorphism for $i<m$.
 Let $X$ and $Y$ be finitely generated free right $R$\+modules
placed in the gradings\/ $0$ and\/~$m$, respectively.
 Then the maps\/ $\Ext^n_{\G'}(X,Y)\rarrow\Ext^n_\G(X,Y)$ between
the Ext groups in the categories of right $C'$- and $C$\+comodules,
free and finitely generated as graded $R$\+modules, are
isomorphisms for $n\ge3$.
\end{lem4}

\begin{proof}
 It is claimed that the groups $\Ext^n_\G(X,Y)$ for $n\ge3$ only
depend on the exact subcategories $\G_{[0,m-1]}$ and
$\G_{[1,m]}\sub\G$. 
 Indeed, any class in $\Ext^n_\G(X,Y)$ can be decomposed into
the product of classes in $\Ext^1_\G(X,Z)$ and $\Ext^{n-1}_\G(Z,Y)$,
where $Z\in\G_{[1,m-n+1]}$.
 The product of such two classes vanishes if and only if
the class in $\Ext^1_\G(X,Z)$ is the composition of a class
in $\Ext^1(X,Z')$ and a morphism $Z'\rarrow Z$, while
the composition of the same morphism with the class in
$\Ext^{n-1}(Z,Y)$ vanishes.
 Moreover, one can choose $Z'\in\G_{[1,m-n+2]}$ (see the proof
in~\ref{general-koszul}, ``Only if'' part).
\end{proof}

 Let us finish the proof of Theorem.
 ``If'': the diagonal cohomology of
the cobar-complex~\eqref{coring-cobar-complex} can be easily
identified with the components~$A_i$.
 Knowing that $A_i$ are flat right $R$\+modules and the cobar-complex
has no cohomology outside of the diagonal, one shows by induction
that the components $C_{-i}$ are flat $R$\+modules, too.
 Then it remains to apply Lemmas~2\+-3.

 ``Only if'': proceed by induction on the internal grading~$m$.
 If the cohomology of the cobar-complex in the internal grading
strictly above~$-m$ are concentrated on the diagonal, then
$C_{-i}$ are flat right $R$\+modules for $i<m$.
 Consider the graded coring $C'$ over $R$ with $C'_{-i}=C_{-i}$
for $i<m$ and $C'_{-i}=0$ for $i\ge m$.
 It follows from Lemmas~2\+-4 that $\Ext^n_\G(X,Y)$ for
$X\in\E_0$ and $Y\in\E_m$ for $n\ge 3$ can be computed in
terms of the cobar-complex of the coring~$C$.
 Since we assume that these groups are zero for $n\ne m$, 
it follows that the cobar-complex has no cohomology outside
of the diagonal in the internal grading~$-m$.
\end{proof}

\begin{cor}
 Let $k$ be a commutative ring and $A$ be a big graded ring with
a $k$\+algebra structure, i.~e.,  the components $A_{\sigma\tau;\.n}$
are $k$\+modules and the multiplications are $k$\+linear.
 Let $S$ be a (conventional associative, not necessarily commutative)
$k$\+algebra.
 Assume that either $S$ is $k$\+flat, or the components of $A$ are
flat $k$\+modules.
 Then the big graded ring $S\ot_k A$ is Koszul whenever the big
graded ring $A$ is Koszul.
 Assuming additionally that for any $k$\+module $M$ the vanishing of
$S\ot_kM$ implies the vanishing of $M$, the converse implication
also holds: $A$ is Koszul if $S\ot_k A$ is.
\end{cor}

\begin{proof}
 Let us first consider the case when $S$ is $k$\+flat.
 Then one can easily see that $S\ot_k A$ is quadratic whenever
$A$ is, and the converse also holds if $S$ is faithfully $k$\+flat.
 The nonpositively graded coring over $S\ot_k A_0$ quadratic dual
to $S\ot_k A$ is naturally isomorphic to $S\ot_k C$, where $C$
is the coring quadratic dual to~$A$.
 The category of graded right comodules over $S\ot_k C$ is naturally
identified with the category of graded right comodules over $C$
endowed with a right $S$\+modules structure which agrees with
the $k$\+module structure and commutes with the coaction of~$C$.

 Let $\H_C$ and $\H_{S\ot_k C}$ denote the exact categories of
$A_0$\+flat graded right $C$\+comodules and $S\ot_k A_0$\+flat
graded right $S\ot_k A_0$\+comodules.
 Then there is a natural exact functor $\H_{S\ot_kC}\rarrow\H_C$
of forgetting the action of $S$, which has an exact left adjoint
functor $\H_C\rarrow\H_{S\ot_kC}$ sending a comodule $X$ to
the comodule $S\ot_k X$.
 
 By~\cite[Lemma~2.1]{Psheaves}, there are natural isomorphisms
$$
 \Ext_{\H_{S\ot_kC}}^n(S\ot_kA_0\;S\ot_kA_0(i))\.\simeq\.
 \Ext_{\H_C}^n(A_0\;S\ot_kA_0(i))
$$
for all $n$ and~$i$.
 By Lemma~2 above, the left hand side of this isomorphism is isomorphic
to the similar group $\Ext$ in the category $\G_{S\ot_kC}$ of
of graded right $S\ot_kC$\+comodules, free and finitely generated
as $S\ot_k A_0$\+modules.

 By Lemma~2 from~\ref{two-lemmas}, which is applicable according to
the proof of Lemma~2 above, the right hand side is isomorphic to
$S\ot_k\Ext_{\G_C}^n(A_0\;A_0(i))$.
 Indeed, one readily checks that for any object $X$ of the exact
category $\G_C$ of right $C$\+comodules, free and finitely generated
as $A_0$\+modules, and any object $Y\in\H_C$ there is a natural
isomorphism $\Hom_{\H_C}(X\;S\ot_k Y)\simeq S\ot_k\Hom_{\H_C}(X,Y)$,
since there is a similar isomorphism for morphisms of graded
$A_0$\+modules and $S$ is $k$\+flat.
 This proves the desired assertion.

 The case when $A$ is $k$\+flat is easily dealt with using
the condition~(c) from Theorem~2 of Subsection~\ref{generalities}
below.
\end{proof}

 When $k$ is an Artinian local commutative ring and $S=f$ is its
residue field, a stronger assertion holds.
 Namely, if $A$ is quadratic and $A_n$ is $k$\+flat for $n=1$, $2$,
and~$3$, then Koszulity of $f\ot_k A$ implies Koszulity of $A$ and
$k$\+flatness of $A_n$ for all $n\ge1$.
 The proof is similar to that of~\cite[Theorem~7.3 of Chapter~6]{PP}.

\subsection{Generalities on flat Koszulity} \label{generalities}
 Recall the definition of a quadratic big ring from
Subsection~\ref{diagonal-is-quadratic} and the definitions
of a quadratic coring and the quadratic duality from the first proof
in Subsection~\ref{quadratic-realized}.

 Let $C=C_0\oplus C_{-1}\oplus C_{-2}\oplus\dsb$ be a graded
coring over a big ring $R$ such that $C_0=R$.
 Let $\G$ denote the exact category of graded right $C$\+comodules
that are free and finitely generated as graded right $R$\+modules,
and let $\H$ be the category of graded right $C$\+comodules that are
flat as graded right $R$\+modules.

 Let $\Ext^n_\G(R,R(m))$ and $\Ext^n_\H(R,R(m))$ denote
the $R$\+bimodules whose components $\Ext^n(R,R(m))_{\sigma\tau}$
are the groups $\Ext^n$ in the categories $\G$ and $\H$ between
the free right $R$\+modules with one generator placed in
the gradings $0$ and~$m$.

\begin{thm1}
 The following conditions are equivalent:
\begin{enumerate}
\renewcommand{\theenumi}{\alph{enumi}}
\item $\Ext^n_\G(R,R(m))=0$ for all\/ $n\ne m$ and\/
$\Ext^n_\G(R,R(n))$ is a flat right $R$\+module for all\/~$n$;
\item $\Ext^n_\H(R,R(m))=0$ for all\/ $n\ne m$ and\/
$\Ext^n_\H(R,R(n))$ is a flat right $R$\+module for all\/~$n$;
\item the cobar-complex~\eqref{coring-cobar-complex} has no
cohomology outside of the diagonal\/ $i+n=0$, and its cohomology
on this diagonal are flat right $R$\+modules;
\item the coring $C$ is quadratic, and for any $m\ge1$
the lattice of subbimodules of the $R$\+bimodule
$C_{-1}^{\ot_R\.m}$ generated by the subbimodules
$C_{-1}^{\ot j-1}\ot_R C_{-2}\ot_R C_{-1}^{\ot m-j-1}$, \
$1\le j\le m-1$, is distributive and the quotient bimodule
for any pair of embedded bimodules in this lattice is
a flat right $R$\+module.
\end{enumerate}
\end{thm1}

 Let $A=A_0\oplus A_1\oplus A_2\oplus\dsb$ be a big graded ring
with the zero-degree component $A_0=R$.
 Consider the reduced bar-complex
\begin{equation}  \label{big-ring-bar-complex}
 R\llarrow A_+\llarrow A_+\ot_R A_+\llarrow A_+\ot_R A_+\ot_R A_+
 \llarrow\dsb,
\end{equation}
where $A_+=A/R$; this complex is bigraded with the (cohomological)
grading~$n$ by the number of tensor factors and the (internal)
grading~$i$ induced by the grading of~$A$.

 For a left $A$\+module $M$ and a right $A$\+module $N$, let
$\Tor_n^A(N,M)$ denote the derived functor of tensor product
of $A$\+modules.
 The groups $\Tor_n^A(N,M)$ inherit the internal grading of
$M$, \ $N$, and $A$.

\begin{thm2}
 Assume that the components $A_i$ are flat right $R$\+modules.
 Then the following conditions are equivalent:
\begin{enumerate}
\renewcommand{\theenumi}{\alph{enumi}}
\item the graded abelian group $\Tor_n^A(N,M)$ is concentrated
in degree~$n$ for any finitely generated free (left and right)
$R$\+modules $M$ and $N$ considered as graded $A$\+modules
concentrated in degree\/~$0$;
\item the graded abelian group $\Tor_n^A(N,M)$ is concentrated
in degree~$n$ for any left $R$\+module $M$ and any flat right
$R$\+module $N$ considered as graded $A$\+modules concentrated
in degree\/~$0$;
\item the bar-complex~\eqref{big-ring-bar-complex} has no
cohomology outside of the diagonal\/ $i=n$;
\item the big ring $A$ is quadratic with the $R$\+bimodule of
quadratic relations $I\sub A_1\ot_R A_1$, and for any $m\ge1$
the lattice of subbimodules of the $R$\+bimodule $A_1^{\ot_R\.m}$
generated by the subbimodules $A_1^{\ot j-1}\ot_R I\ot_R
A_1^{\ot m-j-1}$, \ $1\le j\le m-1$, is distributive.
\end{enumerate}
\end{thm2}

 Suppose a quadratic big ring $A$ and a quadratic coring $C$ are
quadratic dual to each other.
 Consider the tensor products $A\ot_RC$ and $C\ot_RA$; they are
endowed with the differentials constructed as the composition
$A_i\ot_RC_{-j}\rarrow A_i\ot_R C_{-1}\ot_RC_{-j+1}\simeq
A_i\ot_R A_1\ot_R C_{-j+1}\rarrow A_{i+1}\ot_R C_{-j+1}$ of
the comultiplication and the multiplication maps, and analogously
for $C\ot_RA$.
 Define the internal grading on $A\ot_RC$ and $C\ot_RA$ by the rule
that the component $A_i\ot_R C_{-j}$ or $C_{-j}\ot_R A_i$ lives
in the grading $i+j$.

 The complexes $A\ot_RC$ and $C\ot_RA$ are called the \emph{Koszul
complexes} of the quadratic dual big ring $A$ and coring~$C$.

\begin{cor}
 The following conditions are equivalent:
\begin{enumerate}
\renewcommand{\theenumi}{\Alph{enumi}}
\item the big graded ring $A$ satisfies the equivalent 
conditions of Theorem~\textup{1};
\item the graded coring $C$ satisfies the equivalent conditions
of Theorem~\textup{2};
\item the components $A_i$ are flat right $R$\+modules and either
of the Koszul complexes $A\ot_RC$ or $C\ot_RA$ is exact in
the internal grading\/ $m\ge1$. 
\end{enumerate}
\end{cor}

 A big graded ring $A$ or a graded coring $C$ is called
\emph{right flat Koszul} if it satisfies the equivalent conditions
of Theorem~2 or Theorem~1, respectively.
 According to Corollary, quadratic dual quadratic big ring and
quadratic coring are Koszul simultaneously.
 Notice that our terminology is consistent: if $A$ is a big
graded ring whose components $A_i$ are flat right $A_0$\+modules,
then $A$ is right flat Koszul in the sense of the above definition
if and only if it is Koszul in the sense of the definition given
in the beginning of this section.
 To see this, it suffices to compare Theorem from~\ref{flat-koszul}
with the above Theorem~1(c) and use Corollary.

\begin{proof}[Proof of Theorems~\textup{1--2} and Corollary]
 It was explained in~\ref{flat-koszul} how to prove the equivalence
of (a), (b), and~(c) in Theorem~1.
 Proving the equivalences
(a)$\thickspace\Longleftrightarrow\thickspace$(c)
and (b)$\thickspace\Longleftrightarrow\thickspace$(c)
in Theorem~2 is easy, since one can compute the $\Tor$ in
terms of the bar-complex.
 To prove the equivalence of (c) and~(d) in both theorems, one can
use Lemma~1 and the big ring version of Lemma~2(a)
from~\cite[Subsection~11.4.3]{Psemi}.
 This also allows to prove the equivalence of (A) and~(B) in
Corollary and deduce~(C) from (A-B).
 Finally, to pass from~(C) of Corollary to Theorem~2(a) one simply
uses the Koszul complex as a resolution of free (left or right)
$R$\+modules considered as $A$\+modules concentrated in degree~$0$.
\end{proof}

\begin{rem1}
 When one is thinking of the conditions~(a-b) of Theorem~2, one
keeps in mind that for any nonnegatively graded ring $A$ whose
components are flat right $R$\+modules, the graded abelian groups
$\Tor^A_n(N,M)$ are concentrated in degrees~$\ge n$.
 Without the flatness condition, this is no longer true.
 For a counterexample, it suffices to consider a field~$k$, the ring
of polynomials in one variable $R=k[x]$, and the graded commutative
ring $A=k[x,y]/(xy)$ with $\deg x=0$ and $\deg y=1$.
 Another example is the graded ring $B$ over $R=\Z$ with $B_0=\Z$,
\ $B_1=\Q/\Z$, and $B_n=0$ for $n\ge 2$.
 The graded ring $B$ is Koszul (see Example in the beginning
of Section~\ref{koszul-rings-secn}).
 To compute $\Tor^B_*(\Z,\Z)$ one can replace $B$ with its flat
graded DG\+algebra resolution over $\Z$ and write down the bar-complex
of the resolution.
 Hence one finds that the graded group $\Tor^B_n(\Z,\Z)$ vanishes for
even $n\ge2$ and is isomorphic to the group $\Q/\Z$ placed
in the degree $(n+1)/2$ for odd $n\ge1$.
 Of course, one can still use the reduced
bar-complex~\eqref{big-ring-bar-complex} to compute the $\Tor$ over
$A$ \emph{relative to}~$R$ in the nonflat case, so the relative
$\Tor$ is concentrated in the usual degrees.
\end{rem1}

\begin{rem2}
 Let $A=R\op A_1\op A_2\op\dsb$ be a nonnegatively graded ring and
$S\rarrow R$ be a morphism of rings such that $A_i$ are flat right
$R$\+modules and flat right $S$\+modules for all $i\ge1$.
 Then it follows from our results that the graded rings $A$ and
$B=S\op A_1\op A_2\op\dsb$ satisfy the conditions of Theorem~2
simultaneously.
 This can be proved directly in several ways.
 In particular, one can use the spectral sequence
$E^2_{pq}=\Tor_p^A(\Tor_q^B(S,A),R)\implies\Tor_{p+q}^B(S,R)$
corresponding to the morphism of rings $B\rarrow A$.
 There is also a simple lattice-theoretical argument.
\end{rem2}

\Section{Nonfiltered Exact Categories}  \label{nonfiltered-secn}

 For any big graded ring $A=A_0\oplus A_1\oplus A_2\oplus\dsb$,
denote by $\qu A$ the ``quadratic part'' of $A$, i.~e.,
the quadratic ring for which there is a natural morphism of
big graded rings $\qu A\rarrow A$ that is an isomorphism in
degrees $n=0$ and~$1$, and a monomorphism in degree $n=2$.

 The following set of stronger assumptions on the morphism
$\qu A\rarrow A$ and the big graded ring $\qu A$ will play
the key role in this section:
\begin{align}
 &\begin{array}{l}
 \text{the natural morphism $\qu A\rarrow A$ is an isomorphism} \\
 \text{in degree $n=2$ and a monomorphism in degree~$n=3$;}
 \end{array} \label{two-three-degree} \\
 &\begin{array}{l}
 \text{the quadratic big ring $\qu A$ is Koszul.}
 \end{array}\label{koszul-qu-part}
\end{align}

 Let $\E$ be an exact category, $\E_0$ be a small additive category,
and $\Phi_0\:\E_0\rarrow\E$ be a fully faithful additive functor.
 Consider the category $\E_0$ as an exact category with the trivial
exact category structure and use the construction of Example
from Section~\ref{associated-graded-secn} to obtain an exact
category $\F$ with full subcategories $\E_i$ and a twist functor
$X\maps X(1)$.
 Let $\Phi\:\F\rarrow\E$ denote the forgetful functor.

 Let $\J\sub\E_0$ be a full subcategory such that every object of
$\E_0$ is a finite direct sum of objects from~$\J$.
 Consider the big graded ring $A=(\Ext^n_\E(\Phi(X),\Phi(Y)))_
{Y,X\in\J;\.n\ge0}$ over the set $\Ob\J$.

\begin{prop}
  Assume that the big graded ring $A$
satisfies~\textup{(\ref{two-three-degree}--\ref{koszul-qu-part})}.
 Then the natural maps\/ $\Ext^n_\F(X,Y(m))\rarrow\Ext^n_\F(X,Y(m+1))$
are isomorphisms for all $X$, $Y\in\E_0$ and\/ $n\le m$, and
the big graded ring\/ $(\Ext^n_\F(X,Y(n)))_{Y,X\in\J;\.n\ge0}$
is isomorphic to~$\qu A$.
\end{prop}

\begin{proof}
 The case $n=0$ is clear.
 According to part~(2) of Theorem from
Section~\ref{filtered-exact-secn}, the morphisms $\Ext^n_\F(X,Y(m))
\rarrow\Ext^n_\E(\Phi(X),\Phi(Y))$ are isomorphisms for $X$, $Y\in\E_0$
and $m\ge n=1$, and monomorphisms for $m\ge n=2$.
 The former observation proves both assertions of Proposition
for $n=1$.

 By the result of Subsection~\ref{diagonal-is-quadratic}, one has
$\Ext^n_\F(X,Y(m))=0$ for $n>m$ and the big graded ring
$((\Ext^n_\F(X,Y(n)))_{Y,X\in\J;\.n\ge0}$ is quadratic.
 This establishes the second assertion in all degrees.
 We have not yet used the
assumptions~\textup{(\ref{two-three-degree}--\ref{koszul-qu-part})}
so far.

 Since by the assumption~\eqref{two-three-degree} the component $A_2$
is multiplicatively generated by $A_1$, the first assertion of
Proposition for $n=2$ also follows.

 Now consider the exact category $\G$ constructed in
Section~\ref{associated-graded-secn}.
 From the long exact sequence~\eqref{filtered-graded-sequence}
we see that $(\Ext^n_\G(X,Y(n)))_{Y,X\in\J;\.n\ge0}\simeq \qu A$,
\ $\Ext^1_\G(X,Y(m))=0$ for $m\ge2$, and $\Ext^2_\G(X,Y(3))=0$.

 Let us prove by induction that $\Ext^2_\G(X,Y(m))=0$ for $m\ge4$.
 Assume that $\Ext^2_\G(X,Y(j))=0$ for $3\le j\le m-1$.
 Let $\G'$ denote the exact category from
Subsection~\ref{quadratic-realized} corresponding to the quadratic
ring $\qu A$.
 Then the exact subcategories $\G_{[0,m-1]}$ and $\G'_{[0,m-1]}$
are naturally equivalent.
 Since the ring $\qu A$ is Koszul, one has $\Ext^n_{\G'}(X,Y(i))=0$
for $n\ne i$.
 Hence $\Ext^3_\G(X,Y(j))=0$ for $4\le j\le m-1$.
 
 Consequently, the map $\Ext^3_\F(X,Y(3))\rarrow
\Ext^3_\F(X,Y(m-1))$ is an isomorphism.
 By the assumption~\eqref{two-three-degree}, the map
$\Ext^3_\F(X,Y(3))\rarrow\Ext^3_\E(\Phi(X),\Phi(Y))$
is a monomorphism.
 It follows that the map $\Ext^3_\F(X,Y(m-1))\rarrow
\Ext^3_\F(X,Y(m))$ is a monomorphism, too.
 Thus $\Ext^2_\G(X,Y(m))=0$.

 Now since $\qu A$ is Koszul, we have $\Ext^n_\G(X,Y(m))=0$ for
all $n\ne m$, hence the map $\Ext^n_\F(X,Y(m-1))\rarrow
\Ext^n_\F(X,Y(m))$ is an isomorphism for all $n<m$.
\end{proof}

\begin{lem}
 In the above setting, assume that every object of $\E$ can be
obtained from objects of $\Phi_0(\E_0)$ by iterated extensions.
 Then the natural maps\/ $\varinjlim_m\Ext^n_\F(X,Y(m))\allowbreak
\rarrow\Ext^n_\E(\Phi(X),\Phi(Y))$ are isomorphisms for all
$X$, $Y\in\F$ and\/ $n\ge0$.
\end{lem}

\begin{proof}
 The cases $n=0$ and~$1$ are easy and do not depend on
the assumption of Lemma.
 To check surjectivity of our maps for all~$n$, it suffices
to notice that the functor $\Phi$ is surjective on objects.
 One can use Corollary~\ref{exact-triangulated}.1 to prove both
surjectivity and injectivity.
 Alternatively, use Lemma~2 from~\ref{two-lemmas} for
the functor $\Phi$, which satisfies its assumption.
 This argument does not depend on the assumption that the exact
category structure on $\E_0$ is trivial.
 When it is, one can additionally notice that all the maps
$\Ext^n_\F(X,Y(m))\rarrow\Ext^n_\F(X,Y(m+1))$ constituting our
inductive systems are injective for $n\le 2$.
\end{proof}

\begin{ex}
 The following example was suggested to the author by A.~Beilinson.
 Consider the case when the functor $\Phi_0$ is an equivalence of
additive categories.
 Then the natural maps $\Ext^n_\F(X,Y(m))\rarrow
\Ext^n_\E(\Phi(X),\Phi(Y))$ are isomorphisms for all $X$, $Y\in\E_0$
and $n\le m$.
 Indeed, present any class in $\Ext^n_\E(\Phi(X),\Phi(Y))$ as
the composition of a chain of classes in $\Ext^1_\E(T_{i-1},T_i)$
with $T_i\in\E$, \ $Z_0=\Phi(X)$, and $T_n=\Phi(Y)$.
 For all $1\le i\le n-1$, choose objects $Z_i\in\E_i$ so that
$T_i=\Phi(Z_i)$; set $Z_0=X$ and $Z_n=Y(m)$.
 Clearly, there are unique classes in $\Ext^1_\F(Z_{i-1},Z_i)$
lifting the given classes in $\Ext^1_\E(T_{i-1},T_i)$.
 The product of these classes provides a class in
$\Ext^n_\F(X,Y(m))$.
 One checks that this construction defines a map $\Ext^n_\E(\Phi(X),
\Phi(Y))\rarrow\Ext^n_\F(X,Y(m))$ inverse to the natural map that
we are interested in.
 Using the construction of the exact category $\G$
from Section~\ref{associated-graded-secn} and the long exact
sequence~\eqref{filtered-graded-sequence} (cf.~\ref{main-theorem}),
one can conclude that the big graded ring
$A=(\Ext^n_{\E_0}(X,Y))_{Y,X\in\J\;n\ge0}$ is Koszul.
\end{ex}

 Let $\E$ be a small exact category and $\J\sub\E$ be a full 
subcategory such that every object of $\E$ can be obtained
from objects of $\J$ by iterated extensions.
 Consider the big graded ring
$A=(\Ext^n_\E(X,Y))_{Y,X\in\J;\.n\ge0}$ over the set~$\Ob\J$.

 The following theorem is a far-reaching generalization of
the main result of~\cite{PV}.

\begin{thm}
 Assume that the big graded ring $A$ satisfies the
assumptions~\textup{(\ref{two-three-degree}--\ref{koszul-qu-part})}.
 Then $A$ is quadratic (and consequently, Koszul).
\end{thm}

\begin{proof}
 Consider the full additive subcategory $\E_0\sub\E$ consisting of
the finite direct sums of objects from $\J$, and apply Proposition
and Lemma.
\end{proof}

 Let $\D$ be a triangulated category and $\E\sub\D$ be a full
exact subcategory closed under extensions
(see~\ref{exact-triangulated}).
 Let $\J\sub\E$ be a full subcategory such that every object
of $\E$ can be obtained from objects of $\J$ by iterated extensions.

 In addition to the above big graded ring $A$, consider
the big graded ring $B=(\Hom_\D(X,Y[n]))_{Y,X\in\J;\.n\ge0}$
over the same set $\Ob\J$.

 The following corollary generalizes the results
of~\cite[Section~5]{Pbogom}.

\begin{cor}
 Assume that the big graded ring $B$
satisfies~\textup{(\ref{two-three-degree}--\ref{koszul-qu-part})}.
 Then the natural morphism $A\rarrow B$ induces an isomorphism
of big graded rings $A\simeq \qu B$.
 In particular, if the big graded ring $B$ is Koszul, then
$A\simeq B$ and the natural morphisms $\Ext^n_\E(X,Y)\rarrow
\Hom_\D(X,Y[n])$ are isomorphisms for all $X$, $Y\in\E$ and
$n\ge 0$.
\end{cor}

\begin{proof}
 By Corollary~\ref{exact-triangulated}.2, the morphism $A\rarrow B$
is an isomorphism in the degrees~$0$ and~$1$ and a monomorphism in
the degree~$2$.
 It follows that $\qu A\simeq \qu B$ and $A$
satisfies~\textup{(\ref{two-three-degree}--\ref{koszul-qu-part})}
whenever $B$ does.
 If this is the case, $A$ is quadratic by Theorem, hence
$A\simeq \qu B$.
 If $B$ is also quadratic, then $A\simeq B$ and induction by
the number of iterated extensions proves the last assertion
of Corollary.
\end{proof}

 One can easily extend Lemma, Theorem, and Corollary to
the situation when every object of $\E$ is a direct summand of
an object obtained from objects of $\J$ (resp.\ $\E_0$) by
iterated extensions.
 It suffices to use Corollary~\ref{exact-triangulated}.3.

\Section{Conclusions and Epilogue}

\subsection{Main theorem} \label{main-theorem}
 Let $\D$ be a triangulated category and $\E_i\sub\D$ be its full
subcategories, closed under extensions and such that
\begin{equation}
 \Hom_\D(X,Y[n])=0 \quad\text{for $X\in\E_i$, \ $Y\in\E_j$, \
$n\ge-1$, and $n>j-i$.}
\end{equation}
 Let $\E$ be an exact category and $\Phi\:\D\rarrow\D(\E)$ be
a triangulated functor mapping $\E_i$ into~$\E$.
 Assume that
\begin{equation} 
\begin{array}{l}
\text{the induced morphisms $\Hom_\D(X,Y[n])\rarrow
\Ext_\E^n(X,Y)$} \\
\text{are isomorphisms for all $X\in\E_i$, \ $Y\in\E_j$, \ $n\ge-1$,
and $n\le j-i$.}
\end{array}
\end{equation}

 Finally, assume that there exists a triangulated autoequivalence
$X\maps X(1)$ on $\D$ such that $\E_i(1)=\E_{i+1}$ and there is
a functorial isomorphism $\Phi(X(1))\simeq\Phi(X)$ for all $X\in\D$.

 Let $\M$ be the minimal full subcategory of $\D$, containing all
$\E_i$ and closed under extensions; by \cite{Dy}
or~\ref{exact-triangulated}, \ $\M$ has a natural structure of exact
category.

 Let $\J\sub\E_0$ be a full subcategory such that every object of
$\E_0$ is a finite direct sum of objects of~$\J$.
 Consider the big graded ring $A=(\Ext^n_\E(\Phi(X),\Phi(Y)))_
{Y,X\in\J;\.n\ge0}$.

\begin{thm}
 The natural morphisms $\Ext^n_\M(X,Y)\rarrow\Hom_\D(X,Y[n])$ are
isomorphisms for all $X$, $Y\in\M$ and $n\ge0$ if and only if
the big graded ring $A$ is Koszul.
\end{thm}

\begin{proof}
 By part~(1) of Theorem from Section~\ref{filtered-exact-secn},
the exact category $\M$ is equivalent to the exact category $\F$
constructed in Example from Section~\ref{associated-graded-secn}
for the trivial exact category structure on~$\E_0$.
 Now the ``if'' part follows from Proposition from
Section~\ref{nonfiltered-secn}.
 To prove ``only if'', it suffices to consider the exact category
$\G$ constructed in Section~\ref{associated-graded-secn}.
 It follows from the long exact
sequence~\eqref{filtered-graded-sequence} that
$\Ext^n_\G(X,Y)=0$ for all $X\in\E_i$, \ $Y\in\E_j$, and
$n\ne j-i$, and that the big graded ring $(\Ext^n_\G(X,Y(n)))_
{Y,X\in\J;\.n\ge0}$ can be identified with~$A$.
 Thus $A$ is Koszul.
\end{proof}

\subsection{Main conjecture}  \label{main-conjecture}
 Let $k=\Z$, $\.\Q$, $\.\Z[m^{-1}]$, or~$\Z/m$, $\.m\ge2$, be
a coefficient ring and $K$ be a field.
 In the remaining part of this paper we discuss conjectural
properties of certain subcategories of the triangulated category
$\D\M(K,k)$ of motives over $K$ with coefficients in~$k$.

 A few words about the definition of $\D\M(K,k)$ are due.
 We are not in a position to discuss various definitions of
the triangulated category of motives existing in the literature,
so we simply presume that the category we are considering is
the ``right'' triangulated category.
 Voevodsky's definition~\cite{Voev} is certainly a good
approximation, as it is very natural and has many properties that
are important for us.
 However, their proofs sometimes depend on various undesirable
assumptions, such as that the field $K$ is perfect, has finite
\'etale dimension over~$k$, or has resolution of singularities.
 Other known definitions have some other properties that
are also important for us~\cite{GL}.
 In the course of our discussion, we refer to various results as
to indicators of the properties that the ``right'' category is
supposed to have.

 Let $M/K$ be a Galois field extension.
 Consider the full triangulated subcategory $\D\sub\D\M(K,k)$
generated by the Artin--Tate motives $k[L](i)$, where $k[L]$ are
the motives of finite field extensions $L/K$ contained in $M$,
and $X\maps X(i)$, \ $i\in\Z$ denotes the Tate twist.
 Let $\M$ be the minimal full subcategory of $\D$, containing
the objects $k[L](i)$ and closed under extensions.

\begin{conj}
 Any morphism $X\rarrow Y[n]$ of degree~$n\ge2$ in $\D$ between
two objects $X$, $Y\in\M$ can be presented as the composition of
a chain of morphisms $Z_{i-1}\rarrow Z_i[1]$ with $Z_i\in\M$, \
$Z_0=X$, and $Z_n=Y$.
\end{conj}
 
 By Propositions~2\+-4 from Appendix~\ref{silly-filtrations-appx},
this conjecture can be equivalently restated as follows:
$$
 \D=\textstyle\bigcup_{a\le b\in\Z}\M[-b]*\dsb*\M[-a],
$$
where $*$ denotes the class of all extensions
of objects from two given classes in a triangulated
category~\cite[1.3.9-10]{BBD}.
 Iterated extensions of this form generalize silly filtrations on
complexes; thus the above conjecture can be called the \emph{silly
filtration conjecture for Artin--Tate motives}.

\subsection{Elementary construction for finite coefficients}
\label{elementary-description}
 Consider the case when $k=\Z/m$ and $\chr K$ is prime to~$m$.
 In this case, there is the \'etale realization functor $\Phi$
\cite[Subsection~3.3]{Voev} (see also~\cite{Iv}) from $\D\M(K,k)$
to the derived category $\D(\E)$ of the abelian category $\E$ of
discrete $G_K$\+modules over $\Z/m$, where $G_K$ is the absolute
Galois group of~$K$.
 The functor $\Phi$ takes the object $k[L](i)$ to the $G_K$\+module
$\mu_m^{\ot i}[G_K/G_L]$, where $\mu_m\sub\oK$ is the module of
$m$\+roots of unity.

 The Milnor--Bloch--Kato conjecture implies~\cite{VS,GL2}
the Beilinson--Lichtenbaum formulas
\begin{equation} \label{AT-BL-formulas}
 \begin{aligned}
    \Hom_\D(k[L'](i),\.k[L''](j)[n])&\simeq
                       \Ext^n_\E(\Phi(k[L'](i)),\.\Phi(k[L''](j))),
      \quad\! \text{for $\,n\le j-i$;} \\
    \Hom_\D(k[L'](i),\.k[L''](j)[n])&=0,
      \quad\! \text{otherwise.}
 \end{aligned}
\end{equation}

 Assuming~\eqref{AT-BL-formulas}, by part~(1) of Theorem from
Section~\ref{filtered-exact-secn} the exact subcategory $\M\sub\D$
is equivalent to the exact category $\F$ of filtered discrete
$G_K$\+modules $(N,F)$ over~$\Z/m$ with a finite decreasing
filtration $F$ such that the $G_K$\+module $F^iN/F^{i+1}N$ is
isomorphic to a direct sum of copies of ``cyclotomic-permutational''
$G_K$\+modules $\mu_m^{\ot i}[G_K/G_L]$, where $L/K$ are finite
extensions, $K\sub L\sub M$.
 By the results of Appendix~\ref{categories-of-morphisms-appx},
this equivalence can be extended to a triangulated functor
$\Theta\:\D^b(\F)\rarrow\D$.

 The main conjecture from~\ref{main-conjecture} holds for $k$ and
$K\sub M$ if and only if the functor $\Theta$ is an equivalence
of triangulated categories.
 This provides an elementary description of the triangulated
category $\D$ in terms of the Galois group $G_K$, solving
(for Artin--Tate motives over a field) the problem posed
in~\cite[Subsection~5.10.D(vi)]{Beil}.

\begin{rem}
 Notice that even when $M=\oK$ is algebraically/separably closed,
the category of ``cyclotomic-permutational'' $G_K$\+modules is
still substantially different from the category of all discrete
$G_K$\+modules over~$k$, as there are many discrete $G_K$\+modules
that are not direct summands of permutational ones.
 Ignoring this difference, one could consider the exact category
$\F'$ of all finitely filtered discrete $G_K$\+modules over~$k$, with
exact triples that become split after passing to associated graded
$G_K$\+modules.
 For the exact category $\F'$, the Beilinson--Lichtenbaum formulas
always hold (see Example from Section~\ref{nonfiltered-secn}) and no
additional hypotheses are needed.
\end{rem}

\subsection{Diagonal cohomology}  \label{supporting-evidence}
 The supporting evidence for Conjecture from~\ref{main-conjecture}
that we possess is based on the properties of the big graded ring
\begin{equation}  \label{AT-big-ring}
 A=(\Hom_\D(k[L'],k[L''](n)[n]))_{K\sub L''\;L'\sub M}
\end{equation}
over the set of all (isomorphism classes over~$K$ of) intermediate
fields $K\sub L\sub M$ finite over~$K$.
 When the extension $M/K$ is finite, one can equivalently consider
the conventional graded ring~\eqref{AT-algebra}.

 By~\cite[Theorem~3.4]{VS}, the $k$\+module $\Hom_\D(k[L'],k[L'']
(n)[n])$ is isomorphic to the direct sum of the degree\+$n$
components of Milnor K\+theory rings $\KM_n(N)$ over the field
direct summands $N$ of the tensor product $L'\ot_K L''$.
 The multiplication in $A$ is described in terms of
the multiplications in the Milnor rings and the maps of inclusion
and transfer between the Milnor rings induced by morphisms of
fields over~$K$.

 According to Theorem from Subsection~\ref{diagonal-is-quadratic},
it follows from Conjecture that the (big) graded ring $A$ is
quadratic.
 In the case of Tate motives, i.~e., when $K=M$, this is obviously
so, since the Milnor K\+theory ring is quadratic by definition.
 In the general case of Artin--Tate motives, it is not difficult
to see that the (big) graded ring $A$ is multiplicatively generated
by $A_1$ over $A_0$.
 In fact, it is only essential that the set of all fields $L/K$
entering into the definition of $A$ is closed under the operation
of taking the composite extensions.
 For further evidence, see \ref{koszul-cases}
and~\ref{cyclic-extension}, and also~\cite{PV2}.

\subsection{Koszul cases}  \label{koszul-cases}
 Let us consider the three ``Koszul'' cases~(i\+iii)
from Subsection~\ref{introd-list-of-situations} one by one.

 When $k=\Q$ and $\chr K=p\ne0$, the Beilinson--Parshin
conjecture~\cite{Geis} predicts that $\Hom_\D(k[L'],k[L''](i)[n])=0$
for $i\ne n$.
 By part~(1) of Theorem from Subsection~\ref{triangulated-koszul},
it follows that Conjecture from~\ref{main-conjecture} holds if and only
if the (big) graded ring $A$ is Koszul in the sense of the definition
given in the beginning of Section~\ref{koszul-rings-secn}.

 Similarly, when $k=\Z/p^r$ and $\chr K=p$, it is known~\cite{GL}
that $\Hom_\D(k[L'],k[L'']\abk(i)[n])=0$ for $i\ne n$.
 Again, it follows that Conjecture holds if and only if the (big)
graded ring $A$ is Koszul in the sense of the beginning of
Section~\ref{koszul-rings-secn}.

 Finally, when $k=\Z/m$ and $K$ contains a primitive $m$\+root of
unity, the \'etale realization functor $\Phi$ transforms
the twist functor $X\maps X(1)$ on $\D$ to the identity functor on
$\D(\E)$, so the conditions of~\ref{main-theorem} are satisfied.
 Assuming~\eqref{AT-BL-formulas}, by Theorem from~\ref{main-theorem}
(and the last assertion of Corollary~\ref{exact-triangulated}.2)
it follows that Conjecture holds if and only if the (big) graded
ring~$A$ is Koszul in the same sense.

 In any of the above three cases, in order to interpret the Koszulity
condition, one can replace the component $A_0$ of the graded
ring~\eqref{AT-algebra} with $A'_0=k$, or replace the component
$A_0$ of the big graded ring~\eqref{AT-big-ring} with
$A'_{L'',L';\,0}=k$ for $L'=L''$ and $0$~otherwise.
 The Koszul property of the (big) graded ring $A$ is equivalent
to the same property of the (big) graded ring $A'=A'_0\op A_1\op A_2
\op\dsb$, and the latter (big) graded ring is flat over its
zero-degree component $A'_0$, so the characterization of Koszulity
from Subsections~\ref{flat-koszul}\+-\ref{generalities} applies.
 When $k$ is a field, this further simplifies to the classical
Koszul property of $k$\+algebras~\cite{PP,PV,Pbogom}.

 By Corollary from~\ref{flat-koszul}, the Koszulity conditions for
coefficients $k=\Z/p^r$ or $\Z/l^r$ in the cases (ii) and~(iii) are
equivalent to the Koszulity conditions for the coefficients $\Z/p$
or~$\Z/l$, where $p=\chr K$ or $l\ne\chr K$ are prime numbers.

\subsection{Tate motives with finite coefficients}
\label{PV-interpretation}
 Consider the case when $k=\Z/l$ and the field $K=M$ contains
a primitive $l$\+root of unity.

 It was proven in~\cite{PV} that Koszulity of the Milnor algebra
$\KM(K)/l$ for all finite extensions $K$ of a given field $F$
implies the Milnor--Bloch--Kato conjecture for the field $F$,
assuming that the latter conjecture holds in small degrees
(the norm residue homomorphism in an isomorphism in the degree
$n=2$ and a monomorphism in the degree $n=3$).
 Furthermore, Koszulity of $\KM(K)/l$ together with
the Beilinson--Lichtenbaum conjecture implies Conjecture
from~\ref{main-conjecture} for Tate motives with coefficients
$\Z/l$ over~$K$, as explained in~\ref{koszul-cases}.

 Conversely, assuming the Milnor--Bloch--Kato conjecture and
Conjecture from~\ref{main-conjecture} for Tate motives with
coefficients in $\Z/l$, Koszulity of the algebra $\KM(K)/l$ follows
due to~\cite{VS,GL2} and~\ref{koszul-cases}.
 So Koszulity of the Milnor algebra $\KM(K)/l$ together with
the low-degree part of the Milnor--Bloch--Kato conjecture are
equivalent to the full Milnor--Bloch--Kato conjecture together
with Conjecture from~\ref{main-conjecture} for Tate motives
with finite coefficients~$\Z/l$.

 When all these conjectures hold, the triangulated category $\D$
of Tate motives with coefficients $\Z/l$ over $K$ is simply described
as the derived category $\D^b(\F)$ of the exact category $\F$ of
finitely filtered discrete $G_K$\+modules $(N,F)$ over~$\Z/l$ such that
$G_K$ acts trivially on the quotient modules $F^iN/F^{i+1}N$.

 Thus the silly filtration conjecture provides a motivic
interpretation~\cite{Pober} of the Koszulity conjecture from~\cite{PV}.

\subsection{Torsion Tate motives}
 Let $K$ be a field and $m\ge2$ be an integer prime to $\chr K$.
 Imagine that there is a $\Z/m$\+linear triangulated category $\D'$
generated by objects $\Z/d\.(i)$ for all $d$ dividing~$m$ and $i\in\Z$.
 There is a shift functor $X\maps X(1)$ on $\D'$ transforming
$\Z/d\.(i)$ to $\Z/d\.(i+1)$, and the triangulated subcategory in $\D'$
generated by $\Z/d\.(0)$ is identified with the bounded derived
category of $\Z/m$\+modules.
 The conventional triangulated category $\D$ of Tate motives with
coefficients $\Z/m$ over $K$ is embedded into $\D'$ as the triangulated
subcategory generated by $\Z/m(i)$.
 One has $\Hom_{\D'}(\Z/d'(i),\Z/d''(j)[n]) = 0$ for all $n<0$ and
$i$, $j\in\Z$.

 Finally, there is the \'etale realization functor $\Phi\:\D'\rarrow
\D(\E)$, where $\E$ denotes the abelian category of discrete
$G_K$\+modules over $\Z/m$.
 This functor agrees with the \'etale realization functor for
$\D\sub\D'$ (so the groups $\Hom_{\D'}(\Z/m(i),\Z/m(j))$ are
described by the Beilinson--Lichtenbaum formulas) and its
restriction to the subcategory generated by $\Z/d\.(0)$ coincides
with the obvious functor (of ``trivial $G_K$\+action'') from
the bounded derived category of $\Z/m$\+modules to $\D(\E)$.

 Then one can easily prove by induction on the cohomological
degree~$n$ that the subcategories $\E'_i\sub\D'$ consisting of
finite direct sums of the objects $\Z/d(i)$ satisfy the conditions
(\ref{triang-vanish-area}--\ref{triang-bl-area}).
 Thus the minimal full subcategory $\M'\sub\D'$ containing all
$\Z/d\.(i)$ and closed under extensions is equivalent to the exact
category $\F'$ of finitely filtered discrete $G_K$\+modules
$(N,F)$ over $\Z/m$ such that the quotient $G_K$\+modules
$F^iN/F^{i+1}N\ot_{\Z/m}\mu_m^{\ot -i}$ are finitely generated
over $\Z/m$ with a trivial action of~$G_K$.

 Let $\M$ denote the minimal full subcategory of $\D$ containing
the objects $\Z/m(i)$ and closed under extensions.
 It would be interesting to know under what assumptions all
the maps $\Ext^n_\M(X,Y)\rarrow\Ext^n_{\M'}(X,Y)$ are isomorphisms,
or equivalently, all the maps $\Ext^n_{\M'}(X,Y)\rarrow
\Hom_{\D'}(X,Y[n])$ are isomorphisms.

\subsection{Cyclic extension of prime degree}  \label{cyclic-extension}
 Let $k=\Z/l$ and $M/K$ be a cyclic extension of degree~$l$ of
fields containing a primitive $l$\+root of unity.
 Assume that the Milnor--Bloch--Kato conjecture holds for fields
$M$ and $K$ and coefficients~$k$.
 In this case, the graded ring~\eqref{AT-algebra} has the form
\begin{align*}
 A&=\textstyle\bigoplus_n\Ext^n_\E(k\op k[G_K/G_M]\;k\op k[G_K/G_M]) \\
  &\simeq H^*(G_K,\Z/l)\oplus H^*(G_M,\Z/l)\oplus H^*(G_M,\Z/l)\oplus
H^*(G_M,\Z/l)[G_K/G_M],
\end{align*}
where $\E$ is the abelian category of discrete $G_K$\+modules
over~$\Z/l$.

 There is a conjecture~\cite[Conjecture~16]{Pdivis} that
$H^{\ge1}(G_M,\Z/l)$ is a Koszul module over a Koszul algebra
$H^*(G_K,\Z/l)$.
 By~\cite[Corollary~6.2$\.$(b)]{Pbogom}, it follows from this
conjecture that the graded algebra $A'=k\op A_1\op A_2\op\dsb$
is Koszul.
 Consequently, the graded ring $A$ is also Koszul.
 Thus, according to~\ref{koszul-cases}, Conjecture
from~\ref{main-conjecture} follows from~\cite[Conjecture~16]{Pdivis}
for such a field extension $M/K$ and coefficients~$k$.

\subsection{Artin--Tate motives and extensions of the basic field}
\label{at-base-change}
 Let $k$ be a coefficient ring and $M/K'/K$ be a tower of fields
such that $M/K$ is a Galois extension.
 Then the silly filtration conjecture for Artin--Tate motives holds
for the field extension $M/K'$ provided that it holds for
the extension $M/K$ with the same coefficients~$k$.

 Indeed, passage to the inductive limit with respect to the functors
of extension of scalars reduces the quesion to the case when
the field extension $K'/K$ is finite.
 In this case, the functor of extension of scalars $\D\M(K,k)
\rarrow\D\M(K',k)$ is right adjoint to the functor of restriction
of scalars $\D\M(K',k)\rarrow\D\M(K,k)$.
 Both functors send the Artin--Tate motives $k[L/K](i)$ or 
$k[L/K'](i)$ over $K$ or $K'$ corresponding to fields $L\subset M$ to
(finite direct sums of) similar motives over $K'$ or $K$,
respectively.
 Moreover, any motive $k[L/K'](i)$ over $K'$ is a direct summand of
the image of the motive $k[L/K](i)$ under the extension of
scalars functor.
 The assertion now follows from (the dual version of)
Proposition~\ref{silly-filtrations-appx}.1.

\subsection{New approach to Milnor--Bloch--Kato conjecture}
 Let $l$~be a prime number and $K$ be a field of characteristic
different from~$l$ having no finite separable extensions of degree
prime to~$l$.
 Set $M=\oK$ to be the separable closure of $K$ and let $A$
be the corresponding big graded ring of diagonal $\Hom$ for
Artin--Tate motives~\eqref{AT-big-ring} with coefficients~$\Z/l$.
 Consider also the related big graded ring $A'$
(see~\ref{koszul-cases}).

 Assume that the Galois symbol (norm residue homomorphism) for
finite extensions of the field $K$ and coefficients $\Z/l$ is
an isomorphism in degree $n=2$ and a monomorphism in degree $n=3$.
 Then whenever the big graded ring $A$ (or, equivalently, $A'$) is
Koszul, this map is an isomorphism in all degrees for such fields
and such coefficients.
 This follows from Theorem from Section~\ref{nonfiltered-secn}.
 This is another version of the approach to the Milnor--Bloch--Kato
conjecture suggested in~\cite{PV}.
 Notice that by~\ref{at-base-change} the Koszulity condition that
appears here holds for all algebraic extensions of a field $K$ if
holds for~$K$.

\subsection{Koszulity of cohomology}
 As a general heuristic rule, the cohomology algebra $H^*(X)$ of
an object $X$ is Koszul (as a graded ring with the cohomological
grading) if and only if two conditions hold:
\begin{enumerate}
 \item $X$ is of a $\Kpi$ type; and
 \item the DG\+algebra computing $H^*(X)$ is ``quasi-formal'',
meaning that the higher Massey products (defined as
the differentials in the spectral sequence converging from
the cohomology of the bar-construction of $H^*(X)$ to
the cohomology of the bar-construction of the DG\+algebra
computing it) in $H^*(X)$ vanish.
\end{enumerate}
 In particular, any formal DG\+algebra (i.~e., DG\+algebra
which can be connected with its cohomology by a chain of
multiplicative quasi-isomorphisms) is quasi-formal.
 Concerning the $\Kpi$ condition, it has to be interpreted in a way
consistent with the cohomology theory under consideration.
 A discussion of the situation when $X$ is a rational homotopy
type and $H^*(X)=H^*(X,\Q)$ can be found in~\cite{PY}.

 In particular, given a profinite group $G$ and a prime number~$l$,
the cohomology algebra $H^*(G,\Z/l)$ is Koszul if and only if
\begin{enumerate}
 \item the natural map $H^*(G^{(l)},\Z/l)\rarrow H^*(G,\Z/l)$
induced by the map $G\rarrow G^{(l)}$ from the group $G$ to
its maximal quotient pro-$l$-group $G^{(l)}$ is an isomorphism
of graded algebras; and
 \item the higher Massey products in $H^*(G,\Z/l)$ vanish.
\end{enumerate}

 Assuming that $H^*(G,\Z/l)$ is Koszul, a proof of~(1) can be
found in~\cite[Section~5]{Pbogom} and (2)~tautologically holds
for the grading reasons.
 To deduce Koszulity from (1) and~(2), notice that the coalgebra
$\Z/l(G^{(l)})$ of locally constant $\Z/$\+valued functions on
$G^{(l)}$ is conilpotent, hence the bar-construction of its
cobar-construction, being quasi-isomorphic to
it~\cite[Theorem~6.10(b)]{Pkoszul}, has cohomology
concentrated in degree~$0$.

 A very general version of the implication Koszulity
$\Longrightarrow$ $\Kpi$ is provided by Corollary in
Section~\ref{nonfiltered-secn}.

 Since it suffices to prove the Milnor--Bloch--Kato conjecture for
the fields whose Galois groups are pro-$l$-groups, it follows that
this conjecture, taken together with the silly filtration conjecture
for Tate motives with coefficients $\Z/l$ over fields containing
a primitive $l$\+root of unity, are equivalent to the assertion that
the higher Massey products in $H^*(G_K,\Z/l)$ vanish for any such
field~$K$.

 Notice that the DG\+algebra computing the Galois cohomology with
coefficients in $\Z/l$ is \emph{not} in general formal, only
conjecturally quasi-formal.
 Indeed, the Koszulity means that the associated graded coalgebra
of $\Z/l(G_K^{(l)})$ with respect to the coaugmentation filtration
is Koszul and quadratic dual to $H^*(G_K,\Z/l)$ \cite{PV}.
 The formality would mean that $\Z/l(G_K^{(l)})$ is isomorphic to
this associated graded coalgebra, i.~e., admits a Koszul grading.

 Let $K$ be a finite extension of $\mathbb F_p((z))$ or $\Q_p$,
with $p\ne l$; assume that $K$ contains a primitive $l$\+root
of unity if $l$ is odd, or a square root of~$-1$ if $l=2$.
 Then the algebra $H^*(G_K,\Z/l)$ is an exterior algebra with
two generators~\cite{PV2}, so its quadratic dual coalgebra is
a symmetric coalgebra with two generators.
 If the group coalgebra $\Z/l(G^{(l)})$ of $G_K^{(l)}$ were
quasi-isomorphic to the latter coalgebra, it would mean that
$G_K^{(l)}$ is commutative.
 This is not the case; in fact, $G_K^{(l)}$ is a semidirect
product of two copies of $\Z_l$ with one of them acting
nontrivially in the other one.

\subsection{Tate motivic sheaves over a scheme}
 The following attempt to describe the triangulated category of
Tate motives/motivic sheaves with finite coefficients $\Z/m$ over
a smooth variety $S$ over a field $K$ of characteristic not
dividing~$m$ looks natural.
 Its development was influenced by the author's conversations with
V.~Vologodsky.

 Consider the exact category $\F$ whose objects are finitely filtered
\'etale sheaves of $\Z/m$\+modules $(\mathcal{N},F)$ over $S$ such
that the quotient sheaf $F^i\mathcal{N}/F^{i+1}\mathcal{N}$ is
the tensor product of a sheaf of $\Z/m$\+modules lifted from
the Zariski/Nisnevich topology of $S$ with the \'etale sheaf
$\mu_m^{\ot i}$.
 Let $\D$ be the full triangulated subcategory of $\D^b(\F)$
generated by the Tate objects $\Z/m(i)$, which are identified with
the sheaves $\mu_m^{\ot i}$ placed in the filtration component~$i$.

 Then $\D$ is the candidate triangulated category of Tate motivic
sheaves over~$S$.
 The main piece of evidence supporting this conjecture that is
presently known to the author is that both the motivic cohomology
of $S$ and the $\Z/m$\+modules $\Ext$ between the Tate objects in
the triangulated category $\D$ localize in the Zariski topology.
 So it suffices to consider the case when $S$ is the spectrum of
the local ring of a scheme point on a smooth variety.

 A different version of this construction (for Artin--Tate rather
than Tate motivic sheaves) is shown to work, under certain
assumptions, in the forthcoming paper~\cite{Psheaves}.

\subsection{Tate motives and discrete valuation rings}
 The following purports to answer a question posed to the author
by V.~Vologodsky.
 Let $V$ be a discrete valuation ring with the quotient field $K$
and the residue field~$k$.
 Let $m=l^r$ be a power of a prime different from $\chr k$.
 Consider the embeddings of schemes $\upsilon\:\Spec K\rarrow \Spec V$
and $\iota\:\Spec k\rarrow \Spec V$.
 The functor $\iota^*R\upsilon_*\:\D\M(K,\Z/m)\rarrow\D\M(k,\Z/m)$
is supposed to take the triangulated subcategory of Tate motives in
$\D\M(K,\Z/m)$ into the triangulated subcategory of Tate motives
in $\D\M(k,\Z/m)$.
 Below we suggest a conjectural description of this functor on
Tate motives in terms of complexes of filtered modules over
the absolute Galois groups $G_K$ and~$G_k$.

 There is a natural closed subgroup $G_L\subset G_K$ corresponding
to the Henselization $L$ of the discrete valuation field~$K$.
 The group $G_k$ is the quotient group of $G_L$ by the inertia
subgroup~$I$.
 Since $l\ne\chr k$, the maximal quotient pro-$l$-group $I_l=I/I'$
of the group $I$ is naturally isomorphic to the projective limit
$\Z_l(1)$ of the groups $\mu_{l^t}$ of $l^t$\+roots of unity
in~$\overline{k}$ (or equivalently, in~$\overline{L}$), while
the (profinite) order of the group~$I'$ is not divisible by~$l$.

 Let $(N,F)$ be a finitely filtered discrete $G_K$\+module over
$\Z/m$ whose quotient modules $\gr^i_FN$ are finite direct sums of
the modules $\mu_m^{\ot i}$.
 Restrict the action of $G_K$ in $N$ to $G_L$; since the subgroup
$I\subset G_L$ acts in $N$ by endomorphisms unipotent with respect
to the filtration $F$, its action factorizes through~$I_l$.
 Hence the quotient group $G_L/I'$ acts in~$N$. 

 Let $x$ be a topological generator of the group $I_l=\Z_l(1)$
and $\bar x$ be the corresponding generator of the group~$\mu_m$.
 Consider the map $R_x\:(x-1)\ot\bar x^{-1}\: N\rarrow N\ot_{\Z/m}
\mu_m^{\ot -1}$.
 Since $x$ acts trivially in $\gr_FN$, this map shifts the filtration,
i.~e., defines a morphism of filtered $\Z/m$\+modules
$N\rarrow N(-1)$.
 Unless $k$~contains all the $l^t$\+roots of unity, this map does
not respect the action of $G_L/I'$, however.

 To deal with this problem, consider the expression $\phi(x,n) =
1+x+\dsb+x^{n-1}$, defined for all nonnegative integers~$n$ and
taking values in the endomorphisms of~$N$.
 The equations $\phi(x\;a+b) = \phi(x,a) + x^a\phi(x,b)$ and
$\phi(x,ab) = \phi(x,a)\phi(x^a,b)$ hold.
 One easily checks that the function~$\phi$ can be extended by
continuity to a locally constant function of $n\in\Z_l$.
 For any $n\in\Z_l$, set $\phi_x(x^n)=\phi(x,n)$; then $\phi_x$
is a $1$\+cocycle of the group $I_l$ with coefficients in
the filtration-preserving endomorphisms of $N$, i.~e., the equation
$\phi_x(yz) = \phi_x(y) + y\phi_x(z)$ holds for $y$, $z\in I_l$.
 For any $n\in\Z_l^*$, set $\psi_x(n)=\phi(x,n)/n$; then the function
$\psi_x$ takes values in the automorphisms of $N$ unipotent with
respect to $F$ and satisfies the cocycle equation $\psi_x(ab) = 
\psi_x(a)\psi_{x^a}(b)$.

 Let $\rho\:G_L/I'\rarrow\Aut_{\Z/m}(N(-1))$ be the action
of $G_L/I'$ in $N(-1)$ induced by the actions of $G_L/I'$ in $N$
and~$\mu_m$.
 Define a new action $\rho_x\:G_L/I'\rarrow\Aut_{\Z/m}(N(-1))$
by the rule $\rho_x(g)=\psi_x(\chi(g))^{-1}\rho(g)$,
where $\chi\:G_L/I'\rarrow\Z_l^*$ is the cyclotomic character.
 This new action is associative due to the above cocycle equation
for~$\psi$, since $gxg^{-1}=x^{\chi(g)}$.
 Denote the filtered $\Z/m$\+module $N(-1)$ endowed with
this new action of the group $G_L/I'$ by $N(-1)_x$.
 Then $N(-1)_x$ is a finitely filtered $G_L/I'$\+module with
cyclotomic associated quotients, and one readily checks that
$R_x\:N\rarrow N(-1)_x$ is a morphism of filtered
$G_L/I'$\+modules.

 Replacing $x$ with $x^n$, where $n\in\Z_l^*$, we obtain another
similar morphism $R_{x^n}\:N\rarrow N(-1)_{x^n}$.
 The isomorphism of filtered $G_L/I'$\+modules $\psi_x(n)\:\allowbreak
N(-1)_x\rarrow N(-1)_{x^n}$ identifies the morphisms $R_x$ 
and $R_{x^n}$, which makes the two-term complex of filtered
$G_L/I'$\+modules $R_x\:N\rarrow N(-1)_x$, denoted
symbolically by $R\:N\rarrow N(-1)$, defined uniquely
up to a unique isomorphism.
 
 Now we can assign to any bounded complex $(N^\bu,F)$ of finitely
filtered discrete $G_K$\+modules over $\Z/m$ with cyclotomic
associated quotients the total complex of the bicomplex with two
rows $R\:N^\bu \rarrow N^\bu(-1)$.
 The complex so obtained is a bounded complex of finitely filtered
$G_L/I'$\+modules with cyclotomic associated quotients.
 We need to transform it into a similar complex of $G_k$\+modules.

 Pick a uniformizing element $\pi$ in the ring $V$ (or even in
the field~$L$).
 Then one can extend the field~$L$ by adjoining a compatible system
of $l^t$\+roots of~$\pi$.
 This (not necessarily normal, but separable) algebraic field
extension corresponds to a section $G_k\rarrow G_L/I'$ of
the surjective profinite group homomorphism $G_L/I'\rarrow G_k$.
 Composing the action of $G_L/I'$ with this section, we obtain
the desired complex of filtered $G_k$\+modules.

 A change of the compatible system of $l^t$\+roots of~$\pi$
corresponds to an element of the kernel $\Z_l(1)=I_l$ of
the group homomorphism $G_L/I'\rarrow G_k$, and the conjugation
with this element transforms one of the related two sections
$G_k\rarrow G_L/I'$ into the other.
 The action of this element defines the complex of $G_k$\+modules
that we have constructed as independent of the choice of
the compatible system of roots of~$\pi$ up to a unique isomorphism.

 Finally, assume that $\pi'$ and~$\pi''$ are two uniformizing
elements, and compatible systems of $l^t$\+roots have been chosen
for each of them.
 Then we have two sections $\sigma'$, $\sigma''\:G_k\rarrow G_L/I'$
differing by a cocycle $\xi\:G_k\rarrow I_l$; so
$\sigma'(g) = \sigma''(g)\xi(g)$.
 Consider the two two-term complex of $G_k$\+modules
$R_x\:N_{\sigma'}\rarrow N(-1)_{x,\sigma'}$ and
$R_x\:N_{\sigma''}\rarrow N(-1)_{x,\sigma''}$ obtained by composing
the action of $G_K$ in the two-term complex $R_x\:N\rarrow N(-1)_x$
with $\sigma'$ and~$\sigma''$.
 We will construct a two-term complex of finitely filtered
$G_k$\+modules with cyclotomic associated quotients $\widetilde R\:
\widetilde N\rarrow \widetilde M$ together with two quasi-isomorphisms
$p'$ and~$p''$ mapping this intermediate complex onto each of
the two complexes that we need to compare.

 The filtered $G_k$\+module $\widetilde N$ is simply the direct sum
$N_{\sigma'}\oplus N_{\sigma''}$; its morphisms $p'$ and~$p''$
to $N_{\sigma'}$ and $N_{\sigma''}$ are the projections to
the direct summands.
 The filtered $G_k$\+module $\widetilde M$ is an extension of
$N(-1)_{x,\sigma''}$ with the kernel $N_{\sigma'}$.
 We construct it as the filtered $\Z/m$\+module $N\oplus N(-1)$
endowed with the action of $G_k$ given by the formula
$g(u,v)=(\sigma'(g)(u)+\sigma''(g)\phi_x(\xi(g))
(v\ot\bar x)\;\allowbreak\rho_x(\sigma''(g))(v))$ for
$u\in N$, \ $v\in N(-1)$, and $g\in G_k$.
 This action is associative due to the cocycle equations for~$\phi$.
 The morphisms $p'\:\widetilde M\rarrow N(-1)_{x,\sigma'}$ and
$p''\:\widetilde M\rarrow N(-1)_{x,\sigma''}$ are given by
the rules $p'(u,v) = v$ and $p''(u,v) = v + R_x(u)$.
 One checks that the kernels of $p'$ and~$p''$ are contractible
two-term complexes.

 Notice that the above construction is not applicable to
Artin--Tate motives, since the action of $\Z_l(1)$ on the associated
quotient modules of the related filtered modules can be nontrivial,
so the compatibility with the filtration breaks down.

\subsection{Milnor ring with integral coefficients}
 Due to the results of this paper (see
Section~\ref{koszul-rings-secn}), we now know what it \emph{means}
for the Milnor ring (with integral coefficients) $\KM(K)$ of
a field $K$ to be Koszul.

 It is still a mystery whether one should expect it to \emph{be}
Koszul for an arbitrary field~$K$, or what the motivic interpretation
of it being Koszul might consist in.
 In fact, we do not know this even for the algebra $\KM(K)\ot_Z\Q$
(see~\ref{introd-list-of-situations}).

\subsection{Tate motivic DG\+algebra $A$}
 Fix a field~$K$.
 We are interested in a negatively internally graded DG\+algebra $A$
over $\Z$ (see Appendix~\ref{kpi1-appendix}) such that for any
coefficient ring $k$ as above the full triangulated subcategory of
the derived category of internally graded DG\+modules over $A\ot_\Z k$
generated by the free DG\+modules $A\ot_\Z k(i)$ is equivalent to
the full triangulated subcategory of $\D\M(K,k)$ generated by
the Tate objects $k(i)$.
 Moreover, the free DG\+modules $A\ot_\Z k(i)$ should correspond to
the Tate objects $k(i)$ under this equivalence of triangulated
categories, and the underlying bigraded $\Z$\+module of
the DG\+algebra $A$ should be flat (i.~e., torsion-free).
 In particular, the cohomology $H(A\ot_\Z k)$ of the DG\+algebra
$A\ot_\Z k$ should be isomorphic to the motivic cohomology of $K$
with coefficients in~$k$.

 There must be many candidate constructions of the DG\+algebra $A$
obtainable from the contemporary literature; we will suggest just
one such construction adopted to Voevodsky's definition~\cite{Voev}
of the triangulated category of geometric motives.
 It was inspired by the papers~\cite{BV,Bond0}.
 Consider the additive category $\mathit{SmCor}$ of smooth schemes
over~$k$ and finite correspondences between them, and pass to
the DG\+category of bounded complexes over $\mathit{SmCor}$.
 The latter is a tensor DG\+category with respect to the Cartesian
product of schemes over~$k$; in particular, there is
the DG\+endofunctor of Tate twist acting on it.
 Consider the full DG\+subcategory of the ``relations'' of homotopy
invariance and Mayer--Vietoris; close it under the Tate twists and
apply the construction of Drinfeld localization~\cite[Section~3]{Dr}.
 Set
 $$
  A_i=\varinjlim\nolimits_{j\to+\infty}\Hom(\Z(i+j),\Z(j))
 $$
for $i<0$, the $\Hom$ being taken in the DG\+category we have
constructed; $A_0=k$, and $A_i=0$ for $i>0$.
 It is clear that this construction is compatible with tensoring
with a ring of coefficients~$k$ in a reasonable sense.

 Another approach is to use one of the various constructions of
the complex of algebraic cycles (see e.~g.~\cite{Bond0}) in
the role of~$A$.

\subsection{Classical $\Kpi$\+conjecture and integral Tate motives}
\label{classical-kpi1-integral}
 Now let $C$ denote the reduced bar-construction of the DG\+algebra
$A$ over~$\Z$.
 By Theorem~\ref{kpi1-appendix}.1, the main conjecture
from~\ref{main-conjecture} for $M=K$ and $k=\Z$ is equivalent to
the DG\+coalgebra $C$ having no cohomology in the positive
cohomological degrees, $H^n(C)=0$ for $n>0$.

 By Proposition~\ref{kpi1-appendix}.2$\.$(1), the Beilinson--Soul\'e
vanishing conjecture with rational coefficients is equivalent
to $C\ot_Z\Q$ having no cohomology in the negative cohomological
degrees, $H^n(C\ot_\Z\Q)=0$ for $n<0$.
 For the same reason, one expects $H^n(C\ot_\Z\Z/p)=0$ for $n<0$
if $K$ has a prime characteristic~$p$ (see~\ref{koszul-cases}).
 Finally, by Theorem~\ref{kpi1-appendix}.2$\.$(1),
the Beilinson--Lichtenbaum conjecture implies $H^n(C\ot_\Z\Z/l)=0$
for $n<-1$ and any prime number $l\ne\chr K$.

 By the universal coefficients formula, it follows that one should
expect $H^n(C)=0$ for all $n\ne0$.
 However, $H^0(C)$ is not supposed to be a torsion-free $\Z$\+module;
indeed, it should have nontrivial $l$\+torsion for any $l\ne\chr K$.
 Conversely, if $H^n(C)=0$ for $n\ne0$, then $H^n(A/\Z)=0$ for $n\le0$,
$H^n(A\ot_\Z k)=0$ for $n<0$ and any coefficient ring~$k$, and
the silly filtration conjecture holds for $M=K$ and any~$k$.

 As to the cohomology coalgebra $H(C\ot_\Z\Z/l)$, one can expect it
to be described by the assertions of
Theorem~\ref{kpi1-appendix}.2$\.$(2-3).

\subsection{Tate motivic DG\+coalgebra $C$ and its cohomology coalgebra}
 In general, given a flat DG\+coalgebra $C$ over $\Z$ with nonflat
cohomology $H(C)$, there is no way to define a coalgebra structure
on $H(C)$, as the natural map $H(C)\ot_\Z H(C)\rarrow H(C\ot_\Z C)$
has no natural inverse.
 However, when $H^n(C)=0$ for $n\ne0$, there is a natural coassociative
coalgebra structure on $H^0(C)$.
 Moreover, assume that $C$ is negatively internally graded, and consider
the minimal full subcategory $\M$ of the derived category of internally
graded DG\+comodules over $C$, containing the trivial DG\+comodules
$\Z(i)$ and closed under extensions.
 Since the homological dimension of~$\Z$ is equal to~$1$, \ $\M$ is
an exact subcategory of the triangulated category $\D$ it generates.
 Now there is an exact cohomology functor from $\M$ to the exact
category $\G$ of comodules over $H^0(C)$, free and finitely generated
as $\Z$\+modules.

 But this functor is not an equivalence.
 It suffices to consider the example when $K=\overline{\mathbb{F}}_q$
is the algebraic closure of a finite field.
 In this case, the internally graded DG\+coalgebra $C$ should be
quasi-isomorphic to the DG\+coalgebra with the only nonzero components
$C_0=\Z$ and $C_{-1}=(\Z[q^{-1}]\rarrow\Q)$, the term $\Z[q^{-1}]$
being placed in the cohomological degree~$-1$ and the term $\Q$ in
the cohomological degree~$0$, with the identity embedding of
$\Z[q^{-1}]$ into $\Q$ as the differential.
 Using the reduced cobar-construction, one can compute that
$\Hom_\M(\Z(i),\Z(i))=\Z$ and $\Ext^1_\M(\Z(i),\Z(j))=\Q/\Z[q^{-1}]$
for $i<j$, while all the other $\Ext$ groups between the objects
$\Z(i)$ in the exact category $\M$ vanish
(cf.~\ref{finite-field-tate-motives}).
 In the exact category $\G$, the groups $\Ext^1_\G(\Z(i),\Z(j))=0$
for $i+2\le j$ differ from those in $\M$ (while all the other $\Ext$
groups between the objects $\Z(i)$ are the same).

 Nor such a description of the exact category $\M$ in terms of
the nonflat coalgebra $H^0(C)$, if it existed, would be of much use;
see Remark in Section~\ref{koszul-rings-secn}.
 Instead, it appears that one has to learn how to work with
the DG\+coalgebra $C$ up to a quasi-isomorphism in the class of
negatively internally graded DG\+coalgebras with torsion-free
underlying bigraded $\Z$\+modules.
 It is impossible in general to recover such a quasi-isomorphism class
of a DG\+coalgebra $C$ with $H^n(C)=0$ for $n\ne0$ from the graded
coalgebra structure on $H^0(C)$, as one can see in the case when
$C_{-1}$ and $C_{-2}$ are nonzero, while $C_{-i}=0$ for $i\ge3$.
 The point is, the map of complexes of $\Z$\+modules $C_{-2}\rarrow
C_{-1}\ot_\Z C_{-1}$ is not determined, as a morphism in the derived
category of $\Z$\+modules, by the induced map of the cohomology.

\subsection{Integral Tate motives over finite field}
\label{finite-field-tate-motives}
 One example when the derived category of Tate motives with integral
coefficients admits an elementary construction is that of motives
over a finite field.
 One expects that
\begin{equation}
 \begin{aligned}  \label{finite-field}
    \Hom_{\D\M(\mathbb{F}_q,\Z)}(\Z(0),\Z(i)[1])&\simeq
                       \mu_{q^i-1}^{\ot i},
      \quad \text{for \ $i\ge0$;} \\
    \Hom_{\D\M(\mathbb{F}_q,\Z)}(\Z(0),\Z(i)[n])&=0,
      \quad \text{for \ $i\ne0$, \ $n\ne1$.}
 \end{aligned}
\end{equation}
 Since $\Hom_{\D\M}(\Z(i),\Z(j)[n])=0$ for $n\ge2$, the silly
filtration condition is trivial in this case.
 One can construct an exact category $\F$ with the $\Ext$ groups
given by~\eqref{finite-field}, except that $\mu_{q^i-1}$ is replaced
with $\Z/(q^i-1)$, in the following way.

 First let us describe an exact category $\F$ resembling mixed Tate
motives over the algebraic closure of a finite field
$\overline{\mathbb{F}}_q$ with integral coefficients.
 Consider the category of finitely filtered abelian groups $(N,F)$
endowed with a (fixed) splitting of the filtration over $\Z_{(p)}$,
i.~e., $N\ot_{\Z}\Z_{(p)}\simeq \gr_FN\ot_\Z\Z_{(p)}$.
 Here $\Z_{(p)}$ denotes the localization of $\Z$ at the prime
ideal~$(p)$.
 It is also required that $\gr_FN$ be a finitely generated free
abelian group.
 Define the objects $\Z(i)\in\F$ as the groups $\Z$ placed
in the filtration component~$i$.
 Then one has $\Hom_\F(\Z(i),\Z(i))=\Z$ and
$\Ext_\F^1(\Z(i),\Z(j))\simeq \Q/\Z[q^{-1}]$ for $i<j$, while
all the remaining $\Ext$ groups between the objects $\Z(i)$ vanish.
 All objects of $\F$ are iterated extensions of $\Z(i)$.

 Now consider the case of mixed Tate motives over $\mathbb{F}_q$
with coefficients in $\Z[q^{-1}]$.
 Consider the exact category $\F$ of finitely filtered
$\Z[q^{-1}]$\+modules $(N,F)$ endowed with an endomorphism
$\phi\:N\rarrow N$ that is compatible with the filtration $F$ and
acts by multiplication with~$q^i$ on the successive quotient
$F^iN/F^{i+1}N$.
 It is required that $\gr_FN$ be a finitely generated free module
over $\Z[q^{-1}]$.
 Define the objects $\Z[q^{-1}](i)\in\F$ as the modules
$\Z[q^{-1}]$ placed in the filtration component~$i$, with
$\phi$ acting on them by~$q^i$.
 Then one has $\Hom_\F(\Z[q^{-1}](i),\Z[q^{-1}](i))=\Z[q^{-1}]$ and
$\Ext_\F^1(\Z[q^{-1}](i),\Z[q^{-1}](j))\simeq \Z/(q^{j-i}-1)$ for
$i\le j$, while all the remaining $\Ext$ groups between the objects
$\Z[q^{-1}](i)$ vanish.

 Finally, to construct an exact category $\F$ resembling mixed Tate
motives over $\mathbb{F}_q$ with integral coefficients, consider
finitely filtered abelian groups $(N,F)$ endowed with a family of
endomorphisms $\phi^{(i)}\: F^iN\rarrow F^iN$ such that
$\phi^{(i)}|_{F^jN}=q^{j-i}\phi^{(j)}$ for $i\le j$ and the action of
$\phi^{(i)}$ on $F^iN/F^{i+1}N$ is the identity endomorphism.
 Define the objects $\Z(i)\in\F$ as the groups $\Z$ placed
in the filtration component~$i$, with $\phi^{(j)}$ acting on them
by~$q^{i-j}$.
 Then one has $\Hom_\F(\Z(i),\Z(i))=\Z$ and
$\Ext_\F^1(\Z(i),\Z(j))\simeq\Z/(q^{j-i}-1)$ for $i\le j$, while
all the remaining $\Ext$ groups between the objects $\Z(i)$ vanish.

\subsection{Epilogue}
 One of the most important problems of the theory around
the Milnor--Bloch--Kato conjecture is to describe as precisely
as possible the special properties of absolute Galois groups,
their cyclotomic characters, their subgroups, their cohomology, etc.
 One particular aspects of this problem is to describe the behavior
of Galois cohomology with constant/cyclotomic coefficients with
respect to extensions of the field~$K$.

 This project was originally conceived as an approach to the latter
problem.
 The hope was that the main conjecture from~\ref{main-conjecture}
would impose strong restrictions on the behavior of Galois
cohomology in field extensions.
 This hope did not quite materialize, as it now appears that
the main conjecture is not so strong.

 Consider the case of a cyclic extension of prime degree $M/K$ and
arbitrary coefficients~$k$.
 Then the graded algebra $A$ from~\eqref{AT-algebra} is isomorphic to
$$
 k\ot_\Z(\KM(K)\op \KM(M)\op \KM(M)\op \KM(M)[G_K/G_M]).
$$
 The multiplication in this algebra is defined in terms of the map
$\KM(K)\rarrow\KM(M)$ induced by the embedding of fields $K\rarrow M$,
the transfer map $\KM(M)\rarrow \KM(K)$, and the action of
$G_K/G_M$ on $\KM(M)$.

 We have mentioned already in~\ref{supporting-evidence} that
the algebra $A$ is generated by $A_1$ over $A_0$.
 Furthermore, \emph{either} a weak form of Hilbert theorem~90 for
the Milnor K\+theory \emph{or} the statement of the Bass--Tate lemma
are both sufficient to conclude that $A$ is quadratic.
 And in the situation of~\ref{cyclic-extension}, Koszulity of $A$
follows from the conjecture from~\cite{Pdivis}, but there is
no apparent way to obtain the converse implication.

 Thus the main our achievements in this paper are the connection of
the silly filtration conjecture with Koszulity conjectures established
in~\ref{koszul-cases} and the elementary construction of the category
of mixed Artin--Tate motives with finite coefficients obtained
in~\ref{elementary-description}.
 In addition, there is the result of~\ref{classical-kpi1-integral} on
the $\Kpi$\+conjecture with integral coefficients.

\appendix
\Section{Exact Categories}

\subsection{Preadditive categories}  \label{big-graded-rings}
 A \emph{preadditive category} $\A$ is a category in which the set
of morphisms $\Hom_\A(X,Y)$ between any two given objects is
endowed with the structure of an abelian group in such a way that
the composition maps are biadditive.
 A functor $F:\A\rarrow\B$ between preadditive categories is
called \emph{additive} if it takes the sums of morphisms in $\A$
to the sums of morphisms in $\B$.

 In particular, a preadditive category with a single object is
the same that a (noncommutative) ring.
 For this reason, we will use the term \emph{big ring over} a set
$\Sigma$ as another name for a small preadditive category
with the set of objects $\Sigma=\Ob\A$.
 In other words, a big ring over $\Sigma$ is a collection of
abelian groups $(A_{\sigma\tau})_{\sigma,\tau\in\Sigma}$, \
$A_{\sigma\tau}=\Hom_\A(\tau,\sigma)$, together with multiplication
maps $A_{\sigma\tau}\times A_{\tau\rho}\rarrow A_{\sigma\rho}$
and unit elements $e_\sigma\in A_{\sigma\sigma}$ satisfying
the conventional axioms.

 Let $\Ab$ denote the category of abelian groups.
 A \emph{left module} over a big ring $A$ over $\Sigma$ is another
name for a covariant additive functor $\A\rarrow\Ab$, and
a \emph{right module} over $A$ is a contravariant additive functor
between the same categories.
 In other words, a left $A$\+module $M$ is a collection of abelian
groups $(M_\tau)_{\tau\in\Sigma}$ together with action maps
$A_{\sigma\tau}\times M_\tau\rarrow M_\sigma$, and similarly for
right modules.
 Given a big ring $A$ over $\Sigma$ and a big ring $B$ over $\Pi$,
a \emph{bimodule} $K$ over $A$ and $B$ is a collection of abelian
groups $(K_{\tau\rho})_{\tau\in\Sigma;\.\rho\in\Pi}$ together with
biaction maps $A_{\sigma\tau}\times K_{\tau\rho}\times B_{\rho\pi}
\rarrow K_{\sigma\pi}$ satisfying the obvious triadditivity,
associativity, and unit axioms.

 It follows from the latter definition that a bimodule over big
graded rings can not be literally viewed as a left or a right module
over one of the big rings; rather, it is a \emph{collection}
of left $A$\+modules and a collection of right $B$\+modules.
 Nevertheless, we will speak of bimodules as if they have underlying
left and right module structures, omitting the references to
collections of modules for brevity.
 Any operations with collections of modules are performed with
every module in the collection, providing the corresponding
collections of outputs; any properties of collections of modules
are meant to hold for \emph{every} module in the collection.

 Let $M$ be a left $A$\+module and $N$ be a right $A$\+module; then
the \emph{tensor product} $N\ot_AM$ is the abelian group generated
by the formal symbols $n\ot m$, where $n\in N_\sigma$ and
$m\in M_\sigma$, subject to the conventional relations $na\ot m=
n\ot am$ for any $n\in\M_\sigma$, \ $a\in A_{\sigma\tau}$, and
$m\in M_\tau$, and biadditivity.
 As explained above, this simultaneously defines the operation of
tensor product of bimodules; in particular, the tensor product of
two $A$\+bimodules is again an $A$\+bimodule.

 The category of $A$\+(bi)modules is an abelian category satisfying
Ab5 and Ab4*, with a set of generators, and enough projectives and
injectives~\cite{Grot}.
 For any set $I$ mapping into $\Sigma$ one defines the \emph{free}
left $A$\+module $M$ generated by $I$ by the rule
$M_\tau=\bigoplus_{i\in I} A_{\tau,\sigma(i)}$, and similarly for
free right modules.
 These are just direct sums of representable functors on~$\A$.
 An $A$\+module is \emph{projective} if and only if it is a direct
summand of a free module.

 The functor of tensor product over $A$ is right exact.
 A right $A$\+module $N$ is \emph{flat} if the functor $M\maps
N\ot_A M$ is exact on the abelian category of left $A$\+modules.
 Any projective $A$\+module is flat. 
 One defines the left derived functors $\Tor^A_n(N,M)$ for $n\ge0$,
a right $A$\+module $N$ and a left $A$\+module $M$, by
the conventional derived functor procedure, using projective or
flat resolutions.
 The derived functors have the standard properties of the functor
$\Tor$ over a ring; in particular, resolving either argument leads
to the same result.

 A big ring $A$ over a set $\Sigma$ is said to be graded if every
abelian group $A_{\sigma\tau}$ is graded.
 In this paper, we only consider big rings graded by nonnegative
integers.
 Thus a \emph{big graded ring} is a collection of abelian groups
$(A_{\sigma\tau;\.n})_{\sigma,\tau\in\Sigma;\.n\ge0}$ endowed with
multiplication maps $A_{\sigma\tau;\.n}\times A_{\tau\rho;\.m}
\rarrow A_{\sigma\rho;\.n+m}$ and unit elements $e_\sigma\in
A_{\sigma\sigma;\.0}$ satisfying the conventional biadditivity,
associativity, and unit axioms.
 One can consider graded modules over big graded rings in
the obvious sense.

 For a further discussion of rings with several objects,
see~\cite{Mit}.

\subsection{Additive categories}  \label{additive-categories}
 One can show that if an object $X$ of a preadditive category $\A$
is the product of a finite set of objects $X_i\in\A$, then it is
also the coproduct of this set of objects, and vice versa.
 In this case, $X$ is called the \emph{direct sum} of $X_i$.
 An \emph{additive category} is a preadditive category in which
the (co)product of any finite set of objects exists.
 The preadditive category structure on an additive category can be
recovered from its abstract category structure.

 Following the terminology of A.~Neeman's paper~\cite{Neem}, we call
an additive category $\A$ \emph{semi-saturated} (\emph{weakly
idempotent complete}) if any pair of its morphisms $p\:X\rarrow Y$
and $i\:Y\rarrow X$ such that $p\circ i=\id_Y$ comes from
an isomorphism $X\simeq Y\oplus Z$ for some object $Z\in\A$.
 An additive category is called \emph{saturated} (\emph{Karoubian},
\emph{pseudo-abelian}, \emph{idempotent complete}) if any its morphism
$e\:X\rarrow X$ such that $e\circ e=e$ is the projection on
a direct summand of a decomposition $X\simeq Y\oplus Z$.

 For any additive category $\A$ there exists a unique, up to
a uniquely defined equivalence, semi-saturated additive category
$\A^\ss$ (called the \emph{semi-saturation} of $\A$) together with
an additive functor $\ss\:\A\rarrow\A^\ss$ such that any additive
functor from $\A$ to a semi-saturated additive category $\B$
can be factorized through the functor~$\ss$ in a unique, up to
an isomorphism, way.
 There also exists a unique, up to a uniquely defined equivalence,
saturated additive category $\A^\sat$ (called the \emph{saturation}
of $\A$) together with an additive functor $\sat\:\A\rarrow\A^\sat$
satisfying the same condition with respect to functors from
$\A$ to saturated additive categories~$\B$.

 Let $\A$ be a full additive subcategory of an additive category~$\B$.
 Then the \emph{semi-saturation closure} $\A_\B^\ss$ of the
subcategory~$\A$ in $\B$ is the full subcategory of~$\B$ whose
objects are all the objects $Z\in\B$ for which there exists an
object $Y\in\A$ such that the object $X\op Y$ is also isomorphic
to an object from~$\A$.
 The \emph{saturation closure} $\A_\B^\sat$ of a subcategory~$\A$
is defined in the same way except that one allows $Y\in\B$.

 For example, any additive category admitting either the kernels or
the cokernels of all morphisms (see below) is saturated.
 Any triangulated category is semi-saturated.
 A triangulated subcategory of a triangulated category is thick
if and only if it coincides with its saturation closure~\cite{Neem}.
 The saturation of a triangulated category is naturally
a triangulated category again~\cite{BS}.

\begin{ex}
 To see just how bad a nonsaturated additive category can be,
consider the category $\A$ of all vector spaces $V$ over a field~$k$
such that $\dim V\ne 1$, $2$, or~$5$.
 This is an additive subcategory of the category of vector spaces,
since it is closed under direct sums.
 The category $\A$ is not even semi-saturated.
 On the other hand, the category $\B$ of all finite-dimensional
vector spaces of dimension divisible by~$3$ over~$k$ is an example
of a semi-saturated, but not saturated additive category.
 The category of finitely generated free modules over a ring~$R$
is not even semi-saturated in general. 
\end{ex}

 A morphism $k\:K\rarrow X$ in a preadditive category~$\A$
is called the \emph{kernel} of a morphism $f\:X\rarrow Y$
if for any object $Z\in\A$ the sequence
 $$
   0\lrarrow \Hom_\A(Z,K)\lrarrow \Hom_\A(Z,X)\lrarrow \Hom_\A(Z,Y)
 $$
is exact; in this case we write $K=\Ker f$ and $k=\ker f$.
 Analogously, a morphism $c\:Y\rarrow C$ is called
the \emph{cokernel} of a morphism $f\:X\rarrow Y$
if for any object $Z\in\A$ there is an exact sequence
 $$
   0\lrarrow \Hom_\A(C,Z)\lrarrow \Hom_\A(Y,Z)\lrarrow \Hom_\A(X,Z);
 $$
the notation: $C=\Coker f$ and $c=\coker f$.
 A morphism $f$ is called \emph{injective} if one has $\Ker f=0$
and \emph{surjective} if $\Coker f=0$.

 Throughout this appendix, we will sometimes use the following notation
for computations in additive categories: the composition of several
(previously defined) morphisms, say $X\rarrow Y$, \ $Y\rarrow Z$,
and $Z\rarrow T$, will be denoted by $[X\rar Y\rar Z\rar T]$.
 Then an equation like $[X\rar Y\rar Z]=[X\rar T\rar Z]$
means that the morphisms form a commutative square, etc.

\subsection{Axioms}  \label{axioms}
 There is a very detailed recent exposition~\cite{Bueh} of the theory
of exact categories, covering most of the material that one needs
to know in order to feel at ease while working with these things.
 The purpose of the following subsections is to complement that
exposition with several observations which will tend to add some
clarity as far as our goals are concerned, or present
an independent interest.

 Let $\E$ be an additive category endowed with a class of
\emph{admissible} (or \emph{exact}) \emph{triples} of objects
and morphisms
$$
 \T_\E=\{\.X'\rarrow X\rarrow X''\.\},
 \qquad X',\ X,\ X''\in\E.
$$
 A morphism $X\rarrow Y$ in the category $\E$ will be called
an \emph{admissible monomorphism} if it can be embedded into
an admissible triple $X\rarrow Y\rarrow Z$ and an
\emph{admissible epimorphism} if there exists an admissible
triple $T\rarrow X\rarrow Y$.
 An additive category $\E$ together with a class of admissible
triples $\T_\E$ is called an \emph{exact category}
if it satisfies the following axioms~Ex0\+-Ex3:
 \begin{itemize}
   \item[Ex0:\,]
      The zero triple $0\rarrow0\rarrow0$ is admissible.
      Any triple isomorphic to an admissible triple
      is admissible.
   \smallskip
   \item[Ex1:\,]
      For any admissible triple $X'\rarrow X\rarrow X''$
      and any object $Z\in\E$ there are exact sequences
        \begin{gather*}
           0\lrarrow\Hom_\E(Z,X')\lrarrow\Hom_\E(Z,X)\lrarrow
           \Hom_\E(Z,X'') \\
           0\lrarrow\Hom_\E(X'',\.Z)\lrarrow\Hom_\E(X,Z)\lrarrow
           \Hom_\E(X',\.Z)
        \end{gather*}
      In other words, the morphisms $X'\rarrow X$ and
      $X\rarrow X''$ are each other's kernel and cokernel.
   \smallskip
   \item[Ex2:\,]
      Let $X'\rarrow X\rarrow X''$ be an admissible triple.
      Then (a) for any morphism $X'\rarrow Y'$
      there exists a commutative diagram
        \begin{equation} \label{push-diagram}
          \thrfrdiag
          \begin{diagram}
            \node{X'} \arrow{e} \arrow{s}
            \node{X}  \arrow{e} \arrow{s}
            \node{X''}
              \\
            \node{Y'} \arrow{e}
            \node{Y}  \arrow{ne}
          \end{diagram}      
        \end{equation}     
      with an admissible triple $Y'\rarrow Y\rarrow X''$;
      and analogously, (b) for any morphism $Y''\rarrow X''$ 
      there exists a commutative diagram 
        \begin{equation} \label{pull-diagram}
          \thrfrdiag
          \begin{diagram}
            \node{X'}  \arrow{e} \arrow{se}
            \node{X}   \arrow{e}
            \node{X''} 
              \\          
            \node[2]{Y} \arrow{e} \arrow{n} 
            \node{Y''}  \arrow{n}
          \end{diagram}      
        \end{equation}     
      with an admissible triple $X'\rarrow Y\rarrow Y''$.
   \smallskip
   \item[Ex3:\,]
      (a) The composition of any two admissible monomorphisms
      is an admissible monomorphism.
      (b) The composition of any two admissible epimorphisms
      is an admissible epimorphism.
 \end{itemize}

 This list of axioms roughly corresponds to the axioms for exact
categories commonly used in modern expositions~\cite{Kel,Bueh},
with two exceptions.
 Firstly, the numbering is different, and secondly, our axiom~Ex2
is usually stated in a different form; see axiom~Ex$2'$ below.
 We prefer our axiom~Ex2.

 We will show that for any class of admissible triples $\T_\E$
satisfying~Ex0\+-Ex2 the diagrams \eqref{push-diagram}
and~\eqref{pull-diagram} in~Ex2 are defined uniquely up to a unique
isomorphism by the original triple $X'\rarrow X\rarrow X''$ and
the morphism $X'\rarrow Y'$ or $Y''\rarrow X''$, respectively.
 The object $Y$ in the diagram~\eqref{push-diagram} is necessarily
the fibered coproduct of the morphisms $X'\rarrow Y'$ and $X'\rarrow X$
and the object~$Y$ in the diagram~\eqref{pull-diagram} is the fibered
product of the morphisms $Y''\rarrow X''$ and $X\rarrow X''$.
 In other words, the axiom~Ex2 is equivalent modulo Ex0--Ex1
to the following axiom~Ex$2'$.
 \begin{itemize}
   \item[Ex$2'$:\,]
      (a) For any admissible monomorphism $X'\rarrow X$ and
      any morphism $X'\rarrow Y'$ there exists a fibered
      coproduct $Y=Y'\cop_{X'}X$ in the category $\E$ and
      the natural morphism $Y'\rarrow Y$ is an admissible
      monomorphism. \\
      (b) For any admissible epimorphism $X\rarrow X''$ and
      any morphism $Y''\rarrow X''$ there exists a fibered
      product $Y=Y''\pro_{X''}X$ in the category $\E$ and
      the natural morphism $Y\rarrow Y''$ is an admissible
      epimorphism.
 \end{itemize}

 Furthermore, the axiom~Ex2 together with the condition that
the additive category $\E$ be semi-saturated is equivalent modulo
Ex0\+-Ex1 to the following axiom~Ex$2''$.
 \begin{itemize}
   \item[Ex$2''$:\,]
      (a) A right divisor $g$ of an admissible monomorphism $fg$
      is an admissible monomorphism. \\
      (b) A left divisor $f$ of an admissible epimorphism $fg$
      is an admissible epimorphism. \\
      (c) If in the following commutative diagram
        \begin{equation*}
          \thrfrdiag
          \begin{diagram}
           \node{X'} \arrow{e} \arrow{se}
           \node{X_1}  \arrow{e} \arrow{s}
           \node{X''}
             \\
           \node[2]{X_2} \arrow{ne}
          \end{diagram}
        \end{equation*}
       both triples $X'\rarrow X_1\rarrow X''$
       and $X'\rarrow X_2\rarrow X''$ are admissible,
       then the morphism $X_1\rarrow X_2$ is an isomorphism.
 \end{itemize}

 The conditions Ex$2''$(a-b) are a stronger version of the last
``obscure'' axiom in Quillen's list of axioms for exact
categories~\cite{Quil}.
 The obscure axiom claims that a right divisor of an admissible
monomorphism is an admissible monomorphism provided that it has
a cokernel, and a left divisor of an admissible epimorphism is
an admissible epimorphism provided that it has a kernel.
 It is known~\cite{Kel,Bueh} that the obscure axiom follows from
the other ones.
 The above stronger form of the obscure axiom holds if and only if
$\E$ is semi-saturated~\cite{Kel2,Bueh}.
 The condition Ex$2''$(c) follows from Ex0\+-Ex1 and any of
the conditions Ex2$\.$(a) or Ex2$\.$(b).

 Finally, it is known that the conditions Ex3$\.$(a) and Ex3$\.$(b)
are equivalent modulo Ex0--Ex2 (see~\cite{Kel}).
 Sketches of proofs of the above equivalence assertions are given
in the next subsection.

\subsection{Proofs}  \label{exact-proofs}
 In what follows we assume that an additive category $\E$ is endowed
with a class of admissible triples satisfying the axioms Ex0\+-Ex1.

\begin{prop1}
 It follows from the condition~\textup{Ex2$\.$(a)} that if in
the diagram~\eqref{push-diagram} the triple $Y'\rarrow Y\rarrow X''$
is admissible and the morphism $X\rarrow X''$ is the cokernel
of the morphism $X'\rarrow X$, then the object $Y$ is
the fibered coproduct $Y=Y'\cop_{X'}X$ in the category~$\E$.
 Analogously, it follows from~\textup{Ex2$\.$(b)} that if in
the diagram~\eqref{pull-diagram} the triple $X'\rarrow Y\rarrow Y''$
is admissible and the morphism $X'\rarrow X$ is the kernel
of the morphism $X\rarrow X''$, then the object $Y$ is
the fibered product $Y=Y''\pro_{X''}X$.
\end{prop1}

\begin{proof}
 We have to show that for any object $Z\in\E$ there is a bijective
correspondence between the set of all morphisms $Y\ovrarrow{f} Z$
and the set of all pairs of morphisms $Y'\ovrarrow{g} Z$ and
$X\ovrarrow{h} Z$ which form a commutative diagram with
$X'\rarrow Y'$ and $X'\rarrow X$.
 Suppose a morphism $Y\ovrarrow{f} Z$ is annihilated by the
compositions with $Y'\rarrow Y$ and $X\rarrow Y$.
 By~Ex1, $f$ factorizes through the morphism $Y\rarrow X''$.
 Since the morphism $X\rarrow X''$ is assumed to be a cokernel,
it follows that $f=0$.

 Now let us assume that we are given a compatible pair of morphisms
$Y'\ovrarrow{g} Z$ and $X\ovrarrow{h} Z$.
 Applying~Ex2 to the admissible triple $Y'\rarrow Y\rarrow X''$
and the morphism $Y'\ovrarrow{g} Z$, we obtain an admissible
triple $Z\rarrow T\rarrow X''$ together with the commutative diagram
      \begin{equation*}
        \thrfrdiag
        \begin{diagram}
          \node{X'} \arrow{e} \arrow{s}
          \node{X}  \arrow{e} \arrow{s}
          \node{X''}
            \\
          \node{Y'} \arrow{e} \arrow{s,l}{g}
          \node{Y}  \arrow{ne} \arrow{s}
            \\
          \node{Z}  \arrow{e}
          \node{T}  \arrow{nne} 
        \end{diagram}      
      \end{equation*}

 Consider the morphism $\chi\:X\rarrow T$ that is equal
to the difference of two compositions
$\chi = [X\rar Y\rar T] - [X\ovrar{h}Z\rar T]$.
 We have $[X'\rar X\ovrar{\chi}T]=0$, since
$[X'\rar X\rar Y\rar T] = [X'\rar Y'\ovrar{g}Z\rar T]$
and $[X'\rar Y'\ovrar{g}Z] = [X'\rar X\ovrar{h}Z]$.
 Therefore, the morphism $\chi$ factorizes into the composition
$\chi=[X\rar X''\ovrar{\phi}T]$ for some morphism
$X''\ovrarrow{\phi}T$.
 The composition $X''\ovrarrow{\phi} T\rarrow X''$ is the identity,
since $[X\rar X''\ovrar{\phi}T\rar X''] =
[X\ovrar{\chi}T\rar X''] = [X\rar Y\rar T\rar X''] -
[X\ovrar{h}Z\rar T\rar X''] = [X\rar X'']$
and the morphism $X\rarrow X''$ is surjective.
 It follows that the morphism $\id_T - [T\rar X''\ovrar{\phi}T]$
annihilates $T\rarrow X''$ and, therefore, factorizes into
the composition $[T\ovrar{\psi}Z\rar T]$ for some morphism
$T\ovrarrow{\psi}Z$.

 Let us check that the morphism $f = [Y\rar T\ovrar{\psi}Z]$
is the desired one.
 We have $[Y\ovrar{f}Z\rar T] = [Y\rar T\ovrar{\psi}Z\rar T] =
[Y\rar T] - [Y\rar T\rar X''\ovrar{\phi}T]=
[Y\rar T] - [Y\rar X''\ovrar{\phi}T]$.
 Now $[X\rar Y\ovrar{f}Z\rar T] = [X\rar Y\rar T] - 
[X\rar Y\rar X''\ovrar{\phi} T] = [X\rar Y\rar T] - 
[X\rar X''\ovrar{\phi}T] = [X\rar Y\rar T] - [X\ovrar{\chi} T] =
[X\ovrar{h}Z\rar T]$, which implies
$[X\rar Y\ovrar{f} Z] = [X\ovrar{h}Z]$.
 Similarly, $[Y'\rar Y\ovrar{f} Z\rar T] =
[Y'\rar Y\rar T] - [Y'\rar Y\rar X''\ovrar{\phi} T] =
[Y'\rar Y\rar T] = [Y'\ovrar{g} Z\rar T]$, hence
$[Y'\rar Y\ovrar{f}Z]=[Y'\ovrar{g} Z]$.
\end{proof}

\begin{cor}
 The conditions \textup{Ex2$\.$(a)} and~\textup{Ex$2'$(a)} are
equivalent modulo \textup{Ex0\+-Ex1}.
 Analogously, the conditions \textup{Ex2$\.$(b)}
and~\textup{Ex$2'$(b)} are equivalent modulo \textup{Ex0\+-Ex1}. \qed
\end{cor}

\begin{prop2}
 If the axioms \textup{Ex0\+-Ex1} are satisfied, then
the axiom~\textup{Ex$2''$} holds if and only if the additive
category $\E$ is semi-saturated and the axiom~\textup{Ex2} holds.
\end{prop2}

\begin{proof}
 It follows from Proposition~1 that any of the conditions
Ex2$\.$(a) or Ex2$\.$(b) implies~Ex$2''$(c).
 We already know that Ex2 is equivalent to Ex$2'$, and it is proven
in~\cite{Kel,Kel2} that Ex$2'$ implies Ex$2''$(a-b) if $\E$ is
semi-saturated.
 Conversely, is easy to see~\cite{Kel2} that $\E$ is semi-saturated
whenever any of the conditions Ex$2''$(a) or Ex$2''$(b) holds.
 We will deduce Ex2 from Ex$2''$ below.

 Let $X'\rarrow X\rarrow X''$ be an admissible triple and
$X'\rarrow Y'$ be a morphism.
 By the condition Ex$2''$(a), the morphism $X'\ovrarrow{(-1,1)}
Y'\op X$ is an admissible monomorhism.
 Let $Y'\oplus X\ovrarrow{(1,1)}Y$ be its cokernel.
 Clearly, there exists a unique morphism $Y\rarrow X''$ such that
the diagram~\eqref{push-diagram} is commutative and the composition
$Y'\rarrow Y\rarrow X''$ is zero.
 By the condition Ex$2''$(b), the morphism $Y\rarrow X''$ is
an admissible epimorphism.
 It remains to show that the morphism $Y'\rarrow Y$ is its kernel.
 It is clear that the morphism $Y\rarrow X''$ is the cokernel of
the morphism $Y'\rarrow Y$.

 Let $K\rarrow Y$ be the kernel of the morphism $Y\rarrow X''$;
then there is a natural morphism $Y'\rarrow K$.
 First let us show that the latter morphism is surjective.
 Suppose that the composition $Y'\rarrow K\rarrow Z$ vanishes for
a certain morphism $K\rarrow Z$.
 Repeating the above construction starting with the admissible
triple $K\rarrow Y\rarrow X''$ and the morphism $K\rarrow Z$,
we obtain a commutative diagram
       \begin{equation*}
        \thrfrdiag
        \begin{diagram}
          \node{Y'} \arrow{e}
          \node{K} \arrow{e} \arrow{s}
          \node{Y}  \arrow{e} \arrow{s}
          \node{X''}
            \\
          \node[2]{Z} \arrow{e} 
          \node{T}   
        \end{diagram}      
      \end{equation*}    
where the triple $K\ovrarrow{(-1,1)}Z\op Y\ovrarrow{(1,1)} T$
is admissible.
 Since the morphism $K\rarrow Y$ is injective, the morphism
$Z\rarrow T$ is injective, too.
 Since the composition $[Y'\rar K\rar Y\rar T] = 
[Y'\rar K\rar Z\rar T]$ vanishes and the morphism $Y\rarrow X''$
is the cokernel of the morphism $Y'\rarrow Y$, the morphism
$Y\rarrow T$ factorizes into the composition $Y\rarrow X''\rarrow T$
for some morphism $X''\rarrow T$.
 Hence $[K\rar Z\rar T] = [K\rar Y\rar X''\rar T] = 0$, and
it follows that $[K\rar Z] = 0$.

 Now consider the commutative diagram 
\begin{equation*}
        \thrfrdiag
        \begin{diagram}
          \node{X'} \arrow{e} \arrow{se}
          \node{Y'\op X} \arrow{e} \arrow{s}
          \node{Y}  \\
          \node[2]{K\op X} \arrow{ne}
        \end{diagram}
      \end{equation*}
 The upper triple is admissible, and the morphism $X'\rarrow K\op X$
is an admissible monomorphism by~Ex$2''$(a).
 Using the fact that the vertical morphism is surjective, one can
check that the morphism $K\op X\rarrow Y$ is the cokernel of
the morphism $X'\rarrow K\op X$.
 It remains to apply the condition~Ex$2''$(c). 
\end{proof}

\begin{ex1}
 The following counterexample shows that the condition Ex$2''$(c) is
indeed necessary.
 Let $\A$ be the additive category whose objects are morphisms of
vector spaces $f\:V''\rarrow V'$ endowed with a subspace $V\sub\Im f$
and let $\T_\A$ be the class of all triples for which Ex1~holds.
 One can check that this class of admissible triples
satisfies~Ex0\+-Ex1, Ex$2''$(a-b), and~Ex3, but not Ex$2''$(c). 
\end{ex1}

 Note that if two classes of triples $\T'_\E$ and $\T''_\E$
satisfy the axioms~Ex0\+-Ex2, then their union $\T'_\E\cup\T''_\E$
also does.
 In particular, one has to use the axiom~Ex3 in order to prove
that the direct sum of two admissible triples is an admissible triple.
 The next counterexample shows that even if one assumes the
latter property, the conditions Ex0\+-Ex2 still wouldn't imply~Ex3.

\begin{ex2}
 Let $\B$ be the abelian category of 3-term sequences of vector spaces
and morphisms $V^{(1)}\rarrow V^{(2)}\rarrow V^{(3)}$
(the composition can be nonzero).
 There are six indecomposable objects in this category; we will
denote them by $E_1$, $E_2$, $E_3$, $E_{12}$, $E_{23}$, $E_{123}$,
where $\dim E^{(i)}_J=1$ for $i\in J$ and~$0$ otherwise.
 Let $\T_\B$ be the class of all exact triples $X'\rarrow X\rarrow X''$
in the abelian category $\B$ such that for any morphism $X'\rarrow E_3$
the triple $E_{23}\rarrow Y\rarrow X''$ induced from the original
triple using the composition of morphisms $X'\rarrow E_3\rarrow E_{23}$
is split.
 Then it is not difficult to check that the class $\T_\B$ satisfies
Ex0\+-Ex2 and is closed under direct sums, but the composition of
two admissible monomorphisms $E_3\rarrow E_{23}$ and
$E_{23}\rarrow E_{123}$ is not an admissible monomorphism
with respect to~$\T_B$.
\end{ex2}

 An additive functor between exact categories is said to be
\emph{exact} if it sends admissible triples to admissible triples.

\subsection{Examples}  \label{exact-cat-examples}
 Some examples of exact categories are listed below.
\smallskip

 (1) The trivial exact category structure: for any additive
category~$\A$, the class of all split triples $X'\rarrow X'\op X''
\rarrow X''$ satisfies the axioms of an exact category.
 
 (2) The canonical exact category structure on an abelian category:
the class of all short exact sequences in any abelian category
satisfies the axioms Ex0\+-Ex3.

 (3) Let $(\A,\.\T_\A)$ be an exact category and $\E\sub\A$ be a full
additive subcategory.
 Then $\E$ is called a \emph{full exact subcategory} of $\A$ if
the class $\T_\E$ of all triples in the category $\E$ which belong
to $\T_\A$ is an exact category structure on $\E$.
 Any of the following two conditions is sufficient for $\E$ to be
a full exact subcategory in~$\A$:
(a)~the subcategory $\E$ contains the middle term $X$ of any
admissible triple $X'\rarrow X\rarrow X''$ in the category~$\A$
whenever it contains both terms $X'$ and~$X''$; or
(b)~the subcategory $\E$ contains the remaining term of an admissible
triple in the category~$\A$ whenever it contains the middle term
and one of the other terms.

 (4) Let $\A$ and $\B$ be exact categories, $\C$ be an additive
category, and $\alpha\:\A\rarrow\C$, \ $\beta\:\B\rarrow\C$ be
additive functors sending admissible triples to triples
satisfying~Ex1.
 Consider the category $\E$ whose objects are pairs of objects
$A\in\A$ and $B\in\B$ together with an isomorphism $\alpha(A)\simeq
\beta(\B)$.
 Set a triple in $\E$ to be admissible if its images in $\A$ and
$\B$ are admissible.
 Then $\E$ is an exact category.

 (5) A \emph{filtered object} $X$ in an exact category~$\E$ is
a collection of objects $\gr^{a,b}X\in\E$ for all $a\le b\in\Z$ and
morphisms $\gr^{c,d}X\rarrow\gr^{a,b}X$ for all $a\le c$ and
$b\le d$ compatible with the compositions and such that all
the triples $\gr^{b+1,\.c}X\rarrow \gr^{a,c}X\rarrow \gr^{a,b}X$
are admissible.
 A triple of filtered objects $X'\rarrow X\rarrow X''$ is 
admissible if all the triples $\gr^{a,b}X'\rarrow \gr^{a,b}X
\rarrow \gr^{a,b}X''$ are admissible in the category~$\E$, or
equivalently, all the triples  $\gr^aX'\rarrow \gr^aX\rarrow \gr^aX''$
are admissible, where one denotes $\gr^aX=\gr^{a,a}X$.
 A \emph{finitely filtered object} is a filtered object~$X$
such that one has $\gr^aX=0$ for all but a finite number
of indices~$a$.
 The category of (finitely) filtered objects in an exact category
is an exact category.

 (6) Let $\E$ be an exact category.
 Define a triple in $\E^\ss$ or $\E^\sat$ to be admissible if it is
a direct summand of an admissible triple in $\E$.
 Then $\E^\ss$ and $\E^\sat$ become exact categories, and $\E$ is
closed under extensions (cf.~(3a)) in each of them.
 Conversely, if $\A$ is an exact category and $\E\sub\A$ is a full
exact subcategory such that every object of $\A$ is a direct summand
of an object of $\E$, then $\E$ is closed under extensions in $\A$
if and only if every triple from $\T_\A$ is a direct summand of
a triple from $\T_\E$.

 (7) Let $\A$ be an additive category in which all morphisms admit
kernels and cokernels.
 For any morphism $f\:X\rarrow Y$ in $\A$ consider the natural
morphism $\Coim f=\Coker(\ker f)\rarrow\Im f=\Ker(\coker f)$.
 Then the morphism $\Coim f\rarrow \Im f$ is surjective for any
morphism $f$ in $\A$ if and only if the composition of any two
kernels is a kernel (of some morphism) in $\A$, if and only if
a right divisor of any kernel is a kernel, and if and only if any
two morphisms $X'\rarrow X$ and $X'\rarrow Y'$, the former of which
is a kernel, can be embedded into a commutative square $[X'\rar
X\rar Y] = [X'\rar Y'\rar Y]$, where the morphism $Y'\rarrow Y$
is injective.
 The dual assertions relate injectivity of the morphisms $\Coim f
\rarrow \Im f$ with the properties of cokernels.

 The class of all triples satisfying axiom~Ex1 defines an exact
category structure on $\A$ if and only if it satisfies axiom
Ex$2''$(c) and all the above properties of kernels and cokernels
hold in~$\A$.
 In this case, the additive category $\A$ is said to be
\emph{quasi-abelian}.  (Cf.~\cite{Rum}.)

 Example~\ref{exact-proofs}.1 shows that the requirement of axiom
Ex$2''$(c) cannot be dropped.
 For an example of an additive category where kernels do not have
the above properties, consider the abelian category $\B$ from
Example~\ref{exact-proofs}.2 and its full subcategory $\A$ whose
objects are the direct sums of all the indecomposables except $E_{12}$.
 Then the morphism $f\:E_3\rarrow E_{123}$ is a composition of two
kernels but not a kernel, and the morphism $\Coim f\rarrow \Im f$
is not surjective (moreover, it is a kernel).

 (8) One can show that any semi-saturated additive category~$\A$
admits a maximal exact category structure, that is a class of triples
satisfying~Ex0\+-Ex3 and containing any other such class of triples.

 More precisely, it is clear that in any additive category $\A$
there exists a maximal class of triples satisfying Ex0\+-Ex2.
 In fact, this class consists of all triples $X'\rarrow X\rarrow X''$
satisfying Ex1 such that for any morphisms $X'\rarrow Y'$
and $Z''\rarrow X''$ the fibered coproduct $Y=Y'\cop_{X'}X$ exists
and the morphism $Y'\rarrow Y$ is a kernel, the fibered product
$Z=Z''\pro_{X''}X$ exists and the morphism $Z\rarrow Z''$ is
a cokernel; and the fibered product $Z''\pro_{X''}Y$ exists.
 In this case, the fibered coproduct $Y'\cop_{X'}Z$ also exists and
is isomorphic to $Z''\pro_{X''}Y$ (equivalently, one could require
existence of the former and deduce existence of the latter together with
their isomorphism).
 
 We claim that when $\A$ is semi-saturated, this class of triples
$X'\rarrow X\rarrow X''$ also satisfies~Ex3.

\medskip

 All the above assertions are quite straightforward to prove, except
the ones from Examples~(7-8).
 The proofs of the latter are given below.

\begin{proof}[Proof of~\textup{(7)}]
 Let us prove the assertions from the first paragraph; the assertion
from the second paragraph will then follow from
Proposition~\ref{exact-proofs}.2.

 Since the morphism $X\rarrow\Coim f$ is a cokernel, the
morphism $\Coim f\rarrow\Im f$ is surjective if and only if
the morphism $X\rarrow\Im f$ is surjective.
 Let us show that this is the case if the composition of
two kernels is a kernel.
 Let $c\:Y\rarrow C$ be the cokernel of the morphism~$f$;
by the definition, the morphism $\im f\:\Im f\rarrow Y$ is
the kernel of~$c$.
 Let $d\:\Im f\rarrow D$ be the cokernel of the morphism
$X\rarrow\Im f$ and $l\:L\rarrow \Im f$ be the kernel of the
morphism~$d$.
 By the assumption, the composition $(\im f)\circ l\:L\rarrow Y$
should be a kernel.
 Since the morphism $X\rarrow \Im f$ factors through~$l$
and the composition $L\rarrow Y\rarrow C$ is zero, it is easy
to deduce that~$c$ is the cokernel of the morphism $L\rarrow Y$.
 Now it turns out that both the morphisms $\Im f\rarrow Y$
and $L\rarrow Y$ are the kernels of the morphism~$Y\rarrow C$,
hence the morphism $L\rarrow\Im f$ is an isomorphism and $d=0$.

 Conversely, let us show that the composition~$f$ of two kernels
$K\rarrow L$ and $L\rarrow M$ is a kernel whenever the morphism
$\Coim f\rarrow\Im f$ is surjective.
 The morphism~$f$ is equal to the composition
$K\rarrow\Im f\rarrow M$; by the assumption,
the morphism $K\rarrow\Im f$ is surjective.
 Since the morphism $\Im f\rarrow M$ is a kernel, it suffices
to check that the morphism $K\rarrow\Im f$ is an isomorphism.
 The latter is a particular case of the following general 
statement, which is not difficult to prove: a surjective right
divisor of a composition of several kernels is an isomorphism.

 Now let us show that a right divisor of a kernel is a kernel
provided that the composition of two kernels is a kernel.
 Let $f\:K\rarrow Y$ and $Y\rarrow Z$ be a pair of morphisms whose
composition $K\rarrow Z$ is the kernel of a morphism $Z\rarrow D$.
 Let $c\:Y\rarrow C$ be the cokernel of the morphism $K\rarrow Y$
and $\im f\:\Im f\rarrow Y$ be kernel of~$c$; then the morphism
$K\rarrow Y$ factors through~$\im f$.
 Since we have $[K\rar Y\rar Z\rar D]=0$, it is clear that
the composition $Y\rarrow Z\rarrow D$ factors through~$c$
and, consequently, $[\Im f\rar Y\rar Z\rar D]=0$.
 Since the morphism $K\rarrow Z$ is the kernel of the morphism
$Z\rarrow D$, there exists a morphism $\Im f\rarrow K$
such that $[\Im f\rar Y\rar Z]=[\Im f\rar K\rar Y\rar Z]$.
 We have constructed morphisms between~$K$ and~$\Im f$ in both
directions; since the morphism $K\rarrow Z$ is injective and
$[K\rar\Im f\rar K\rar Y\rar Z]=[K\rar\Im f\rar Y\rar Z]
=[K\rar Y\rar Z]$, it follows that $[K\rar\Im f\rar K]=\id_K$.
 Therefore, the morphism $K\rarrow\Im f$ is the embedding of
a direct summand, hence it is a kernel and the morphism
$f\:K\rarrow Y$ is the composition of two kernels $K\rarrow\Im f$
and $\Im f\rarrow Y$.

 Furthermore, assume that a right divisor of a kernel is a kernel.
 Let a morphism $K\rarrow L$ be the kernel of a morphism
$c\:L\rarrow C$ and a morphism $l\:L\rarrow M$ be the kernel
of a morphism $d\:M\rarrow D$.
 Consider the anti-diagonal morphism $(c,\.-l)\:L\rarrow C\op M$;
since its composition with the projection $C\op M\rarrow M$
is a kernel, the morphism $(c,\.-l)$ should be the kernel
of a certain morphism $(i,e)\:C\op M\rarrow E$.
 Since the morphism $l\:L\rarrow M$ is injective, it is easy
to see that the morphism $i\:C\rarrow E$ is injective, too.
 Now one can check that the composition $K\rarrow L\rarrow M$
is the kernel of the diagonal morphism $(d,e)\:M\rarrow D\op E$.
 We have also shown that any pair of morphisms $L\rarrow M$ and
$L\rarrow C$, where $L\rarrow M$ is a kernel, can be embedded into
a commutative square $[L\rar M\rar E]=[L\rar C\rar E]$ with
an injective morphism $C\rarrow E$.

 Finally, assume that the latter condition holds.
 Let $U\rarrow T$ and $T\rarrow V$ be a pair of morphisms whose
composition $U\rarrow V$ is the kernel of a morphism $V\rarrow W$.
 Consider the pair of morphisms $U\rarrow V$ and $U\rarrow T$, and
suppose that there is a commutative square $[U\rar V\rar Q]=
[U\rar T\rar Q]$ such that the morphism $T\rarrow Q$ is injective.
 Then one can check that the morphism $U\rarrow T$ is the kernel
of the morphism $T\rarrow Q\op W$ with the components $[T\rar Q]
- [T\rar V\rar Q]$ and $[T\rar V\rar W]$. 
\end{proof}

\begin{proof}[Proof of~\textup{(8)}] 
 To prove the assertion from the second paragraph, notice first
of all that the triples $Y'\rarrow Y\rarrow X''$ and $X'\rarrow Z
\rarrow Z''$ satisfy~Ex1.
 For $T=Z''\pro_{X''}Y$, the functoriality of fibered (co)products
provides two morphisms of triples with one common object
$(X'\to Z\to Z'') \rarrow (Y'\to T\to Z'')\rarrow (Y'\to Y\to X'')$.
 By the definition of $T$ and since $Y'=\Ker(Y\to X'')$, it follows
that $Y'=\Ker(T\to Z'')$.
 It remains to show that $T=Y'\cop_{X'}Z$; then it will follow that
the triple $Y'\rarrow T\rarrow Z''$ satisfies Ex1, and the class of
all triples of this form, where the original triple $X'\rarrow X
\rarrow X''$ is fixed and the morphisms $X'\rarrow Y'$ and
$Z''\rarrow X''$ vary, satisfies Ex2.
 
 For any morphism $Y'\rarrow Y'_2$, applying our construction to
the composition $Y\rarrow Y'\rarrow Y'_2$ in place of the morphism
$Y\rarrow Y'$, we obtain an object $T_2=Z''\pro_{X''}
(Y'_2\cop_{X'}X)$ together with a morphism of triples with one
common object $(Y'\to T\to Z'')\rarrow (Y'_2\to T_2\to Z'')$.
 Since $Y_2'=\Ker(T_2\to Z'')$ and $Z''=\Coker(X'\to Z)$, the argument
from the proof of Proposition~\ref{exact-proofs}.1 shows that
$T=Y'\cop_{X'}Z$.

 Let us turn to the assertion from the third paragraph.
 Whenever one is dealing with the axiom Ex3 in the assumption of
Ex0\+-Ex2, the following commutative
diagram~\eqref{composition-diagram} plays the key role.
\begin{figure}[htb]
\begin{equation} \label{composition-diagram}
\newcommand{\clap}[1]{\hbox to 0pt{\hss $#1$\hss}} 
         \begin{gathered}
          \begin{picture}(160,69)
            \put(0,0)   {\clap{U}}    \put(11,5.5){\vector(4,1){56}}
            \put(160,0) {\clap{V}}    \put(93,19.5){\vector(4,-1){56}}
            \put(40,30) {\clap{P}}    \put(93,26){\vector(4,1){17}}
            \put(120,30){\clap{Q}}    \put(50,30.25){\vector(4,-1){17}}
            \put(80,19) {\clap{H}}    \put(9,12.75){\vector(4,3){21}}
            \put(80,60) {\clap{W}}    \put(130,28.5){\vector(4,-3){21}}
                                  \put(49,42.75){\vector(4,3){21}}
                                  \put(90,58.5){\vector(4,-3){21}}
          \end{picture}
         \end{gathered}
\end{equation}
\end{figure}

 Given admissible monomorphisms $U\rarrow P$ and $P\rarrow H$, one
considers admissible triples $U\rarrow P\rarrow W$ and $P\rarrow H
\rarrow V$.
 Then one applies the property Ex2$\.$(a) or Ex$2'$(a) to the latter
admissible triple and the morphism $P\rarrow W$, obtaining
an admissible triple $W\rarrow Q\rarrow V$.
 Hence the above commutative diagram.
 It is easy to check that the triple $U\rarrow H\rarrow Q$
satisfies~Ex1. 
 The morphism $U\rarrow H$ is the composition of two admissible
monomorphisms, and the morphism $H\rarrow Q$ is a direct summand
of the composition of two admissible epimorphisms $P\oplus H
\rarrow W\oplus H\rarrow Q$ (see~\cite{Kel}).

 In the situation of Example~(8), one can readily check that
the triple $U\rarrow H\rarrow Q$ allows the application of
the push-forward procedure of axiom Ex$2'$(a) and then
the pull-back procedure of axiom Ex$2'$(b), and continues
to satisfy Ex1 after that.
\end{proof}

\subsection{Embedding theorem}  \label{quillen-embedding}
 A small additive category $\E$ endowed with a class of admissible
triples $\T_\E$ is an exact category if and only if there exists
an abelian category $\A$ and a fully faithful functor
$\rho\:\E\rarrow\A$ such that the subcategory $\rho(\E)\sub\A$
is closed under extensions and a triple $X'\rarrow X\rarrow X''$
in the category $\E$ belongs to $\T_\E$ if and only if its image
$\rho(X')\rarrow \rho(X)\rarrow \rho(X'')$ is a short exact
sequence in~$\A$.
 In other words, the exact category structure on~$\E$ is induced
from the canonical exact category structure on the abelian
category $\A$ in the sense of the above
example~\ref{exact-cat-examples}$\.$(3a).
 This is the statement of an embedding theorem first formulated
by Quillen in~\cite{Quil} and proven in detail in several papers,
including~\cite{TT}, \cite{Kel}, and~\cite{Bueh}.

 In fact, any small exact category has \emph{two} natural embeddings
to abelian categories satisfying all the properties mentioned above;
the two embeddings differ by the categorical duality.
 They are constructed in the following way.
 Consider the abelian category $\Fun_\ad(\E^\opp,\Ab)$ of contravariant
additive functors from $\E$ to the category of abelian groups~$\Ab$.
 A functor $F\:\E^\opp\rarrow\Ab$ is called \emph{left exact}
if for any admissible triple $X'\rarrow X\rarrow X''$ in
the category $\E$ the sequence of abelian groups
 $$
  0\lrarrow F(X'')\lrarrow F(X)\lrarrow F(X')
 $$
is exact.
 Let us denote by $\A'(\E)\sub\Fun_\ad(\E^\opp,\Ab)$ the full
additive subcategory of left exact functors and by
$\rho'_\E\:\E\rarrow\A'(\E)$ the additive functor sending
an object $X\in\E$ to the representable contravariant
functor $\Hom_\E({-},X)$.

 The category $\A'(\E)$ is abelian.
 This result is an additive version of the sheafification theory:
arbitrary additive functors play the role of presheaves and
the left exact functors are the sheaves.
 Note that the category $\A'(\E)$ is not an abelian subcategory
of the category $\Fun_\ad(\E^\opp,\Ab)$: the embedding functor
$\A'(\E)\rarrow\Fun_\ad(\E^\opp,\Ab)$ is only left exact.
 The functor $\Fun_\ad(\E^\opp,\Ab)\rarrow\A'(\E)$ left adjoint
to the embedding is exact, however; this is the additive
sheafification functor.
 Its construction is the main part of the proof.
 One can also view the category $\A'(\E)$ as the quotient category
of $\Fun_\ad(\E^\opp,\Ab)$ by the kernel of the sheafification.

 The second embedding $\rho''_\E\:\E\rarrow\A''(\E)$ is defined
in terms of the covariant representable functors; the category
$\A''(\E)$ is the full subcategory of $\Fun_\ad(\E,\Ab)^\opp$ whose
objects are all the left exact covariant functors.
 For example, if $\E$ is the category of finitely generated free
(or projective) left modules over a ring $R$ endowed with the trivial
exact category structure, then the functor $\rho'_\E$ is the obvious
embedding into the abelian category of all left $R$\+modules
and the functor $\rho''_\E$ is the embedding into the abelian category
opposite to the category of all right $R$\+modules given by the rule
$\rho''(M)=\Hom_R(M,R)^\opp$.
 The following result demonstrates the difference between the two
canonical embeddings $\rho'_\E$ and~$\rho''_\E$.

\begin{prop1}
 The functor~$\rho'_\E$ preserves kernels, and more generally, any
limits.
 In particular, for any morphism~$f$ in an exact category~$\E$,
the morphism $\rho'_\E(f)$ is injective if and only if a morphism~$f$
is injective.
 For any morphism~$f$ in an exact category~$\E$, the morphism
$\rho'_\E(f)$ is surjective if and only if $f$~is a left divisor
of some admissible epimorphism $fg$ in~$\E$.
 In particular, if the category~$\E$ is semi-saturated, then
the morphism $\rho'_\E(f)$ is surjective if and only if
a morphism~$f$ is an admissible epimorphism.
 The dual statements hold for the embedding~$\rho''_\E$.
\end{prop1}

\begin{proof}
 The first assertion holds since all limits exist in a Grothendieck
category $\A'(\E)$ and the embedding $\A'(\E)\rarrow
\Fun_\ad(\E^\opp,\Ab)$ preserves limits.
 The third one follows from exactness of the functor~$\rho'_\E$ with
respect to admissible triples in $\E$ and the next Proposition~2,
and the fourth one from axiom Ex$2''$(b).
\end{proof}

\begin{prop2}
 For any objects $A\in\A'(\E)$ and $X\in\E$ and any surjective
morphism $A\rarrow \rho'_\E(X)$ in $\A'(\E)$ there exists
an admissible epimorphism $Y\rarrow X$ in $\E$ and
a morphism $\rho'_\E(Y)\rarrow A$ in $\A'(\E)$ such that the triangle
$\rho'_\E(Y)\rarrow A\rarrow\rho'_\E(X)$ is commutative in $\A'(\E)$.
\end{prop2}

\begin{proof}
 See~\cite[Lemma~A.7.15]{TT} or~\cite[Lemma~A.22]{Bueh}.
\end{proof}

\subsection{Derived categories}  \label{exact-derived}
 Whenever one is working with exact categories, one is faced with
an unpleasant choice between making the (semi-)saturatedness
assumptions or doing without them.
 The former approach restricts generality, while the second one
complicates matters; both the restrictions and the complications
are felt as irrelevant and unnecessary.
 One ends up oscillating between the two ways by passing from
a category to its (semi-)saturation and back, thus experiencing
the worst aspects of both approaches.
 No good solution to this problem is known to the author.

 In the definition of the derived category of an exact category,
the (semi-)saturation problem presents itself in a particularly
complicated form.
 We attempt to clarify the issues involved in the exposition below.

 Given an additive category $\A$, we denote by $\K(\A)$
the homotopy category of (unbounded) complexes over~$\A$.
 The homotopy category of complexes bounded from above, below,
and both sides are denoted by $\K^-(\A)$, \ $\K^+(\A)$, and
$\K^b(\A)\sub\K(A)$.
 For $*=\empty$, $-$, $+$, or~$b$, we say that a complex $X^\bu$
over $\A$ is $*$\+bounded if $X$ is respectively any complex, or
a complex bounded from above, etc.

 Let $\E$ be an exact category.
 A complex $X^\bu$ over $\E$ is said to be \emph{exact} if there
exist objects $Z^i\in\E$ and admissible triples $Z^i\rarrow X^i
\rarrow Z^{i+1}$ such that the differentials $X^i\rarrow X^{i+1}$
in $X^\bu$ are equal to the compositions $X^i\rarrow Z^{i+1}
\rarrow X^{i+1}$.
 It is known~\cite{Neem,Bueh} that the cone of a closed morphism
of exact complexes is exact.
 A complex $X^\bu$ is called \emph{acyclic} if it is homotopy
equivalent to an exact complex.
 The full subcategory of $*$\+bounded acyclic complexes is denoted
by $\Ac^*(\E)\sub\K^*(\E)$.

\begin{lem} \
\begin{enumerate}
\renewcommand{\theenumi}{\arabic{enumi}}
 \item For any $*=\empty$, $-$, $+$, or~$b$, a $*$\+bounded complex
       over $\E$ is acyclic if and only if it is homotopy equivalent
       to a $*$\+bounded exact complex and if and only if it is
       a direct summand of a $*$\+bounded exact complex.
 \item For any exact category $\E$, the full subcategory
       $\Ac^*(\E)\sub\K^*(\E)$ is thick.
 \item An exact category $\E$ is saturated if and only if
       any acyclic complex over $\E$ is exact.
       A complex over $\E$ is acyclic if and only if it is exact
       as a complex over $\E^\sat$.
       For $*=-$, $+$, or~$b$, an exact category $\E$ is
       semi-saturated if and if any $*$\+bounded acyclic complex
       over~$\E$ is exact.
       A $*$\+bounded complex over $\E$ is acyclic if and only if
       it is exact as a complex over $\E^\ss$.
 \item For $*=\empty$, $-$, or~$+$, the natural functor $\K^*(\E)
       \rarrow\K^*(\E^\sat)$ is an equivalence of triangulated
       categories identifying $\Ac^*(\E)$ with $\Ac^*(\E^\sat)$.
       For any~$*$, the natural functor $\K^*(\E)\rarrow\K^*(\E^\ss)$
       is an equivalence of triangulated categories identifying
       $\Ac^*(\E)$ with $\Ac^*(\E^\ss)$.
\end{enumerate}
\end{lem}

\begin{proof}
 This lemma is just a somewhat more precise rephrasing of the results
of~\cite{Neem}.
 Notice that the canonical truncations are defined for exact complexes
over an exact category.
 In particular, any $*$\+bounded complex homotopy equivalent to
an exact complex is homotopy equivalent to a $*$\+bounded
exact complex.

 If $X^\bu$ is a $*$\+bounded complex over~$\E$ homotopy equivalent
to a $*$\+bounded exact complex $Y^\bu$ over~$\E$, then $X^\bu$ is
a direct summand of the $*$\+bounded exact complex
$Y^\bu\oplus\Cone(\id_{X^\bu})$ over~$\E$ \cite{Neem}.
 Clearly, any direct summand of a $*$\+bounded exact complex over
$\E^\sat$ is itself a $*$\+bounded exact complex over $\E^\sat$.
 If $X^\bu$ is a $*$\+bounded exact complex over $\E^\sat$, then
there is a $*$\+bounded contractible complex $Y^\bu$ over $\E^\sat$
such that $X^\bu\op Y^\bu$ is an exact complex over~$\E$.
 It is easy to show this using the fact that $\E$ is closed under
extensions in $\E^\sat$.
 Consequently, $X^\bu$ is homotopy equivalent to $X^\bu\op Y^\bu$.
 This suffices to prove parts~(1-2), the second assertion of~(3),
the ``if'' part of the fourth assertion of~(3), and the second
parts of both assertions of~(4).

 Similarly one proves the first parts of the assertions of~(4).
 The ``if'' parts of the first and third assertions of~(3) are
easy~\cite{Neem}, and the ``only if'' parts are equivalent to
the ``only if'' parts of the second and fourth assertions. 
 Finally, the ``only if'' part of the third assertion of~(3) follows
by induction from axiom Ex$2''$(a-b) or the last two assertions
of Proposition~\ref{quillen-embedding}.1.
\end{proof}

\begin{ex}
 Let $\E$ and $\E_i$, \ $i\in\Z$ be exact categories, $\Phi_i\:
\E_i\rarrow\E$ be exact functors, and $\F$ be the exact category
of finitely filtered objects in $\E$ with successive quotients
lifted to $\E_i$ as defined in Section~\ref{filtered-exact-secn}.
 Then a complex $X^\bu$ over $\F$ is exact (acyclic) if and only
if the complex of successive quotients $q_i(X^\bu)$ is exact
(acyclic) over $\E_i$ for every $i\in\Z$.
 In particular, if the exact category structures on $\E_i$ are
trivial, then the complex $X^\bu$ is acyclic if and only if
all the complexes $q_i(X^\bu)$ are contractible.
\end{ex}

 A morphism of complexes $X^\bu\rarrow Y^\bu$ in $\K(\E)$ is called
a \emph{quasi-isomorphism} if its cone is acyclic. 
 The (bounded or unbounded) \emph{derived categories} of an exact
category $\E$ are defined as the triangulated quotient categories
$\D(\E)=\K(\E)/\Ac(\E)$ or $\D^*(\E)=\K^*(\E)/\Ac^*(\E)$ for
$*=-$, $+$, or~$b$.
 We are not discussing the set-theoretical issue of existence of
localizations here; they certainly do exist when the category $\E$
is essentially small.

\begin{cor1}
 For any exact category $\E$ and any $*=-$, $+$, or~$b$, the natural
functor $\D^*(\E)\rarrow\D(\E)$ is fully faithful.
\end{cor1}

\begin{proof}
 This follows from the existence of canonical truncations of exact
complexes.
\end{proof}

\begin{cor2}
 For any small exact category $\E$, the functors $\D^-(\E)\rarrow
\D^-(\A'(\E))$ and\/ $\D^+(\E)\rarrow\D^+(\A''(\E))$ induced by
the canonical embeddings $\rho'\:\E\rarrow\A'(\E)$ and
$\rho''\:\E\rarrow\A''(\E)$ are fully faithful.
\end{cor2}

\begin{proof}
 This follows from Proposition~\ref{quillen-embedding}.2 and its
dual version.
\end{proof}

 Let $\E$ be a small exact category.
 By the definition, we put $\Ext^n_\E(X,Y)=\Hom_{\D^b(\E)}(X,Y[n])$.
 The composition of morphisms in the derived category defines
multiplication maps $\Ext^n_\E(Y,Z)\tm\Ext^m_\E(X,Y)\rarrow
\Ext^{n+m}_\E(X,Z)$.

\begin{prop}
 In any exact category~$\E$, the Ext groups with negative numbers
are zero, $\Ext^n_\E(X,Y)=0$ for $n<0$.
 The natural morphism $\Hom_\E(X,Y)\rarrow\Ext^0_\E(X,Y)$ is
an isomorphism; in other words, the functor $\E\rarrow\D^b(\E)$
is fully faithful.

 The Ext groups with positive numbers $n>0$ are computed by
the following Yoneda construction.
 Consider the set $\Yon^n_\E(X,Y)$ of all exact complexes $A^\bu$
over $\E$ such that $A^i=0$ for all\/ $i<0$ and\/ $i>n+1$, \ $A^0=Y$,
and $A^{n+1}=X$.
 Two elements $A$ and $B\in\Yon^n_\E(X,Y)$ are said to be equivalent
if there exists a third element $C\in\Yon^n_\E(X,Y)$ and two morphisms
of complexes $C\rarrow A$ and $C\rarrow B$, both inducing identity
morphisms on the terms $Y$ in degree~$0$ and $X$ in degree~$n+\nbk1$.
 This is indeed an equivalence relation and the quotient set
of the set of Yoneda extensions $\Yon^n_\E(X,Y)$ modulo this
equivalence relation is canonically bijective to $\Ext^n_\E(X,Y)$.
 Alternatively, one can use morphisms of exact complexes going in
the opposite direction, $A\rarrow C$ and $B\rarrow C$; this defines
the same equivalence relation.

 The addition in the Ext groups is given by the Baer sum
construction.
 The multiplication on the Ext groups corresponds to the obvious
composition operation on the Yoneda extensions.
\end{prop}

\begin{proof}
 This follows from the construction of the triangulated quotient
category~\cite{Ver} and the existence of canonical truncations of
exact complexes.
\end{proof}

 Another important explicit construction of the higher $\Ext$ groups
in an exact category is provided by
Corollary~\ref{exact-triangulated}.1 below.

\begin{cor3}
 The product of the elements in\/ $\Ext^1_\E(X,Z)$ and\/
$\Ext^1_\E(Z,Y)$ corresponding to admissible triples
$Z\rarrow V\rarrow X$ and $Y\rarrow U\rarrow Z$ in $\E$ vanishes in
$\Ext^2_\E(X,Y)$ if and only if the composition of morphisms
$U\rarrow Z\rarrow V$ can be factorized as $U\rarrow T\rarrow V$ in
such a way that the triples $Y\rarrow T\rarrow V$ and
$U\rarrow T\rarrow X$ are admissible in~$\E$.
\end{cor3}

\begin{proof}
 This is most simply deduced from the long exact sequence of $\Ext$
groups corresponding to an admissible triple, which in turn follows
from the definition of $\Ext_\E$ in terms of morphisms in $\D^b(\E)$.
\end{proof}

 An object of an exact category is called \emph{projective}
(\emph{injective}) if the functor of morphisms from (into) this
objects sends admissible triples to short exact sequences.
 A projective (injective) \emph{resolution} of an object, or a complex
of objects, in an exact category is a bounded above (below) complex of
projective (injective) objects endowed with a quasi-isomorphism into
(from) the given object or complex.
 Since any morphism into (from) an acyclic complex from (into)
a bounded above (below) complex of projective (injective) objects
is homotopic to zero, one can use projective (injective) resolutions
to compute the groups $\Ext$ in an exact category.

\subsection{Exact subcategories of triangulated categories}
\label{exact-triangulated}
 A full subcategory $\E$ in a triangulated category $\D$ is called
an \emph{exact subcategory} if $\Hom_\D(X,Z[-1])=0$ for all
$X$, $Z\in\D$ and the class of all triples $X\rarrow Y\rarrow Z$
in $\E$ that can be embedded into a distinguished triangle
$X\rarrow Y\rarrow Z\rarrow X[1]$ in $\D$ defines an exact
category structure on~$\E$.
 Notice that the former condition guarantees that the distinguished
triangle in the latter condition is unique if it exists.
 Axioms Ex0 and~Ex1 are satisfied automatically by the class of
admissible triples defined in the above way; it is only axioms
Ex2 and~Ex3 that have to be checked.

 A full subcategory $\E$ in a triangulated category~$\D$
is said to be \emph{closed under extensions} if it contains the
middle term~$Y$ of any distinguished triangle
$X\rarrow Y\rarrow Z\rarrow X[1]$ with $X$,~$Z\in\E$.
 Any full subcategory $\E\sub\D$ that is closed under extensions
and satisfies the condition that $\Hom_\D(X,Z[-1])=0$ for all
$X$, $Z\in\E$ is an exact subcategory of~$\D$~\cite{Dy}.
 One applies the octahedron axiom in $\D$ in order to show that
axioms Ex2 and~Ex3 hold in~$\E$.

 In particular, it follows from the description of $\Ext^n_\E$ in
Proposition~\ref{exact-derived} for $n=-1$, $0$,~$1$ that any
exact category $\E$ is an exact subcategory, closed under extensions,
in its derived category $\D^b(\E)$.
 For any exact subcategory $\E$ in a triangulated category $\D$,
there is a natural injective map of abelian groups $\Ext^1_\E(X,Z)
\rarrow\Hom_\D(X,Z[1])$ for all $X$, $Z\in\E$ assigning to
an admissible triple $X\rarrow Y\rarrow Z$ in $\E$ the rightmost map
of the distinguished triangle $X\rarrow Y\rarrow Z\rarrow Y[1]$.
 To check additivity, one can consider the direct sum of two
copies of either $X$ or $Z$ and use the functoriality, which is
verified using the condition of vanishing of $\Hom_\D(X,Z[-1])$.
 The maps $\Ext^1_\E(X,Z)\rarrow\Hom_\D(X,Z[1])$ are isomorphisms
for all $X$, $Z\in\E$ if and only if the subcategory $\E$ is closed
under extensions in~$\D$.

 Let $\E\sub\D$ be a small exact subcategory in a triangulated
category.
 Consider the big graded ring of higher $\Hom$ groups
$$
 \H(\E,\D)=\bop_{n=0}^\infty\H^n(\E,\D); \qquad
 \H^n(\E,\D)=(\Hom_\D(X,Y[n]))_{Y,X\in\E}.
$$
 Let $\H(\E)=\H(\E,\D^b(\E))$ denote the big graded ring of
$\Ext$ groups in an exact category~$\E$.
 There is a natural isomorphism of big rings $\H^0(\E)\simeq
\H^0(\E,\D)$ and a natural injective morphism of bimodules
$\H^1(\E)\rarrow\H^1(\E,\D)$.

\begin{prop}
 For any small exact subcategory $\E$ in a triangulated
category~$\D$, the multiplication maps $\H^1(\E)\ot_{\H^0(\E)}
\H^n(\E,\D)\rarrow\H^{n+1}(\E,\D)$ and\/ $\H^n(\E,\D)\ot_{\H^0(\E)}
\H^1(\E)\rarrow\H^{n+1}(\E,\D)$ are injective for all\/ $n\ge0$.
\end{prop}

\begin{proof}
 It is easy to check using the existence of finite direct sums in
the category $\E$ that any element of $\H^1(\E)\ot_{\H^0(\E)}
\H^n(\E,\D)$ has the form $\xi\ot\eta$ for some
$\xi\in\Ext^1_\E(Y,Z)$ and $\eta\in\Hom_\D(X,Y[n])$.
 Consider the distinguished triangle $Z\rarrow T \ovrarrow f Y
\rarrow Z[1]$ in $\D$ corresponding to the admissible triple
$Z\rarrow T\rarrow Y$ in $\E$ representing~$\xi$; we will denote
the morphism $Y\rarrow Z[1]$ also by~$\xi$.
 Now if the composition $\xi\circ\eta\:X\rarrow Z[n+1]$
is zero, there exists a morphism $\eta'\:X\rarrow T[n]$
such that $\eta=f\circ\eta'$.
 Then we have $\xi\ot\eta=\xi\ot f\eta'=\xi f\ot\eta'=0$
in the group $\H^1(\E)\ot_{\H^0(\E)}\H^n(\E,\D)$. 
\end{proof}

\begin{cor1}
 For any small exact category~$\E$, the big graded ring of
Ext groups $\H=\H(\E)$ is a free big graded ring generated by
the $\H^0$\+bimodule $\H^1$, that is the multiplication morphisms
$\H^1\ot_{\H^0}\ds\ot_{\H^0}\H^1\rarrow\H^n$ are isomorphisms.
\end{cor1}

\begin{proof}
 The morphisms $\H^1\ot_{\H^0}\H^n\rarrow\H^{n+1}$ are injective
by the above proposition and surjective by
Proposition~\ref{exact-derived}.
\end{proof}

\begin{cor2}
 For any exact subcategory $\E$ in a triangulated category $\D$,
there exists a unique sequence of morphisms of functors of two
arguments $\E^\opp\tm\E\rarrow\Ab$
 $$
  \theta^n=\theta^n_{\E,\D}\:\Ext^n_\E(X,Y)\lrarrow\Hom_\D(X,Y[n]),
  \qquad n\ge0,
 $$
with the following properties: the maps~$\theta^n$ are compatible
with the composition, the map~$\theta^0$ is the identity, and
the map~$\theta^1$ sends the Yoneda class of an admissible triple
$Y\rarrow T\rarrow X$ in the exact category~$\E$ to the third
arrow $X\rarrow Y[1]$ of the corresponding distinguished triangle
in the triangulated category~$\D$.

 Moreover, if the morphism of functors~$\theta^n_{\E,\D}$ is
an isomorphism for a certain $n\ge0$, then all
the maps~$\theta^{n+1}_{\E,\D}$ are injective.
 In particular, if $\E$ is closed under extensions in $\D$, then
$\theta^1_{\E,\D}$ is an isomorphism and $\theta^2_{\E,\D}$ is
injective; if all the maps~$\theta^n$ are surjective for $n\le m$,
then all of them are isomorphisms.
\end{cor2}

\begin{proof}
 Existence and uniqueness follow from Corollary~1; alternatively,
one can use the universal property of effaceable homological
functors (see~\cite{Grot}).
 Injectivity follows from the proposition.
\end{proof}

\begin{cor3}
 Let $\A$ be an exact category and $\E$ be its full exact subcategory
closed under extensions and such that any object of $\A$ is a direct
summand of an object of~$\E$
\textup{(}see~\textup{\ref{exact-cat-examples}$\.$(6)}\textup{)}.
 Then the natural maps\/ $\Ext^n_\E(X,Y)\rarrow\Ext^n_\A(X,Y)$ are
isomorphisms for all $X$, $Y\in\E$ and $n\ge0$. \qed
\end{cor3}

 A full subcategory $\A$ in a triangulated category $\D$ is the heart
of a bounded t\+structure on $\D$ if and only if $\A$ is closed under
extensions, one has $\Hom_\D(X,Y[n])=0$ for all $X$, $Y\in\A$ and 
$n<0$, the exact category structure on $\A$ is the canonical exact
category structure of an abelian category, and the triangulated
category $\D$ is generated by $\A$ \cite[Proposition 1.3.13]{BBD}.
 Thus all the above results are applicable in the case when $\E=\A$
is the heart of a t\+structure on a triangulated category~$\D$.

\Section{Silly Filtrations}  \label{silly-filtrations-appx}

 The results of this appendix elaborate on the condition that the
morphisms $\theta^n_{\E,\D}$ of Corollary~\ref{exact-triangulated}.2
are surjective (or equivalently, isomorphisms) for an exact
subcategory $\E\sub\D$ closed under extensions.
 It turns out that one does not have to assume that the groups
$\Hom(X,Y[-1])$ are zero in order to discuss this condition.

 We will freely use the $*$\+operation notation for classes of objects
in triangulated categories~\cite[1.3.9-10]{BBD}.
 Besides, will use the notation of~\ref{additive-categories} related
to semi-saturated completions and saturated closures.

\begin{prop1}
 Let $\D$ be a triangulated category, $\A\sub\D$ be a full subcategory,
and $\M=\bigcup_m\A^{*m}\sub\D$ be the minimal full subcategory
containing $\A$ and closed under extensions.
 Then for any $n\ge2$ the following three conditions are equivalent:
 \begin{enumerate}
 \renewcommand{\theenumi}{\alph{enumi}}
    \item Any morphism $X\rarrow Y[k]$ with $2\le k\le n$
          between two objects $X$, $Y\in\M$ 
          can be presented as the composition of
          a chain of morphisms $Z_{i-1}\rarrow Z_i[1]$
          with $Z_i\in\M$, \ $Z_0=X$, and $Z_k=Y$.
    \item Any morphism $A\rarrow Y[k]$ with $2\le k\le n$ between
          two objects $A\in\A$ and $Y\in\M$
          can be presented as the composition of 
          a morphism $A\rarrow Z[1]$ and a morphism $Z[1]\rarrow Y[k]$
          with $Z\in\M$.
    \item Any morphism $A\rarrow Y[k]$ with $2\le k\le n$ 
          between two objects $A\in\A$ and $Y\in\M$
          can be presented as the composition of 
          a morphism $A\rarrow Z[k-1]$ and a morphism
          $Z[k-1]\rarrow Y[k]$ with $Z\in\M$.
  \end{enumerate}
 Furthermore, any of the next two conditions implies the previous
three:
 \begin{enumerate}
 \renewcommand{\theenumi}{\alph{enumi}}
 \setcounter{enumi}{3}
    \item Any morphism $A\rarrow B[k]$ with $2\le k\le n$ between
          two objects $A$, $B\in\A$ can be presented as
          the composition of a morphism $A\rarrow Z[1]$ and
          a morphism $Z[1]\rarrow B[k]$ such that $Z\in\M$ and
          a cone of the morphism $A\rarrow Z[1]$ belongs to $\A[1]$.
    \item Any morphism $A\rarrow Y[k]$ with $2\le k\le n$ 
          between two objects $A$, $B\in\A$
          can be presented as the composition of 
          a morphism $A\rarrow Z[k-1]$ and a morphism
          $Z[k-1]\rarrow B[k]$ such that $Z\in\M$ and a cone of
          the morphism $Z[k-1]\rarrow B[k]$ belongs to $\A[k]$.
  \end{enumerate}
\end{prop1}

\begin{proof}
 A simple induction on~$n$ proves (c)$\implies$(b); including
the equivalence of (a-c) for $k\le n-1$ in the induction
assumption, we obtain (b)$\implies$(c).
 It is obvious that (a) implies (b) and~(c).

 To prove (c)$\implies$(a), reformulate~(c) as the assertion that
for any morphism $A\rarrow Y[k]$ with $A\in\A$ and $Y\in\M$ there
exists a distinguished triangle
$Y[k]\rarrow W[k]\rarrow Z[k]\rarrow Y[k+1]$ with $Z\in\M$
such that the composition $A\rarrow Y[k]\rarrow W[k]$ vanishes.
 The condition~(a) can be restated as the similar assertion with
$A\in\A$ replaced by $X\in\M$.
 Let $A_1\rarrow A\rarrow A_2\rarrow A_1[1]$ be a distinguished
triangle in $\D$.
 Assuming that both objects $A_1$ and $A_2$ have the above property
for any morphism $A_i\rarrow Y[k]$ with $Y\in\M$, we will show that
the object $A$ has the same property.
 Let $A\rarrow Y[k]$ be any morphism; consider the composition
$A_1\rarrow A\rarrow Y[k]$.
 Then there exists a distinguished triangle $Y[k]\rarrow W_1[k]
\rarrow Z_1[k]$ with $Z_1\in\M$ such that the composition
$A_1\rarrow A\rarrow Y[k]\rarrow W_1[k]$ vanishes.
 Hence the composition $A\rarrow Y[k]\rarrow W_1[k]$ factorizes
through the morphism $A\rarrow A_2$.
 Consider the morphism $A_1\rarrow W_1[k]$ that we have obtained.
 There exists a distinguished triangle $W_1[k]\rarrow W[k]\rarrow
Z_2[k]\rarrow W_1[k+1]$ with $Z_2\in\M$ such that the composition
$A_2\rarrow W_1[k]\rarrow W[k]$ vanishes.
 Then the composition $A\rarrow Y[k]\rarrow W_1[k]\rarrow W[k]$
also vanishes.
 By the octahedron axiom, the cone of the composition $Y[k]\rarrow
W_1[k]\rarrow W[k]$ is an extension of $Z_2[k]$ and $Z_1[k]$, so
it belongs to $\M[k]$.

 To prove (d)$\implies$(b), reformulate (d) as the assertion that
for any morphism $A\rarrow B[k]$ with $A$, $B\in\A$ there exists
a distinguished triangle $Z\rarrow C\rarrow A\rarrow W[1]$ with
$Z\in\M$ and $C\in\A$ such that the composition $C\rarrow A
\rarrow B[k]$ vanishes.
 Then argue as above.
 Now we know that (d)$\implies$(a), and the implication
(e)$\implies$(a) follows by duality.
\end{proof}

\begin{prop2}
 Let $\D$ be a triangulated category and $\M\sub\D$ be a full
subcategory closed under extensions.
 Then any morphism $X\rarrow Y[n]$ with $n\ge 2$ between two objects
$X$, $Y\in\D$ can be presented as the composition of a chain of
morphisms $Z_{j-1}\rarrow Z_j[1]$ with $Z_j\in\M$, \ $Z_0=X$, and
$Z_n=Y$ if and only if the following two conditions hold:
 \begin{enumerate}
 \renewcommand{\theenumi}{\roman{enumi}}
    \item One has $\M[n]*\M\sub\M*\M[1]*\ds*\M[n]$ for any $n\ge0$.
    \item Put $\M^{[a,b]}=\M[-b]*\ds*\M[-a]$ for any $a\le b$, \ 
          $\M^{\le b}=\bigcup_a\M^{[a,b]}$, and
          $\M^{\ge a}=\bigcup_b\M^{[a,b]}$.
          Then one should have
          $\M^{\ge a}\cap\M^{\le b}\sub (\M^{[a,b]})^\sat_\D$.
 \end{enumerate}
 Besides, the triangulated subcategory generated by $\M$ in $\D$
coincides with $\bigcup_{a,b}\M^{[a,b]}$ in this case.
\end{prop2}

 The key ideas of the proof are summarized in the following Lemma.

\begin{lem}
 Let $\D$ be a triangulated category. Then
 \begin{enumerate} 
 \renewcommand{\theenumi}{\arabic{enumi}}
    \item Suppose that a morphism $f\:A\rarrow B$ in $\D$ factorizes
          through an object $E$.
          Then $\{B\}*\{A[1]\}\ni\Cone(f)\in
          \{E\}*\{A[1]\}*\{B\}*\{E[1]\}$.
    \item Suppose that the cones of two morphisms $A\rarrow B$ and
          $C\rarrow D$ are isomorphic in $\D$.
          Then both morphisms factorize through an object $E$
          belonging to the intersection $\{A\}*\{D\}\cap\{C\}*\{B\}$.
    \item Suppose that an object $W$ is a direct summand of
          an extension of objects $A$ and $D$ and also a direct
          summand of an extension of objects $C$ and $B$.
          Then, in particular, there are morphisms $A\rarrow W$
          and $W\rarrow B$; suppose moreover, that their composition
          factorizes through an object $E$.
          Then $W$ is a direct summand of an object in
          $\{C\}*\{E\}*\{E\}*\{D\}$.
 \end{enumerate}
\end{lem}

\begin{proof}
 Part~(1): Consider the distinguished triangles
\begin{gather*}
 A\lrarrow E\lrarrow D\lrarrow A[1] \\
 B[-1]\lrarrow C\lrarrow E\lrarrow B.
\end{gather*}
 By the octahedron axiom, there are also distinguished triangles
\begin{gather*}
 A\lrarrow B\lrarrow K\lrarrow A[1] \\
 C\lrarrow D\lrarrow K\lrarrow C[1].
\end{gather*}
 So $D\in \{E\}*\{A[1]\}$ and $C[1]\in\{B\}*\{E[1]\}$, hence
$$
 \{B\}*\{A[1]\}\ni K\in\{D\}*\{C[1]\}\sub
 \{E\}*\{A[1]\}*\{B\}*\{E[1]\}.
$$

 Part~(2): By the octahedron axiom, there exist four distinguished
triangles as above and the morphisms $A\rarrow B$ and $C\rarrow D$
factorize through~$E$.
 Clearly, $E\in\{A\}*\{D\}\cap\{C\}*\{B\}$.

 Part~(3): By our assumption, there exist objects $S$, $T\in\D$
and distinguished triangles
\begin{gather*}
 A\lrarrow S\oplus W\lrarrow D\lrarrow A[1] \\
 B[-1]\lrarrow C\lrarrow T\oplus W\lrarrow B.
\end{gather*}
 Combining the given morphisms with appropriate signs, we get
a pair of morphisms $A\rarrow S\oplus T\oplus W\oplus E\rarrow B$
with zero composition.
 There are distinguished triangles
\begin{gather*}
 A\lrarrow S\oplus T\oplus W\oplus E\lrarrow T\oplus K\lrarrow A[1]
 \\
 E\lrarrow K\lrarrow D\lrarrow E[1] 
\end{gather*}
and
\begin{gather*}
 S\oplus L\lrarrow S\oplus T\oplus W\oplus E\lrarrow B\lrarrow
 S[1]\oplus L[1] \\
 C\lrarrow L\lrarrow E\lrarrow C[1]
\end{gather*}
and a pair of morphisms $A\rarrow S\oplus L$, \ $T\oplus K\rarrow B$
forming a morphism of distinguished triangles with a common vertex
$S\oplus T\oplus W\oplus E$.
 Applying the octahedron axiom, we obtain distinguished triangles
\begin{gather*}
 A\lrarrow S\oplus L\lrarrow F\lrarrow A[1] \\
 F\lrarrow T\oplus K\lrarrow B\lrarrow F[1]
\end{gather*}
 Finally, there is a morphism $T\oplus K\rarrow S[1]\oplus L[1]$
that can be obtained as the composition $T\oplus K\rarrow A[1]
\rarrow S[1]\oplus L[1]$ or $T\oplus K\rarrow B\rarrow
S[1]\oplus L[1]$.
 This morphism can be included in the distinguished triangle
$$
 T[-1]\oplus K[-1]\lrarrow S\oplus L\lrarrow S\oplus T\oplus W
 \oplus E\oplus F\lrarrow T\oplus K,
$$
where the morphisms $S\rarrow S\oplus T\oplus W\oplus E\rarrow T$
are just the obvious embedding and projection, while the components
$L\rarrow S$ and $T\rarrow K$ vanish.
 Hence we obtain the distinguished triangle
$$
 K[-1]\lrarrow L\lrarrow W\oplus E\oplus F\lrarrow K.
$$
 Now $L\in\{C\}*\{E\}$ and $K\in\{E\}*\{D\}$, hence
$W\oplus E\oplus F\in\{L\}*\{K\}\sub\{C\}*\{E\}*\{E\}*\{D\}$.
\end{proof}

\begin{proof}[Proof of Proposition~\textup{2}]
 First let us check the implication ``if''. 
 Let $X$, $Y\in\M$, and $f\:X\rarrow Y[n+1]$ be a morphism
in $\D$, where $n\ge 1$.
 Consider the object $Z=\Cone(f[-1])$; then there is a distinguished
triangle
$$
 Y[n]\lrarrow Z\lrarrow X\rarrow Y[n+1],
$$
so $Z\in\M[n]*\M$.
 By~(i), there exists a distinguished triangle
$$
 U\lrarrow Z\lrarrow V\lrarrow U[1],
$$
where $U\in\M$ and $V\in\M[1]*\ds*\M[n]$.
 So $Z$ is also a cone of the morphism $V[-1]\rarrow U$.
 According to part~(2) of the above Lemma, the morphism $f[-1]$
factorizes through an object $E\in\{X[-1]\}*\{U\}\cap\{V[-1]\}*\{Y[n]\}
\sub (\M[-1]*\M)\cap(\M*\ds*\M[n])$.
 By~(ii), $E\in \M^\sat_\D$.
 Now the morphism $f$ factorizes through $E[1]$, thus it also
factorizes through an object of $\M[1]$.

 It is not difficult to prove by induction that the minimal
subcategory of $\D$ containing $\M$, \ds, $\M[n]$ and closed under
extensions coincides with $\M*\ds*\M[n]$ for any $n\ge 0$ provided
that (i)~holds.
 Specifically, one proceeds by induction on~$n$, and then in 
the number of factors $\M[i]$ with the largest possible value
$i=n$ in an iterated extension.
 This proves the last assertion of the Proposition.

 Now let us prove ``only if''.
 To verify~(i), consider an object $Z\in\M[n]*\M$.
 As above, we have $Z=\Cone(f[-1])$ for some morphism
$f\:X\rarrow Y[n+1]$, where $X$, $Y\in\M$.
 Decompose the morphism $f$ as $X\rarrow E[1]\rarrow Y[n+1]$,
where $E\in\M$.
 By part~(1) of Lemma, $Z\in \{E\}*\{X\}*\{Y[n]\}*\{E[1]\}
\sub\M*\M*\M[n]*\M[1]$.
 A simple induction on~$n$ finishes the proof of~(i).

 To prove~(ii), let us first check that
$$
 (\M[-1]*\ds*\M[n-1])^\sat_\D\cap(\M*\ds*\M[n])^\sat_\D
 \sub(\M*\ds*\M[n-1])^\sat_\D.
$$
 Let $W$ be an object in the intersection; then $W$ is a direct
summand of an extension of $A=X[-1]\in\M[-1]$ and
$D\in\M*\ds*\M[n-1]$ and also a direct summand of an extension
of $C\in\M*\ds*\M[n-1]$ and $B=Y[n]\in\M[n]$.
 The morphism $A\rarrow B$ factorizes through an object $E\in\M$.
 According to part~(3) of Lemma, $W$ is a direct summand of
an object from $\{C\}*\{E\}*\{E\}*\{D\}\sub\M*\ds*\M[n-1]*\M
*\M*\M*\ds*\M[n-1]$.
 Thus $W\in(\M*\ds*\M[n-1])^\sat_\D$.
 A simple induction allows to deduce~(ii).
\end{proof}

 In some cases the condition~(ii) is satisfied automatically.
 First of all, it always holds when $\M$ is the heart of a
t\+structure on~$\D$ (see below).
 Secondly, it holds in the filtered case described in the next
proposition.

\begin{prop3}
 Let $\D$ be a triangulated category endowed with a sequence
of triangulated subcategories $\D_i\sub\D$, \ $i\in\Z$ such that\/
$\Hom_\D(X,Y)=0$ for any $X\in\D_i$ and $Y\in\D_j$ with $i>j$.
 Let $\M_i\sub\D_i$ be full subcategories closed under extensions.
 Then the minimal full subcategory containing all $\M_i$
and closed under extensions $\M=\bigcup_{i\le j}\M_j*\ds*\M_i\sub\D$
satisfies the condition~\textup{(ii)} of Proposition~\textup{2} for
$a=b$ if and only if all the subcategories $\M_i\sub\D_i$ do. 
 Consequently, if $\M_i\sub\D_i$ satisfy~\textup{(ii)} for $a=b$
and $\M\sub\D$ satisfies~\textup{(i)}, then $\M$ also
satisfies~\textup{(ii)}.

 Furthermore, if all $\M_i\sub\D_i$ satisfy the following stronger
version of condition~\textup{(ii)}, then so does $\M\sub\D$:
\begin{itemize}
 \item[(ii${}'$)] For any $a\le b$ and $c\le d$, the intersection
                  of $\M^{[a,b]}$ with the minimal full subcategory
                  of\/ $\D$, containing $\M[-c]$, \ds, $\M[-d]$ and
                  closed under extensions is equal to
                  $\M^{[a,b]\cap[c,d]}$, where $[a,b]\cap[c,d]$
                  is the intersection of the segments $[a,b]$
                  and $[c,d]$ in~$\Z$.
\end{itemize}
\end{prop3}

\begin{proof}
 To verify the first assertion, if suffices to consider the associated
graded object functors $\D\rarrow\D_i$.
 The second one follows from Proposition~2 and its proof in
the ``if'' direction, which only uses (ii) for $a=b$.
 To prove the last assertion, let us show that the intersection of
$\M[-1]*\ds*\M[n]$ with the minimal full subcategory of $\D$,
containing $\M[k]$ for $k\ge0$ and closed under extensions coincides
with $\M*\ds*\M[n]$.
 Let $W\in\M[-1]*\ds*\M[n]$; then there exists a distinguished
triangle
$$
 V[-1]\lrarrow U\lrarrow W\lrarrow V
$$
with $U\in\M[-1]*\M$ and $V\in\M[1]*\ds*\M[n]$.
 Assume that $W$ belongs to the minimal full subcategory of $\D$,
containing $\M[k]$ for $k\ge0$ and closed under extensions; then,
since $U\in \{V[-1]\}*\{W\}$, so does $U$.
 By our assumption, the images $U_i$ of the object $U$ under
the functors of associated graded objects $\D\rarrow \D_i$ belong
to $\M_i$.
 Thus $U\in\M$ and $W\in\M*\ds*\M[n]$.
 The dual argument shows that the intersection of $\M[-n]*\ds*\M[1]$
with the minimal full subcategory of $\D$, containing $\M[k]$ for
$k\le0$ and closed under extensions coincides with $\M[-n]*\ds*\M$.
\end{proof}

\begin{ex}
 The following counterexample shows, however, that the condition~(ii)
does not always hold for an exact subcategory $\M\sub\D$,
and moreover, (i) does not imply (ii) in this case.
 Let $\B$ be the abelian category of $3$\+term complexes of vector
spaces $V^{(1)}\rarrow V^{(2)}\rarrow V^{(3)}$ (the composition of
the two arrows must be zero).
 There are $5$ indecomposable objects in this category, denoted
$E_1$, $E_2$, $E_3$, $E_{12}$, $E_{23}$ (see
Example~\ref{exact-proofs}.2).
 Let $\M\sub\B$ be the full additive subcategory whose objects
are the direct sums of all the indecomposables except~$E_2$.
 Set $\D=\D^b(\B)$.
 Then $\M$ is closed under extensions in $\B$ and $\D$ and inherits
a trivial exact category structure, i.~e., all the exact triples in
$\M$ are split.
 One can check that $\M$ generates $\D$ and satisfies~(i), but
not~(ii).
\end{ex}

 Furthermore, under the assumptions similar to those mentioned above
the condition~(i) allows a very simple reformulation.

\begin{prop4}
 Let $\D$ be a triangulated category and $\M\sub\D$ be a full
subcategory closed under extensions.
 Assume that we are in one of the following two situations:
 \begin{enumerate}
 \renewcommand{\theenumi}{\alph{enumi}}
 \renewcommand{\labelenumi}{(\theenumi)}
    \item $\M$ is the heart of a bounded t\+structure on $\D$.
    \item $\D$ is endowed with a sequence of triangulated
          subcategories $\D_i\sub\D$, \ $i\in\Z$ such that\/
          $\Hom_\D(X,Y)=0$ for any $X\in\D_i$ and $Y\in\D_j$
          with $i>j$.
          Each $\D_i$ is equivalent to the bounded derived category
          $\D^b(\M_i)$ of an exact category $\M_i$.
          The subcategory $\M\sub\D$ is the minimal full
          subcategory containing all~$\M_i$ and closed under
          extensions, $\M=\bigcup_{i\le j}\M_j*\ds*\M_i$.
 \end{enumerate}
 Then the condition~\textup{(ii)} of Proposition~\textup{2} always
holds and \textup{(i)}~is equivalent to its following weaker form:
 \begin{itemize}
    \item[(i${}'$)] $\D = \bigcup_{a,b}\M^{[a,b]}$, where
                    $\M^{[a,b]}=\M[-b]*\ds*\M[-a]$ for any $a\le b$.    
 \end{itemize}
\end{prop4}

\begin{proof}
 In the case~(b), according to Lemma~\ref{exact-derived}$\.$(4),
$\D_i\simeq\D^b(\M_i^\ss)$.
 Clearly, the condition~(ii) does not become weaker when one
replaces $\M_i$ with $\M_i^\ss$ and $\M$ with $\M^\ss$.
 According to Proposition~2, the sum total of the conditions
(i) and~(ii) does not change, either.
 So we can assume that $\M_i$ are semi-saturated.
 With this assumption, in both cases (a) and~(b) we will prove
the condition~(ii${}'$) of Proposition~3.
 Clearly, it implies both (ii) and (i${}'$)$\implies$(i).

 According to Proposition~3, in the case~(b) it suffices to consider
the case when $\D=\D_i=\D^b(\M_i)$ and $\M=\M_i$.

 As demonstrated in the proof of Proposition~3, it suffices to show
that the intersection of $\M[-1]*\M$ with the minimal full subcategory
of $\D$, containing $\M[k]$ for $k\ge0$ and closed under extensions
coincides with $\M$.
 Let $U\in\M[-1]*\M$; then there is a distinguished triangle
$$
 X[-1]\lrarrow Y[-1]\lrarrow U\lrarrow X
$$
with $X$, $Y\in\M$.
 Assume that $U$ belongs to the minimal full subcategory of $\D$,
containing $\M[k]$ for $k\ge0$ and closed under extensions.
 In the case~(a), we have $U\in\D^{\le0}$, and it follows immediately
from the long exact sequence of
t\+cohomology~\cite[Th\'eor\`eme~1.3.6]{BBD} that the morphism
$X\rarrow Y$ is surjective.
 In the case~(b), using, e.~g., Corollary~\ref{exact-derived}.2
and Proposition~\ref{quillen-embedding}.1, we can conclude that
the morphism $X\rarrow Y$ is an admissible epimorphism.
 Another way is to use the existence of canonical truncations of
exact complexes over $\M$ and axiom~Ex$2''$(b).
 Alternatively, one can use Proposition~2 in order to conclude that
$U\in\M^\sat_\D$, and then deduce the assertion that $X\rarrow Y$ is
an admissible epimorphism in $\M$ from the facts that it is such
in $\M^\sat_\D$ and the category $\M$ is semi-saturated.

 In both cases it follows that $U\in\M$. 
\end{proof}

\Section{Classical $\Kpi$ Conjecture}  \label{kpi1-appendix}

 Let $k$ be a commutative ring and $A$ be DG\+algebra over~$k$
endowed with two $\Z$\+valued grading, called the \emph{internal}
and the \emph{cohomological} gradings~\cite[Appendix~A]{Pkoszul}.
 The differential in $A$ raises the cohomological grading by~$1$
and preserves the internal grading; it also satisfies
the super-Leibniz rule with respect to the cohomological grading.
 The cohomological grading is denoted by the upper indices and
the internal grading by the lower ones.
 Let $X\maps X[1]$ denote the shift of cohomological grading of
internally graded complexes by~$1$ down (as usually) and
$X\maps X(1)$ denote the shift of their internal grading by~$1$ up.
 For an internally graded complex $M$ of $k$\+modules with
the differential~$d$, let $H^n(M)$ denote its internally graded
$k$\+module of cohomology of $M$ with respect to~$d$ in
the cohomological degree~$n$.

 Assume further that $A$ (with the differential forgotten) is
a flat bigraded $k$\+module, that $A_i=0$ for $i>0$, and
the complex $A_0$ is concentrated in the cohomological degree~$0$
and freely generated as a $k$\+module by the unit element of $A$.
 We will also consider internally graded DG\+coalgebras $C$ over~$k$
satisfying the same list of conditions, except that the counit map
$C_0\rarrow k$ is an isomorphism of complexes.
 The reduced bar- and cobar-constructions assign to an algebra $A$
of the above kind a coalgebra $C$ of the above kind and vice versa;
these constructions preserve quasi-isomorphisms of algebras and
coalgebras and are mutually inverse up to natural
quasi-isomorphisms~\cite[Theorem~A.1.1 and Remark~A.1]{Pkoszul}.

 Let $\D(A{-}\mathit{mod})$ denote the derived category of
internally graded left DG\+modules over~$A$.
 We are interested in the triangulated subcategory $\D\sub
\D(A{-}\mathit{mod})$ generated by the free DG\+modules $A(i)$, \
$i\in\Z$.
 Let $C$ be the reduced bar-construction of $A$ and
$\D(C{-}\mathit{comod})$ be the derived category of internally
graded left DG\+comodules over~$C$.
 No $k$\+flatness conditions are imposed on the terms of
DG\+(co)modules.
 The categories $\D(A{-}\mathit{mod})$ and $\D(C{-}\mathit{comod})$
are not equivalent in general; however, their full subcategories
formed by DG\+modules and DG\+comodules with cohomology bounded from
above in the internal grading are equivalent~\cite[Theorem~A.1.2 and
Remark~A.1]{Pkoszul}.
 It follows that the category $\D$ is equivalent to the full
triangulated subcategory of $\D(C{-}\mathit{comod})$ generated by
the trivial DG\+comodules $k(i)$; the equivalence sends
$A(i)$ to $k(i)$.
 Let $\M$ denote the minimal full subcategory of $\D$ containing
the DG\+modules $A(i)$ (or the DG\+comodules $k(i)$) and closed
under extensions. 

\subsection{Positive cohomology}  \label{positive-cohomology}
 Assume that the ring $k$ is Noetherian and has a finite homological
dimension.
 Then the triangulated subcategory $\D\sub\D(C{-}\mathit{comod})$
can be equivalently defined as the subcategory of all DG\+comodules
whose bigraded $k$\+modules of cohomology are finitely generated.

\begin{thm}
 One has $H^n(C)=0$ for all\/ $n>0$ if and only if any morphism
$X\rarrow Y[n]$ of degree $n\ge2$ in $\D$ between two objects
$X$, $Y\in\M$ can be presented as the composition of a chain of
morphisms $Z_{i-1}\rarrow Z_i[1]$ with $Z_i\in\M$, \ $Z_0=X$,
and $Z_n=Y$.
\end{thm}

\begin{proof}[Proof \textup{(cf.~\cite[Theorem~1.9]{Pkoszul})}]
 ``If'': it is easy to see that any DG\+comodule over $C$ is
a union of its DG\+subcomodules that are finitely generated
bigraded $k$\+modules.
 Let $X$ be such a DG\+subcomodule in the DG\+comodule $C$ over~$C$.
 By the last assertion of Proposition~\ref{silly-filtrations-appx}.2,
there exists a distinguished triangle $Y\rarrow X\rarrow Z\rarrow
Y[1]$ in $\D$ such that $H^n(Y)=0$ for $n\le0$ and $H^n(Z)=0$
for $n\ge 1$.
 Then the composition $Y\rarrow X\rarrow C$ vanishes as a morphism
in $\D$, since $\Hom_{\D(C{-}\mathit{comod})}(W,C)\simeq
\Hom_k(H^0(W_0),k)$ for any DG\+comodule $W$ over~$C$.
 On the other hand, the morphism $H^n(Y)\rarrow H^n(X)$ is surjective
for $n\ge1$.
 Hence the morphism $H^n(X)\rarrow H^n(C)$ vanishes for $n\ge1$.
 Since this holds for all $X$, it follows that $H^n(C)=0$.
 In fact, this argument does not depend on the assumption that
the internal grading of $C$ is negative.

 ``Only if'': by Proposition~\ref{silly-filtrations-appx}.4$\.$(b),
it suffices to prove that the condition (i$'$) is satisfied.
 Let $\D_{[-m,0]}$ denote the full triangulated subcategory of $\D$
generated by the DG\+comodules $k(i)$ with $-m\le i\le 0$ and
$\M_{[-m,0]}\sub\D_{[-m,0]}$ be the minimal full subcategory,
containing $k(i)$ and closed under extensions.
 We will show by induction on~$m$ that $\D_{[-m,0]}=\bigcup_{a,b}
\M_{[-m,0]}^{[a,b]}$.
 Since the homological dimension of $k$ is finite, it suffices to
check that for any object $X\in\D_{[-m,0]}$ there exists
a distinguished triangle
$$
 Y\lrarrow X\lrarrow Z\lrarrow Y[1]
$$
such that $Y$ belongs to the minimal full subcategory of
$\D_{[-m,0]}$ generated by $\M_{[-m,0]}[-n]$ with $n>0$ and closed
under extensions, while $H^n(Z)=0$ for $n>0$.

 Using the assumption that the internal grading of $C$ is negative,
one can show by induction on the internal grading that any
DG\+comodule $M$ over $C$ is the filtered inductive limit of its
DG\+subcomodules $M'$ such that the map $H(M')\rarrow H(M)$ is
injective and the underlying $k$\+module of $M'$ is finitely
generated.
 In particular, we can assume that our DG\+comodule $X$ is finitely
generated as a bigraded $k$\+module.
 Consider the complex of $k$\+modules $X_0$; clearly, there exists
a distingushed triangle $Y_0\rarrow X_0\rarrow Z_0\rarrow Y_0[1]$
in the derived category of $k$\+modules such that $Y_0$ is
a complex of finitely generated free $k$\+modules concentrated
in the cohomological degrees $n>0$, while $H^n(Z_0)=0$ for $n>0$.
 Represent the object $Z_0$ by a finite complex of finitely generated
$k$\+modules such that the morphism $X_0\rarrow Z_0$ in the derived
category comes from a morphism of complexes of $k$\+modules.

 Let $C_{\ge -m}$ be the DG\+subcomodule of the DG\+comodule $C$
consisting of all components of the internal degree $i\ge -m$.
 The morphism of complexes of $k$\+modules $X\rarrow Z_0$
induces a morphism of DG\+comodules $X\rarrow C_{\ge -m}\ot_k Z_0$.
 Let $W$ be a DG\+subcomodule of $C_{\ge -m}\ot_k Z_0$ such that
the map $H(W)\rarrow H(C_{\ge -m}\ot_k Z_0)$ is injective,
the underlying $k$\+module of $W$ is finitely generated, 
$W$ contains $C_0\ot_k Z_0$, and the morphism $X\rarrow C_{\ge -m}
\ot_k Z_0$ factorizes through~$W$.
 Then, in particular, one has $H^n(W)=0$ for $n>0$.
 Let $T[1]$ be the cone of the morphism $X\rarrow W$; then
the component $T_0$ is isomorphic to the complex $Y_0$ in
the derived category of $k$\+modules.
 There is a distinguished triangle $T_0\rarrow T\rarrow T_{\le -1}
\rarrow T_0[1]$ in the triangulated category~$\D$.
  
 By the assumption of induction on~$m$, there exists a distinguished
triangle $U\rarrow T_{\le-1}\rarrow V\rarrow U[1]$ in
$\D_{[-m,-1]}=\D_{[-m+1,0]}(-1)$ such that $U$ belongs to the minimal
full subcategory of $\D_{[-m,-1]}$ generated by $\M_{[-m,-1]}[-n]$
with $n>0$ and closed under extensions, while $H^n(V)=0$ for $n>0$.
 It remains to use the $*$\+associativity lemma in order to obtain
the desired distinguished triangle $Y\rarrow X\rarrow Z\rarrow Y[1]$.
\end{proof}

\subsection{Negative cohomology}
 Assume that $k$ is a field.

\begin{prop} \
\begin{enumerate}
\renewcommand{\theenumi}{\arabic{enumi}}
\item $\M$ is the heart of a t\+structure on $\D$ if and only if
$H^n(C)=0$ for $n<0$.
\item One has $\Hom_\D(X,Y[n])=0$ for all $X$, $Y\in\M$ and\/
$n<0$ provided that $H^n(C)=0$ for $n<-1$.
\end{enumerate}
\end{prop}

\begin{proof}
 One way to prove~(1) is to use Theorem~\ref{t-existence} for
$E_i=A(i)\in\D\sub\D(A{-}\mathit{mod})$.
 By that theorem, $\M$ is the heart of a t\+structure if and only if
$H^n(A/A_0)=0$ for $n\le 0$.
 It follows immediately from the forms of the reduced bar- and
cobar-constructions and their mutual inverseness up to
quasi-isomorphism that $H^n(A/A_0)=0$ for $n\le 0$ if and only if
$H^n(C/C_0)=0$ for $n<0$.

{\hbadness = 1500
 This approach depends on the assumption of negative internal grading;
the next one doesn't.
 To prove ``if'', replace $C$ with its canonical truncation
$\tau^{\ge 0}C = C/(d(C^{-1}) + \sum_{n<0}C^n)$, which is
quasi-isomorphic to~$C$; then the canonical truncations of
DG\+comodules over $\tau^{\ge0}C$ considered as complexes of
$k$\+vector spaces become also their canonical truncations as
DG\+comodules (cf.~\cite[Theorem~1.8$\.$(b)]{Pkoszul}).
 To check ``only if'', argue as in the proof of
Theorem~\ref{positive-cohomology}, the ``if'' part, using
the existence of a distinguished triangle $Y\rarrow X\rarrow Z
\rarrow Y[1]$ in $\D$ such that $H^n(Y)=0$ for $n\ge 0$ and
$H^n(Z)=0$ for $n\le -1$.
\par}

 Part~(2) follows immediately from the form of the cobar-complex
of~$C$.
\end{proof}

\begin{thm} \
\begin{enumerate}
\renewcommand{\theenumi}{\arabic{enumi}}
\item Assume that the bigraded $k$\+algebra $H(A)\simeq
\Hom^*_{\D(C{-}\mathit{comod})}(k,k(*))$ contains a central
nonzero-dividing element~$t$ of internal degree~$-e$ and
cohomological degree~$0$ such that the quotient algebra $H(A)/(t)$
modulo its zero internal degree component $H(A_0)=k$ is concentrated
in positive cohomological degrees.
 Then $H^n(C)=0$ for $n<-1$.
\item In the situation of~\textup{(1)}, assume that $e=1$ and
the quotient algebra $H(A)/(t)$ is a Koszul algebra over~$k$
concentrated on the diagonal where sum of the cohomological and
internal gradings is equal to zero.
 Let $D$ denote the bar-construction of $H(A)/(t)$.
 Then the bigraded coalgebra $H(C)$ is the tensor product of
the internally graded coalgebra $H(D)$, which is concentrated in
cohomological degree~$0$, and the exterior coalgebra $\Lambda(kt)$
with one cogenerator~$t$ in the internal degree~$-1$ and
the cohomological degree~$-1$.
\item Assume that the DG\+algebra $A$ contains an element~$t$
of internal degree~$-e$ and cohomological degree~$0$, annihilated
by the differential, central in $A$, and nonzero-dividing in both
$A$ and $H(A)$.
 Assume further that $H^n(A_i)=0$ for all\/ $n\le 0$ and\/ $0>i>-e$.
 Let $D$ denote the bar-construction of the DG\+algebra
$B=A/(t)$.
 Then one has $H^n(D)=\nbk0$ for all\/ $n>0$ if and only if
$H^n(C)=\nbk0$ for all\/ $n>0$, and $H^n(D)=\nbk0$ for all\/ $n<0$
if and only if $H^n(C)=\nbk0$ for all\/ $n<\nbk-1$.
 If both of these conditions hold, the bigraded coalgebra $H(C)$
is the tensor product of the internally graded coalgebra $H(D)$
and the exterior coalgebra $\Lambda(kt)$ with one cogenerator~$t$
in the internal degree~$-e$ and the cohomological degree~$-1$.
\end{enumerate}
\end{thm}

\begin{proof}
 The silly filtration of the reduced bar-complex $\dsb\rarrow
A/A_0\ot_k A/A_0\rarrow A/A_0\rarrow k$ provides a spectral sequence
starting from the cohomology of the bar-construction of $H(A)$ and
converging to the cohomology of the bar-construction of~$A$.
 This reduces part~(1) to part~(3), if one takes also into account
(the first proof of) part~(1) of the preceding proposition.
 To prove part~(2), notice that in its assumptions the cohomology
coalgebras of the bar-constructions of $A$ and $H(A)$ are isomorphic,
since the grading on the cohomology of the bar-construction of $H(A)$
coming from the filtration coincides with the internal grading.
 So it remains to prove~(3).

 Consider the morphism of DG\+algebras $f\:A\rarrow B$. 
 Applying the left derived functor of extension of scalars
$X\maps\mathbb L E_f(X)=B\ot^{\mathbb L}_A X$
\cite[Subsection~1.7]{Pkoszul} to the trivial DG\+module $k$ over
the DG\+algebra $A$, we obtain an object
$\mathbb L E_f(k)\in\D(B{-}\mathit{mod})$.
 Since the DG\+module $B$ over $A$ is quasi-isomorphic to the cone
of the morphism $A(-e)\ovrarrow{t} A$, the bigraded $k$\+vector space
$H(\mathbb L E_F(k))$ is two-dimensional, with one cohomology class
in the internal grading~$0$ and cohomological grading~$0$, and
another one in the internal grading~$-e$ and cohomological
grading~$-1$.
 For the reasons of the natural filtration associated with
the internal grading of DG\+modules over $B$, there is a distinguished
triangle
$$
 k(-e)[1]\lrarrow \mathbb L E_f(k)\lrarrow k\lrarrow k(-e)[2]
$$
in $\D(B{-}\mathit{mod})$.
 Applying the triangulated functor of DG\+module bar-construction to
this distinguished triangle, we obtain a distinguished triangle
$$
 D(-e)[1]\lrarrow C\lrarrow D\lrarrow D(-e)[2]
$$
in $\D(D{-}\mathit{comod})$ (for the computation of the (pre)image
of the extension of scalars DG\+module $\mathbb L E_f(k)$ under
the bar-construction, see~\cite[Proposition~6.9]{Pkoszul}).

 The vector space of morphisms $D\rarrow D(-e)[2]$ in
$\D(D{-}\mathit{comod})$ is isomorphic to $H^{-2}(D_{-e})^*$;
the linear function corresponding to a morphism is the induced map
$H^{-2}(D_{-e})\rarrow H^0(D_0)=k$.
 For the morphism in the above distinguished triangle, this map is
zero, since the map $H^0(D_0)\rarrow H^{-1}(C_{-e})$ sends the unit
element to the class of the element $t\in A/A_0[1]\sub C$.
 It suffices to check this in the case $A=k[t]$ and $B=k$, and
then use the functoriality of our construction of distinguished
triangle.
 It follows from the condition of vanishing of $H^n(A_i)$ for $n\le 0$
and $0>i>-e$ that this class in $H^{-1}(C_{-e})$ is nontrivial.

 Thus there is a short exact sequence of left $H(D)$\+comodules
$$
 0\lrarrow H(D(-e)[1])\lrarrow H(C)\lrarrow H(D)\lrarrow 0.
$$
 The ``if and only if'' assertions of part~(3) follow immediately.

 Now assume that $H^n(D)=0$ for all $n\ne0$.
 Define a decreasing filtration $F$ on the DG\+algebra $A$ by
the rule $F^iA=t^iA$.
 This filtration is finite on every internal degree component, so
there is a spectral sequence starting from the cohomology of
the bar-construction of $\gr_FA$ and converging to the cohomology
of the bar-construction of~$A$.
 The associated graded DG\+algebra $\gr_FA$ is isomorphic to
$B\ot_k k[t]$, hence the cohomology of the bar-construction of
$\gr_FA$ is isomorphic to $H(D)\ot_k\Lambda(kt)$.
 To check this, it suffices to construct the ``shuffle product''
morphism from the tensor product of bar-constructions to
the bar-construction of the tensor product~\cite[Proposition~1.1
of Chapter~3]{PP}, and then reduce the question to the case of
DG\+algebras with trivial multiplication and/or differential by
passing to associated graded objects with respect to appropriate
additional filtrations.
 The grading on the cohomology of the bar-construction of $\gr_FA$
induced by the filtration $F$ coincides with minus the cohomological
grading, hence the cohomology coalgebras of the bar-constructions
of $A$ and $\gr_FA$ are isomorphic.
\end{proof}

\begin{rem}
 In the situation of part (2) or~(3) of Theorem, one does \emph{not}
expect the DG\+coalgebra $C$ to be quasi-isomorphic to its
cohomology coalgebra $H(C)$, even though all the Massey
comultiplications on $H(C)$ clearly vanish.
\end{rem}

\Section{Triangulated Categories of Morphisms}
\label{categories-of-morphisms-appx}

 The results of~\ref{exact-triangulated} concerning the existence
of (canonically defined and compatible with the compositition)
morphisms of functors $\theta^n_{\E,\D}$ for an exact subcategory
$\E$ in a triangulated category $\D$ strongly suggest that there
should exist a triangulated functor $\Theta\:\D^b(\E)\rarrow\D$
compatible with the embeddings of $\E$ into $\D^b(\E)$ and~$\D$.
 Unfortunately, it seems to be impossible to construct such
a functor in general, because of the problems with nonfunctoriality
of the cone in a triangulated category.

 Here we will describe a refinement of the triangulated category
structure that allows for construction of a functor~$\Theta$.
 Two different versions of such a refinement were suggested in
the papers~\cite{Beil2,Neem2}; here we follow some kind of a middle
path between the two.
 Namely, we use a modified form of the argument of~\cite{Neem2} to
show that the entire filtered triangulated category structure as
defined in~\cite{Beil2} is not needed for the above-mentioned
purpose; it suffices to have the triangulated category of two-step
filtrations, or, which is the same, the triangulated category of
morphisms, or of distinguished triangles.
 Our goal is to prove that a functor $\Theta$ exists whenever
the triangulated category $\D$ can be obtained from the homotopy
category of (unbounded) complexes over some additive category by
passing to triangulated subcategories and quotient categories any
number of times.
 
\subsection{Filtered triangulated categories}
\label{filtered-triangulated}
 Let $[a,b]\sub\Z$ be a finite or infinite segment of the integers,
$-\infty\le a < b\le +\infty$.
 An \emph{$[a,b]$\+filtered triangulated category} $(\D,\D\F)$ is
the following set of data.
 A triangulated category $\D\F$ is endowed with a family of
triangulated subcategories $\D\F_i$, \ $i\in[a,b]$ such that
$\D\F$ is generated by $\D\F_i$ and
\begin{equation}  \label{ft-semiorthogonal}
 \Hom_{\D\F}(X,Y)=0\quad\text{for all $X\in\D\F_i$, \ $Y\in\D\F_j$,
 and $i>j$.}
\end{equation}

 For a segment $[c,d]\sub[a,b]$, denote by $\D\F_{[c,d]}$ the full
triangulated subcategory of $\D\F$ generated by $\D\F_i$ with
$i\in[c,d]$.
 A twist functor $X\maps X(1)$ acting from $\D\F_{[a,b-1]}$ to
$\D\F_{[a+1,b]}$ is given such that $\D\F_{i+1}=\D\F_i(1)$ for
all $a\le i\le b-1$.
 There is a natural transformation $\sigma_X\:X\rarrow X(1)$ for
all $X\in\D\F_{[a,b-1]}$ such that $\sigma_{X(1)}=\sigma_X(1)$
for all $X\in\D\F_{[a,b-2]}$.

 Finally, there is a functor $w\:\D\F\rarrow\D$ such that
$w(\sigma_X)$ is an isomorphism for all $X\in\D\F_{[a,b-1]}$.
 The restriction of the functor~$w$ to $\D\F_i$ is an equivalence
of categories $\D\F_i\simeq\D$, and the map
\begin{equation} \label{ft-hom-isomorph}
 \begin{array}{l}
 w\:\Hom_{\D\F}(X,Y)\lrarrow\Hom_\D(w(X),w(Y)) \\
 \text{is an isomorphism for all $X\in\D\F_{[a,i]}$ and
 $Y\in\D\F_{[i,b]}$}.
 \end{array}
\end{equation}

 It follows from~\eqref{ft-semiorthogonal} that there is
a triangulated functor of ``successive quotients''
$q=(q_i)_{i\in[a,b]}\:\D\F\rarrow\prod\D\F_i$
(see~\cite[Sections~1 and~4]{BoKa} or~\cite[Lemma~1.3.2]{BGSch}).

\begin{ex}
 Fix a segment $[a,b]$ as above.
 Let $\E$ be an exact category and $\F$ be the exact category of
finitely filtered objects $Z$ in $\E$ as defined
in~\ref{exact-cat-examples}$\.$(5) for which $\gr^iZ=0$ for
$i\notin[a,b]$.
 Set $\D=\D(\E)$ and $\D\F=\D(\F)$.
 Let $\D\F_i\sub\D\F$ be the image of the fully faithful triangulated
functor $\D\rarrow\D\F$ induced by the embedding of exact categories
$\E\rarrow\F$ assigning to an object $X\in\E$ the filtered object
$Z\in\F$ with $\gr^iZ=X$ and $\gr^jZ=0$ for $j\ne i$.
 The twist functor $Z\maps Z(1)$ is induced by the shift of
the filtration, and the natural transformation~$\sigma$ is defined
in the obvious way.
 Finally, let $w\:\D\F\rarrow\D$ be the triangulated functor induced
by the forgetful functor $\F\rarrow\E$ assigning to a finitely
filtered object $Z\in\F$ the stabilizing limit of $\gr^{i,j}X$ as
$i\to-\infty$ and $j\to\infty$.
 The condition~\eqref{ft-semiorthogonal} is easy to check.
 As to the condition~\eqref{ft-hom-isomorph}, it is obvious when
the exact category structure on $\E$ is trivial, so $\D=\K(\E)$
and $\D\F=\K(\F)$.
 The general case follows from the next proposition.
\end{ex}

\begin{prop}
 Let $(\D,\D\F)$ be an $[a,b]$\+filtered triangulated category and
$\C\sub\D$ be a triangulated subcategory.
 Denote by $\C\F\sub\D\F$ the triangulated subcategory consisting
of all the objects $X\in\D\F$ such that $w(q_i(X))\in\C$ for all
$i\in[a,b]$.
 Then
\begin{enumerate}
\renewcommand{\theenumi}{\arabic{enumi}}
\item The pair $(\C,\C\F)$ with the additional data induced from
that for the pair $(\D,\D\F)$ is an $[a,b]$\+filtered triangulated
category.
\item If $\C\sub\D$ is a thick subcategory, then so is
$\C\F\sub\D\F$, and the pair $(\D/\C\;\D\F/\C\F)$ with
the additional  data induced from that for the pair $(\D,\D\F)$
is an $[a,b]$\+filtered triangulated category.
\end{enumerate}
\end{prop}

\begin{proof}
 Part~(1) is obvious; let us prove part~(2).
 Condition~\eqref{ft-semiorthogonal} and the equivalence
$w\:(\D\F/\C\F)_i\simeq\D/\C$ follow
from~\cite[Subsection~1.3.3]{BGSch}, so it remains
to check~\eqref{ft-hom-isomorph}.

 Let us prove surjectivity.
 A morphism in the category $\D/\C$ is represented by a fraction
$w(X)\larrow Q\rarrow w(Y)$, where $A=\Cone(Q\to w(X))\in\C$.
 Let $\kappa_i\:\D\simeq\D\F_i\rarrow\D\F$ be the embedding inverse
to the equivalence $w\:\D\F_i\rarrow\D$.
 By the condition~\eqref{ft-hom-isomorph} for the pair $(\D,\D\F)$,
there is a morphism $X\rarrow \kappa_i(A)$ in $\D\F$ corresponding
to the morphism $w(X)\rarrow A$ in~$\D$.
 Let $T=\Cone(X\to \kappa_i(A))[-1]$; since $w\kappa_i(A)\simeq A$,
we have $w(T)\simeq Q$.
 Furthermore, $T\in\D\F_{\le i}$, hence there is a morphism
$T\rarrow Y$ corresponding to the morphism $Q\rarrow X$.
 The fraction $X\larrow T\rarrow Y$ provides the desired morphism
$X\rarrow Y$ in $\D\F/\C\F$.

 To check injectivity, consider a fraction $X\larrow T\rarrow Y$
in $\D\F$ representing a morphism $X\rarrow Y$ in $\D\F/\C\F$.
 First of all let us show that one can choose $T\in\D\F_{[a,i]}$.
 Let $C=\Cone(T\to X)\in\C\F$; then there exists a distinguished
triangle 
$$
 C_{\ge i+1}\rarrow C\rarrow C_{\le i}\rarrow C_{\ge i+1}[1],
 \quad C_{\ge i+1}\in\D\F_{[i+1,b]}, \ \ C_{\le i}\in\D\F_{[a,i]}.
$$
 In fact, it will be only important for us that $C_{\ge i+1}\in
\D\F_{[i,b]}$.
 It follows from~\eqref{ft-hom-isomorph} that the morphism
$C_{\le i}[-1]\rarrow C_{\ge i+1}$ decomposes as $C_{\le i}[-1]
\rarrow\kappa_i w(C_{\ge i+1}) \rarrow C_{\ge i+1}$.
 Set $C'=\Cone(C_{\le i}[-1]\to \kappa_i w(C_{\ge i+1}))$
and $E=\Cone(\kappa_i w(C_{\ge i+1})\to C_{\ge i+1})$.
 By~\eqref{ft-hom-isomorph}, one has $\Hom_{\D\F}(X,E[n])=0$
for all~$n$.
 By the octahedron axiom, there is a distinguished triangle
$$
 C'\lrarrow C\lrarrow E\lrarrow C'[1],
$$
hence the morphism $X\rarrow C$ factorizes through the morphism
$C'\rarrow C$.
 Now let $T'=\Cone(X\to C')[-1]$; then the morphism $T'\rarrow X$
factorizes through the morphism $T\rarrow X$.
 By the construction, $C'\in\C\F_{[a,i]}$ and $T'\in\D\F_{[a,i]}$.

 Finally, suppose that the morphism $X\rarrow Y$ in $\D\F/\C\F$
represented by a fraction $X\larrow T\rarrow Y$ with
$T\in\D\F_{[a,i]}$ maps to a zero morphism in $\D/\C$.
 This means that the morphism $w(T)\rarrow w(Y)$ factorizes
through an object $A\in\C$.
 Then, by~\eqref{ft-hom-isomorph}, the morphism $T\rarrow Y$
factorizes through the object $\kappa_i(A)\in\C\F$, hence
our morphism $X\rarrow Y$ in $\D\F/\C\F$ vanishes.
\end{proof}

\begin{rem}
 It is also easy to see that if $(\D,\D\F)$ is an $[a,b]$\+filtered
triangulated category, then so is $(\D^\sat,\D\F^\sat)$, where
the triangulated category structure on the saturation of a triangulated
category~\cite{BS} is presumed for $\D^\sat$ and $\D\F^\sat$.
\end{rem}

\subsection{Category of morphisms and realization functor}
 From now on we will consider $[0,1]$\+filtered triangulated
categories, or triangulated categories of two-step filtrations.
 In this case, the twist functor $Z\maps Z(1)$ and the natural
transformation~$\sigma$ on $\D\F$ can be omitted in
the definition of a filtered triangulated category, as they
are determined by the remaining data.

 Recall the notation $\kappa_i$, \ $i=0$,~$1$ for the functors
$\D\simeq\D\F_i\rarrow\D\F$ introduced in the proof of
Proposition~\ref{filtered-triangulated}.
 The condition~\eqref{ft-hom-isomorph} simply means that
the functor~$w$ is right adjoint to~$\kappa_0$ and left
adjoint to~$\kappa_1$.
 It follows from~\eqref{ft-semiorthogonal} that the functor
$w\circ q_0$ is left adjoint to~$\kappa_0$ and the functor
$w\circ q_1$ is right adjoint to~$\kappa_1$.

 Finally, there is a functor $\lambda\:\D\rarrow\D\F$ left
adjoint to the functor $w\circ q_0$.
 It is given by the rule $\lambda(X)=\Cone(\kappa_0(X)\to
\kappa_1(X))[-1]$, where the morphism $\kappa_0(X)\rarrow\kappa_1(X)$
corresponds to the identity endomorphism of~$X$ under
the isomorphism~\eqref{ft-hom-isomorph}.
 The functor $\lambda[1]$ is right adjoint to the functor $w\circ q_1$.

 The situation has a triangular symmetry.
 Namely, setting $\D\F'_0=\lambda(\D)$, \ $\D\F'_1=\D\F_0$, and
$w'=w\circ q_0$ defines a new $[0,1]$\+filtered triangulated category
structure on the pair $(\D,\D\F)$, and the third such structure is
given by $\D\F''_0=\D\F_1$, \ $\D\F''_1=\lambda(\D)$, and
$w''=w\circ q_1$.
 Ideally, one would like these three sets of data to be permuted by
a triangulated autoequivalence of $\D\F$ whose third power would be
isomorphic to the shift functor $Z\maps Z[1]$, but it is not clear
how to obtain such an autoequivalence from our
conditions~(\ref{ft-semiorthogonal}--\ref{ft-hom-isomorph}).

\begin{thm}
 Let $(\D,\D\F)$ be a $[0,1]$\+filtered triangulated category and
$\E$ be an exact subcategory in $\D$ in the sense
of~\textup{\ref{exact-triangulated}} such that\/ $\Hom_\D(X,Y[n])=0$
for all $X$, $Y\in\E$ and\/ $n<0$.
 Then there exists a natural triangulated functor $\Theta:\D^b(\E)
\rarrow\D$ whose restriction to $\E$ is the identity.
\end{thm}

\begin{proof}
 We start with constructing a functor from the category of bounded
complexes $\Com^b(\E)$ over $\E$ (and closed morphisms between them)
into~$\D$.
 Let $\Com^{[c,d]}(\E)$ denote the category of complexes concentrated
in the interval $[c,d]$ of cohomological degrees.
 We proceed by induction in~$n$ constructing a compatible system of
functors $T_n\:\Com^{[-n,0]}(\E)\rarrow\D$ such that
$\Im T_n\sub\E[n]\dsb*\E\sub\D$ in the notation
of~\cite[1.3.9-10]{BBD}.
 The embedding $E\rarrow\D$ provides the functor~$T_0$.

 Given the functor $T_{n-1}$, we construct the functor $T_n$
as follows.
 Let $C^\bu$ be a complex over $\E$ concentrated in the degrees
$[-n,0]$; then $C$ is the cone of a closed morphism of complexes
$A^\bu\rarrow C^0$, where the complex $A^\bu$ is concentrated in
the degrees $[-n+1,0]$ and the object $C^0\in\E$ is considered as
a complex concentrated in degree~$0$.
 Consider the morphism $f_{C^\bu}\:\kappa_0(T_{n-1}(A^\bu))\rarrow
\kappa_1(C^0)$ in $\D\F$ corresponding to the morphism
$T_{n-1}(A^\bu)\rarrow C^0$ in~$\D$.
 Since $\Hom_{\D\F}(\kappa_1(C^0),\kappa_0(T_{n-1}(A^\bu))) = 0
= \Hom_{\D\F}(\kappa_0(T_{n-1}(A^\bu)),\kappa_1(C^0)[-1])$, one
can see that the cone of the morphism~$f_{C^\bu}$ is functorial
in $C^\bu$, i.~e., defines a functor $\Com^{[-n,0]}(\E)\rarrow\D\F$.
 Composing this functor with the functor~$w$, we obtain
the desired functor~$T_n$.

 It is clear that the functors $T_n$ agree with each other and are
compatible with the cohomological shift, thus they extend to
a functor $T\:\Com^b(\E)\rarrow\D$.
 The functor $T$ is additive, since it preserves finite direct sums.
 A morphism $A^\bu\rarrow B^\bu$ in $\Com^b(E)$ is homotopic to zero
if and only if it factorizes through the complex $\Cone(\id_{A^\bu})$
(or $\Cone(\id_{B^\bu})[-1]$), and this complex is isomorphic to
the direct sum of the complexes $\Cone(\id_{A^i})$.
 The latter complexes are clearly annihilated by $T$, hence
the functor $T$ factorizes through the functor $\Com^b(\E)\rarrow
\K^b(\E)$.
 The thick subcategory $\Ac^b(\E)\sub\K^b(\E)$ is generated by
complexes of the form $0\rarrow X'\rarrow X\rarrow X''\rarrow 0$,
where $X'\rarrow X\rarrow X''$ is an admissible triple in~$\E$.
 One easily checks that such complexes are also annihilated by~$T$,
so $T$ factorizes through $\D^b(\E)$.

 We have constructed the functor~$\Theta$; it remains to check that
it sends distinguished triangles to distinguished triangles.
 Any distinguished triangle in $\D^b(\E)$ comes from a term-wise
split short exact sequence of bounded complexes $C^\bu\rarrow D^\bu
\rarrow E^\bu$ over~$\E$.
 One can assume that all the three complexes belong to
$\Com^{[-n,0]}(\E)$.
 Present the complex $C^\bu$ as the cone of a closed morphism
$A^\bu\rarrow C^0$ and the complex $D^\bu$ as the cone of a closed
morphism $B^\bu\rarrow D^0$, where $A^\bu$, $B^\bu\in
\Com^{[-n+1,0]}(\E)$.
 To check that the image of our distinguished triangle in $\D^b(\E)$
is a distinguished triangle in $\D$, it suffices to apply
the $3{\times}3$\+lemma \cite[Proposition~1.1.11]{BBD} to
the commutative square formed by the morphisms
$\kappa_0(T_{n-1}(A^\bu))\rarrow\kappa_1(C^0)\rarrow\kappa_1(D^0)$
and $\kappa_0(T_{n-1}(A^\bu))\rarrow\kappa_0(T_{n-1}(B^\bu))
\rarrow\kappa_1(D^0)$ in $\D\F$, together with the assumption of
induction on~$n$.
\end{proof}

 Notice that one could also proceed in the dual way, cutting
the leftmost term out of a complex $C^\bu$, rather than the rightmost
term, as we have done.
 This would provide another construction of a functor~$\Theta$.
 To prove that the two constructions produce isomorphic functors,
one would probably need to assume the existence of a $[0,2]$\+filtered
triangulated category $(\D,\D\widetilde{\F})$ containing $(\D,\D\F)$.

\end{document}